\documentclass[12pt]{amsart}

\usepackage{amsmath}
\usepackage{amssymb}
\usepackage{amsfonts}
\usepackage{amsthm}
\usepackage{bbm}
\usepackage{cite}
\usepackage{latexsym}
\usepackage[pdftex]{graphicx}
\usepackage{color}
\usepackage{comment}
\usepackage{url}
\usepackage{hyperref}
\usepackage{todonotes}
\usepackage{mathtools}
\usepackage{empheq}
\usepackage{multicol}
\usepackage{bm}
\usepackage[all]{xy}
\xyoption{all}
\usepackage{mdframed} 

\newcommand{\R}{\mathbb{R}}

\newcommand{\Z}{\mathbb{Z}}

\newcommand{\D}{\mathbb{D}}
\newcommand{\F}{\mathbb{F}}

\newcommand{\PP}{\mathbb{P}}
\newcommand{\scr}{\mathcal}

\newcommand{\id}{\mathbbm{1}}

\newcommand{\wt}{\widetilde}

\newcommand*\quot[2]{{^{\textstyle #1}\big/_{\textstyle #2}}}

\theoremstyle{plain}
\newtheorem{thm}{Theorem}[section]
\newenvironment{fthm}
{\begin{mdframed}\vspace{0.5em}\begin{thm}}
		{\end{thm}\vspace{0.5em}\end{mdframed}}

\newtheorem{lem}[thm]{Lemma}
\newtheorem{prop}[thm]{Proposition}

\theoremstyle{definition}
\newtheorem{defn}[thm]{Definition}

\newtheorem{cor}[thm]{Corollary}

\newtheorem*{claim*}{Claim}
\theoremstyle{remark}
\newtheorem{rmk}[thm]{Remark}

\title[Dg-algebras, trivalent plane graphs, representations]{Differential graded algebras for trivalent plane graphs and their representations}

\author{Kevin Sackel}
\address{Department of Mathematics and Statistics, University of Massachusetts, Amherst, MA 01003, USA}
\email{ksackel@umass.edu}

\begin{document}

\begin{abstract} To any trivalent plane graph embedded in the sphere, Casals and Murphy associate a differential graded algebra (dg-algebra), in which the underlying graded algebra is free associative over a commutative ring. Our first result is the construction of a generalization of the Casals--Murphy dg-algebra to non-commutative coefficients, for which we prove various functoriality properties not previously verified in the commutative setting. Our second result is to prove that rank $r$ representations of this dg-algebra, over a field $\F$, correspond to colorings of the faces of the graph by elements of the Grassmannian $\operatorname{Gr}(r,2r;\F)$ so that bordering faces are transverse, up to the natural action of $\operatorname{PGL}_{2r}(\F)$. Underlying the combinatorics, the dg-algebra is a computation of the fully non-commutative Legendrian contact dg-algebra for Legendrian satellites of Legendrian 2-weaves, though we do not prove as such in this paper. The graph coloring problem verifies that for Legendrian 2-weaves, rank $r$ representations of the Legendrian contact dg-algebra correspond to constructible sheaves of microlocal rank $r$. This is the first explicit such computation of the bijection between the moduli spaces of representations and sheaves for an infinite family of Legendrian surfaces.
\end{abstract}

\maketitle


\section{Introduction}

In this paper, a \textbf{trivalent plane graph} $G$ will always mean a nonempty connected trivalent planar graph together with an embedding into $S^2$, up to isotopy of embeddings. The adjective \emph{trivalent} indicates that each vertex has exactly three incident edges, whereas the adjective \emph{plane} indicates that such an isotopy class of embedding has been fixed.\footnote{We use the adjective \emph{plane} instead of \emph{planar} intentionally. A planar graph is abstractly embeddable in $S^2$. Meanwhile, the choice of (isotopy class of) embedding is important for our purposes and will affect all constructions which follow.} Once we have fixed an embedding, we obtain a decomposition of $S^2$ into faces by cutting along the edges of the graph. Throughout, we will denote by $V$, $E$, and $F$ the set of vertices, edges, and faces of the graph $G$.

Casals and Murphy \cite{CM_DGA} have constructed for each trivalent plane graph $G$, upon making a number of auxiliary choices, a certain algebraic invariant, in the form of a differential graded algebra (dg-algebra). Our first goal in this paper is to build a souped-up version of their dg-algebra. Whereas their underlying algebra is free associative over the group ring $\Z[H_1(\Lambda_G)]$, where $\Lambda_G$ is topologically a branched double cover of $S^2$ branched over $V$, our underlying algebra is instead free associative over the non-commutative group ring $\Z[\pi_1(\Lambda_G)]$, where the coefficients no longer commute with the generators, and where we must keep track of a choice of base point. Upon abelianization of the coefficient ring and commuting coefficients with generators, our dg-algebra recovers the Casals--Murphy version. It was shown in the appendix to the paper of Casals and Murphy, written by the author of the present article, that their dg-algebra encodes certain graph coloring data. As a second goal, and indeed the initial impetus for this work, we extend this graph-coloring result to the fully non-commutative setting. 

Although this article is mostly presented in the guise of algebraic combinatorics with a few details from algebraic topology, on a deeper layer lies contact geometry, which provides context for the present article, as we now describe.

Associated to a trivalent plane graph is a Legendrian surface, the surface $\Lambda_G$ of the previous paragraph, as was defined by Treumann and Zaslow \cite{TZ}. Associated to any Legendrian (in a nice enough contact space) is a Legendrian contact dg-algebra, as originally discussed by Eliashberg \cite{Eliashberg_invariants} and simultaneously computed for Legendrian knots by Chekanov \cite{Chekanov}, and later formally shown to be well-defined by Ekholm, Etnyre, and Sullivan \cite{EES,EES_PxR}. The dg-algebra constructed by Casals and Murphy, though presented purely combinatorially, is a computation of a certain version of the Legendrian contact dg-algebra attached to $\Lambda_G$. The computation utilizes gradient flow trees, a technique pioneered by Ekholm \cite{Ekholm}. Our fully non-commutative dg-algebra simply keeps track of a little more geometric detail: the homotopy classes (as opposed to the homology classes) of the curves which comprise the boundary of the involved gradient flow trees.

Our graph coloring result also admits a contact-geometric interpretation. With underpinnings arising in the work of Nadler and Zaslow \cite{NZ} and Nadler \cite{Nad} providing a quasi-equivalence between the Fukaya category of a cotangent bundle with a category of constructible sheaves on the base, it was conjectured that representations of the Legendrian contact dg-algebra correspond to certain subclasses of constructible sheaves (see the work of Chantraine, Dimitroglou Rizell, Ghiggini, and Golovko \cite{CDGG} for the higher rank version). Since this article originally appeared on the arXiv, `representations are sheaves' was proved in full generality by Asplund and Ekholm \cite{AE}. Indeed, they proved a conjecture of Ekholm and Lekili \cite{EL} that the Legendrian contact dg-algebra with coefficients given by chains on the based loop space is quasi-equivalent to a certain partially wrapped Floer cohomology $A_{\infty}$-algebra, modules over which were proved by Ganatra, Pardon, and Shende \cite{GPS_microlocal} to correspond to constructible sheaves.
	
This article is more hands-on, continuing a history of explicit computations of this correspondence that are interesting in their own right and difficult in general to produce from the abstract theory alone. For Legendrian knots in $\mathbb{R}^3$, and rank $1$ representations, called augmentations, the correspondence was verified as an $A_{\infty}$-categorical equivalence by Ng, Rutherford, Shende, Sivek, and Zaslow \cite{NRSSZ}. Higher rank representations on Legendrian knots have also been understood for torus knots by Chantraine, Ng, and Sivek \cite{CNS}. Increasing to contact dimension $5$ (and Legendrians of dimension $2$), there have been explicit verifications for the rank $1$ version of the correspondence for knot and link conormals by Gao \cite{Gao_knot,Gao_link}, and for cellular Legendrian surfaces by Rutherford and Sullivan \cite{RS}. The graph coloring result in this paper provides a bijection between representations and sheaves for all Treumann--Zaslow Legendrian surfaces $\Lambda_G$, marking the first explicit nontrivial computation of the higher rank correspondence for an infinite collection of Legendrian surfaces.

We will keep the contact geometry mostly in the background, with two notable exceptions. First, throughout the introduction, we will expand, when appropriate, upon the short description above, primarily for the purposes of providing motivation for the results and constructions we discuss in this article. This includes a contact-geometric description of much of the combinatorial package in Section \ref{ssec:intro_contact}, so that contact geometers should feel satisfied that our dg-algebra is indeed the Legendrian contact dg-algebra, even if it is not fully proved. Second, we fundamentally use the fact that $\Lambda_G$ is a double cover of $S^2$ branched over the vertices $V$ of $G$ to find a family of combinatorial models for the coefficient ring $\Z[\pi_1(\Lambda_G)]$ with respect to different base points, and further, to describe how to transition between the various members of this combinatorial family. Our computation of the Legendrian contact dg-algebra, therefore, is not just a single dg-algebra, but a collection of dg-algebras, together with intertwining dg-isomorphisms which compose in a functorial manner depending upon a number of combinatorial choices. In practice, the reader interested in the algebro-combinatorial content of the paper should be able to read what follows with only this minimal topological input (explicitly described in the statement and proof of Theorem \ref{thm:geometric_coefficient_ring} of Section \ref{ssec:nc-coef_geom}).

\subsection{Generalizing the Casals--Murphy dg-algebra}

Associated to any nonempty trivalent plane graph $G$ is a constant $g \geq 0$ which we call the \textbf{genus}, such that $G$ has $|V| = 2g+2$ vertices, $|E| = 3g+3$ edges, and $|F| = g+3$ faces. It is called the genus because the double cover of the sphere branched over the vertices is a surface of genus $g$. This is more than just a convenience in our situation: there exists a Legendrian surface $\Lambda_{G}$ in the unit cotangent sphere bundle $S^*\R^3$ with its standard contact structure, as constructed by Treumann and Zaslow \cite{TZ}, which is a double cover of the sphere branched over the vertices. In the more modern terminology of $N$-graphs and Legendrian $N$-weaves initially studied by Casals and Zaslow \cite{CZ} and greatly generalizing the Treumann--Zaslow examples, $G$ is a $2$-graph on the sphere, and $\Lambda_G$ is the associated Legendrian $2$-weave. For the rest of the paper, we assume $G$ is connected.

From our graph, one may single out a face at infinity $f_{\infty} \in F$ so that we may draw $G$ as embedded in $\R^2$ via stereographic projection away from a point in the chosen face at infinity. We write $G_{f_{\infty}}$ for this trivalent plane graph embedded in $\R^2$, up to isotopy of embeddings (in $\R^2$). In the contact geometric picture, a choice of face at infinity yields a relatively explicit model for the Legendrian satellite of $\Lambda_G$, which lives in a standard contact Darboux ball $\R^5_{\mathrm{std}}$, as described e.g. in Casals--Zaslow for the case of $N$-weaves which includes our case when $N=2$ \cite{CZ}. By the Legendrian satellite operation, we mean the following. We have that a Legendrian admits a contact germ depending only upon the smooth topology of the Legendrian itself. Since $\Lambda_G$ is constructed in an arbitrarily small neighborhood of a Legendrian $S^2$, one may satellite it along the standard Legendrian unknot $S^2 \subset \R^5_{\mathrm{std}}$. We note that $H_1(\Lambda_G)$ and $\pi_1(\Lambda_G)$ still have meaning when we use the satellite, since the original surface $\Lambda_G$ is diffeomorphic to its satellite, with diffeomorphism determined up to isotopy.

The algebraic invariant of Casals and Murphy \cite{CM_DGA}, from this perspective, is a computation of the Legendrian contact dg-algebra of this Legendrian satellite using Ekholm's gradient flow trees \cite{Ekholm}. Though a proof that the combinatorics matches the geometry does not currently appear in the literature, we will provide some details of the correspondence in Section \ref{ssec:intro_contact}, assigning contact-geometric meaning to the combinatorial package below, providing a sketch of why the object we construct is indeed the Legendrian contact dg-algebra.

Suppose $G_{f_{\infty}}$ is trivalent plane graph with a specified face at infinity. Let $F_{\mathrm{fin}} = F - \{f_{\infty}\}$ be the set of all finite faces, where we will always assume the choice of $f_{\infty}$ has been made clear from the context. Over the coefficient ring $R = \Z[H_1(\Lambda_{G})]$, we form the free associative algebra
$$\scr{A}_{G,f_{\infty}} = R \langle F_{\mathrm{fin}},x,y,z \rangle \cong R\langle f_1,f_2,\ldots,f_{g+2},x,y,z \rangle,$$
where in the last isomorphism, $f_1,\ldots,f_{g+2}$ correspond to some labeling of the $g+2$ finite faces. We turn it into a graded algebra by specifying degrees: elements of $R$ have degree zero, elements of $F_{\mathrm{fin}}$ have degree $1$, and $|x|=|y|=|z|=2$. It is this graded algebra that Casals and Murphy upgrade into a dg-algebra.\\

\begin{defn} \label{defn:dga}
	A \textbf{differential graded algebra (dg-algebra)} is a pair $(A,\partial)$, where $A = \oplus_{i \in \Z} A_i$ is a graded algebra, and $\partial \colon A \rightarrow A$ is a linear map of degree $-1$ (so $\partial(A_i) \subset A_{i-1}$) satisfying
	\begin{itemize}
		\item \textbf{the Leibniz rule:} $\partial(xy) = (\partial x)y + (-1)^{|k|}x(\partial y)$ for $x \in A_k$ and $y \in A$
		\item \textbf{the differential condition:} $\partial^2 = 0$
	\end{itemize}
	A \textbf{dg-homomorphism} of dg-algebras is a morphism of graded algebras intertwining the differentials, i.e. $\phi \colon (A,\partial^A) \rightarrow (B,\partial^B)$ such that $\partial^B \circ \phi = \phi \circ \partial^A$. A \textbf{dg-isomorphism} is a dg-homomorphism which is an isomorphism of graded algebras.
\end{defn}

A dg-algebra has within it the structure of a chain complex, and dg-homomorphisms are a fortiori homomorphisms of chain complexes. We may therefore recall the following notions from the theory of chain complexes.

\begin{defn}	
	A dg-homomorphism $\phi \colon (A,\partial^A) \rightarrow (B,\partial^B)$ is called a \textbf{quasi-isomorphism} if the induced morphism on homology $\phi_* \colon H_*(A,\partial^A) \rightarrow H_*(B,\partial^B)$ is an isomorphism. Two dg-homomorphisms $\phi,\psi \colon (A,\partial^A) \rightarrow (B,\partial^B)$ are said to be \textbf{(linearly/chain) homotopic} if there exists a linear map $H \colon (A,\partial^A) \rightarrow (B,\partial^B)$ of degree $+1$, called a \textbf{(chain) homotopy}, such that $\phi - \psi = H\partial^A + \partial^BH$. In such a case, $\phi$ and $\psi$ induce the same map on homology.
\end{defn}

On a trivalent plane graph $G$, there are a number of auxiliary combinatorial choices which one can make, which we package into the notion of a garden $\Gamma$; this is Definition \ref{defn:garden} below. We refer to Figure \ref{fig:Garden} for an intuitive picture: we choose an orientation of the graph, a point called a `center' in each face, arcs between all adjacent centers and vertices called `threads', and a collection of non-intersecting coherently oriented paths which each pass through a single center, one for each face, called `tines', and which start and end at some specified point at infinity, called a `seed'. We will come back to the contact-geometric meaning of the garden in Section \ref{ssec:intro_contact}. Casals and Murphy use a slightly more restrictive definition, which we call a finite-type garden. Associated to each garden (regardless of whether it is finite-type or not) is a specified face $f_{\Gamma} \in F$ (the one containing the seed), which in the finite-type case we take to be the face at infinity. Casals and Murphy prove that associated to any finite-type garden, there exists a differential which makes the resulting dg-algebra with respect to the corresponding face at infinity an invariant.

\begin{thm}[Casals--Murphy \cite{CM_DGA}] \label{thm:CM}
For any choice of finite-type garden $\Gamma$ on $G$, there exists a differential $\partial^{\scr{A}}_{G,\Gamma}$ yielding a dg-algebra $(\scr{A}_{G,f_{\Gamma}},\partial^{\scr{A}}_{G,\Gamma})$, the dg-isomorphism type of which is independent of any two gardens with the same associated face at infinity.\\
\end{thm}

There are many reasonable names for this dg-algebra:
\begin{itemize}
	\item the \textbf{Casals--Murphy dg-algebra}, in reference to the mathematicians who computed it
	\item the invariant dg-algebra of binary sequences, as it was called in the paper of Casals and Murphy \cite{CM_DGA}
	\item the Legendrian contact dg-algebra (of the Legendrian satellite), in recognition of the contact-geometric picture
	\item the Chekanov-Eliashberg dg-algebra, which is synonymous with the Legendrian contact dg-algebra
\end{itemize}
In this paper, we will use the first of this list, because the second is rather cumbersome, and the third and fourth invoke the underlying contact geometry which we will largely ignore.

There are a few versions of the Legendrian contact dg-algebra one could reasonably define, depending upon how much structure one wishes to encode. The Casals--Murphy dg-algebra corresponds to one such choice; instead, we use a fully non-commutative version, which on the level of the graded algebra is realized by removing commutativity using the following two-step procedure:
\begin{itemize}
	\item First, we do not need that the elements of the coefficient ring commute with the other generators. We may instead consider $\Z$-linear combinations of words of the form $n_1 w_1 n_2 w_2 \cdots n_k w_k$ where $n_i \in R$ and $w_i$ is a word in the other variables.
	\item We may change our coefficient ring from $R = \Z[H_1(\Lambda_{G})]$ to $R^{nc} = \Z[\pi_1(\Lambda_{G})]$ instead, with respect to some base point. The marker $nc$ stands for \emph{non-commutative}.
\end{itemize}
Because of the definition of the differential of a Legendrian contact dg-algebra, in order to apply the second step, we must have also applied the first, so our use of non-commutative coefficients also necessitates removing commutativity between the coefficients and the generators. The reader interested in a contact-geometric explanation of this point may consult e.g. Appendix A of an article of Chantraine, Dimitroglou Rizell, Ghiggini, and Golovko \cite{CDGG}.

Let us be precise about our base points. If we fix a vertex $v \in V$, then since $\Lambda_G$ is branched at $v$, there is a unique point $\wt{v} \in \Lambda_{G}$ lying over $v$, which we use as a base point for our computations. Let $\scr{B}_{G,f_{\infty},v}$ be the resulting graded algebra; we may write it as a free product of algebras
$$\scr{B}_{G,f_{\infty},v} = \Z[\pi_1(\Lambda_{G},\wt{v})] * \Z\langle F_{\mathrm{fin}},x,y,z \rangle,$$
where again the elements of $\pi_1(\Lambda_{G},\wt{v})$ are in degree zero, the elements of $F_{\mathrm{fin}}$ are in degree $1$, and $|x|=|y|=|z|=2$.

In order to obtain nice combinatorial models for $\pi_1(\Lambda_G,\wt{v})$ for $v \in V$, it is convenient to specify a set of generators and relations. This can be achieved by picking a tree $T \subset G$ spanning all of the vertices in $V$ except $v$ (recalling our assumption that $G$ is connected so that such $T$ exists). In fact, a choice of such a tree is already performed by Casals and Murphy for essentially the same reason: the Poincar\'e duals to the cycles in $\Lambda_G$ lying over $T$ yield a convenient set of generators of $H_1(\Lambda_G)$. We denote the vertex missed by $T$ as $v_T \in V$, and for the lift $\wt{v}_T$ we use the slightly more convenient notation $*_T \in \Lambda_G$. Accordingly, we build a collection of groups $\Pi_T$ with explicit generators and relations which come with canonical isomorphisms $\Pi_T \cong \pi_1(\Lambda_G, *_T)$. As these models are purely combinatorial, we may therefore instead use the graded algebras
$$\scr{B}_{G,f_{\infty},T} = \Z[\Pi_T] * \Z\langle F_{\mathrm{fin}},x,y,z \rangle$$
with the understanding that if $T$ and $T'$ miss the same vertex $v_T = v_{T'}$, then there are canonical isomorphisms
$$\scr{B}_{G,f_{\infty},T} \cong \scr{B}_{G,f_{\infty},v_T} = \scr{B}_{G,f_{\infty},v_{T'}} \cong \scr{B}_{G,f_{\infty},T'}.$$

There is good reason for having these graded algebras depend upon $T$ (and not just $*_T$) aside from just that they are combinatorially convenient. Unlike in the version with homology coefficients, in order to obtain the Legendrian contact differential, each gradient flow tree counted in the differential must be connected to the base point $*$ in order to obtain an element in $\pi_1(\Lambda_G,*)$. In contact geometric parlance, such choices are referred to as capping paths. Conveniently, our trees $T$ spanning the vertices $V - \{v_T\}$ yield natural choices of capping paths, and hence fix the differentials. However, even if $v_T = v_{T'}$, the capping paths will not be the same, and hence a dg-isomorphism between the models will not in general be realized by the natural canonical isomorphisms $\scr{B}_{G,f_{\infty},T} \cong \scr{B}_{G,f_{\infty},T'}$ of the previous paragraph, since this will not in general intertwine the differentials. Nonetheless, we have the following, the first main theorem of this paper.\\

	\begin{fthm} \label{thm:nc_CM}
		Let $G$ be a trivalent plane graph. For any choice of finite-type garden $\Gamma$ on $G$ and tree $T \subset G$ spanning all but one vertex, there exists a differential $\partial^{\scr{B}}_{G,\Gamma,T}$ on $\scr{B}_{G,f_{\Gamma},T}$, yielding a dg-algebra\\
		$$(\scr{B}_{G,f_{\Gamma},T},\partial^{\scr{B}}_{G,\Gamma,T}),$$\\
		which, under the map given by abelianizing coefficients and commuting ring elements with generators, induces a dg-homomorphism\\
		$$\mathrm{Ab} \colon (\scr{B}_{G,f_{\Gamma},T},\partial^{\scr{B}}_{G,\Gamma,T}) \rightarrow (\scr{A}_{G,f_{\Gamma}},\partial^{\scr{A}}_{G,\Gamma}).$$
		If $\Gamma$ and $\Gamma'$ are two finite-type gardens and $T$ and $T'$ are trees in $G$ spanning all but one vertex, then the resulting dg-algebras for the pairs $(\Gamma,T)$ and $(\Gamma',T')$ are dg-isomorphic.
	\end{fthm}

Theorem \ref{thm:CM} follows as an immediate corollary from Theorem \ref{thm:nc_CM}.

\subsection{Functoriality for the non-commutative Casals--Murphy dg-algebra}

We may further refine the dg-isomorphisms of Theorem \ref{thm:nc_CM}, as in Theorem \ref{thm:nc_CM_functoriality} below. The refined version provides an understanding of two further questions:
\begin{enumerate}
	\item What properties do the dg-isomorphisms between the dg-algebras associated to $(\Gamma,T)$ and $(\Gamma',T')$ satisfy?
	\item What combinatorial data is required to specify a \emph{canonical} dg-isomorphism for different pairs $(\Gamma,T)$ and $(\Gamma',T')$?
\end{enumerate}

Let us begin with the first of these questions. In the setting of commutative coefficients, it is well known that the Legendrian contact dg-algebra is invariant up to a notion known as stable tame (dg-)isomorphism, c.f. \cite{EES}. In our setting, one can easily deduce that this is the case for the commutative coefficient setting by carefully studying the arguments of Casals and Murphy \cite{CM_DGA}. It was noted previously that there should be a notion of stable tame (dg-)isomorphism in the noncommutative setting \cite{CDGG}, though formal definitions did not appear. We make those definitions under the assumption of non-negative grading, which will be the case throughout this paper (all of our generators will have degree $1$ or $2$).

\begin{defn}
	Let $R$ be any unital ring, $S$ any set, and $\deg \colon S \rightarrow \Z_{\geq 0}$ a collection of (non-negative) degrees for elements of $S$. Let
	$$A_S = R * \Z \langle S \rangle$$
	be the fully non-commutative associative graded algebra freely generated by $R$ and elements of $S$, where the grading of elements of $R$ is $0$ and the grading of elements of $S$ matches the degree. The algebra $A_S$ together with its explicit set of generators $S$ is known as a \textbf{semi-free graded algebra}. If $A_S$ comes with a differential $\partial$ (of degree $-1$), then we call $(A_S,\partial)$ a \textbf{semi-free dg-algebra}. Note that $\partial r = 0$ for all $r \in R$ for degree reasons.
	
	An \textbf{elementary automorphism} of a semi-free graded algebra $A_S$ is an automorphism $\Phi \colon A_S \rightarrow A_S$ which is the identity on elements of $R$ and $S \setminus \{s\}$ for some $s \in S$, and satisfies
	$$\Phi(s) = s + \alpha$$
	for some $\alpha \in A_{S \setminus \{s\}}$. A \textbf{tame automorphism} of $A_S$ is a composition of elementary automorphisms.
	
	A \textbf{regeneration} of $A_S$ is a graded algebra automorphism of the form $\Phi(s) = u_s\cdot \sigma(s) \cdot v_s$ on all generators $s \in S$, where $u_s,v_s \in R$ are invertible elements and $\sigma$ is a permutation on $S$.
	
	If $(A_S,\partial)$ is a semi-free dg-algebra, then any graded algebra automorphism $\Phi \colon A_S \rightarrow A_S$ induces a unique differential $\partial'$ on $A_S$ so that $\Phi \colon (A_S,\partial) \rightarrow (A_S,\partial')$ is a dg-isomorphism. If $\partial$ is given, we will often conflate $\Phi$ as a graded algebra automorphism with $\Phi$ as a dg-isomorphism. (For example, we may call a tame automorphism a tame dg-isomorphism.)
		
	The \textbf{stabilization} of a semi-free dg-algebra $(A_S,\partial)$ is the semi-free dg-algebra $(A_{S^+},\partial^+)$, where $S^+ = S \sqcup \{x,y\}$ with $\deg(y) = \deg(x)+1$, $\partial^+$ matching $\partial$ on $S$, and
	$$\partial^+y = x, \qquad \partial^+x = 0.$$
	The natural inclusion $A_{S} \hookrightarrow A_{S^+}$ is a dg-homomorphism, called the \textbf{stabilization (dg-homomorphism)}. Conversely, $A_S$ is called the \textbf{destabilization} of $A_{S^+}$ (with respect to $x,y \in S$), and the dg-homomorphism $A_{S^+} \rightarrow A_{S}$ which kills $x$ and $y$ but preserves the other generators and elements of $R$ is called the \textbf{destabilization (dg-homomorphism)}.
	
	A \textbf{stable regenerative tame dg-homomorphism} is a composition of stabilizations, destabilizations, regenerations, and elementary dg-automorphisms. Any of the adjectives stable, regenerative, or tame may be removed one at a time if there are no stabilizations or destabilizations, no regenerations, or no tame dg-homomorphisms. We may similarly replace the word dg-homomorphism with dg-automorphism, where appropriate.
\end{defn}

It is a fact that every stable regenerative tame dg-homomorphism is a quasi-isomorphism, as is proved in Proposition \ref{prop:dg-homomorphism_is_quasi-isomorphism}.

Because we allow for stabilizations and destabilizations, we work in a slightly more symmetric set-up. Whereas the finite-type gardens considered by Casals and Murphy avoid considering the face at infinity, we may stabilize, including the face at infinity as an extra degree 1 generator at the cost of also introducing a degree $2$ generator $w$ which kills it (up to an elementary dg-homomorphism). We will use the superscript $+$ for this enlargement. For any garden, not just those of finite-type, we form the graded algebra
$$\scr{B}^+_{G,T} := \Z[\Pi_T] * \Z \langle F,x,y,z,w \rangle.$$
This allows us to treat all of the faces together in a slightly cleaner manner. As per Proposition \ref{prop:dg-homomorphism_is_quasi-isomorphism}, when we put the proper differential on this graded algebra, we obtain a model quasi-isomorphic to the version without the $+$ enlargement. Contact geometrically, these two extra generators can be produced by a specific Legendrian isotopy of the satellite of $\Lambda_G$, and so it is no surprise that we obtain a quasi-isomorphic model.

The second question is partly motivated by the fact that the constructed dg-isomorphisms seem to require a number of choices, even in the commutative coefficient setting of Casals and Murphy \cite{CM_DGA} (though they leave this point implicit in their constructions). We work to explain the extra underlying algebraic structure.

One piece of the story is relatively simple. A garden comes with a choice of orientation of the edges of $G$. If $E$ are the edges of $G$, then there is always a natural $(\Z_2)^E$ action on gardens given by flipping edge orientations. For two gardens which are the same except for the orientation of edges on $G$, the underlying graded algebras are the same, and we will see that the dg-isomorphisms are induced by ring automorphisms of $\Z[\Pi_T]$. Other changes of the garden are more complicated.

Aside from changing the garden, we may also change the tree. We have already pointed out that even if $*_T = *_{T'}$, then for fixed $\Gamma$, the canonical identification $\scr{B}_{G,f_{\Gamma},T} \cong \scr{B}_{G,f_{\Gamma},T'}$ given by the association $\Pi_T \cong \pi_1(\Lambda_G,*_T) = \pi_1(\Lambda_G,*_{T'}) \cong \Pi_{T'}$ will not induce a dg-isomorphism, because the induced capping paths are not the same. More generally, if $*_T \neq *_{T'}$, then to even compare $\pi_1(\Lambda_G,*_T)$ and $\pi_1(\Lambda_G,*_{T'})$ in the first place, we must choose a non-canonical homotopy class of path $\gamma$ from $*_T$ to $*_{T'}$, yielding the isomorphism $\scr{C}(\gamma) \colon \pi_1(\Lambda_G,*_T) \rightarrow \pi_1(\Lambda_G,*_{T'})$ defined by
$$\scr{C}(\gamma)(\eta) := \gamma^{-1} * \eta * \gamma.$$
We recall here that the composition $*$ means the following: if $\eta_1$ and $\eta_2$ are two homotopy classes of paths from $p$ to $q$ and $q$ to $r$ respectively, then $\eta_1 * \eta_2$ is the homotopy class of paths from $p$ to $r$ with representative given by first following a representative of $\eta_1$ and then following a representative of $\eta_2$. Finally, let $\scr{C}_T^{T'}(\gamma) \colon \Pi_T \rightarrow \Pi_{T'}$ be the induced isomorphism fitting into the commutative diagram
$$\xymatrix{\Pi_T \ar[r]^{\scr{C}_T^{T'}(\gamma)} \ar[d]_[left]{\sim} & \Pi_{T'} \ar[d]^[right]{\sim} \\
	\pi_1(\Lambda_G,*_T) \ar[r]_{\scr{C}(\gamma)} & \pi_1(\Lambda_G,*_{T'})},$$
which in turn induces graded algebra isomorphisms $\scr{C}_{T}^{T'}(\gamma) \colon \scr{B}_{G,T}^+ \rightarrow \scr{B}_{G,T'}^+.$\\

\begin{fthm} \label{thm:nc_CM_functoriality}
	Let $G$ be a nonempty connected trivalent plane graph of genus $g$. For any garden $\Gamma$ and tree $T$, there exists a canonical differential $\partial^{\scr{B}^+}_{G,\Gamma,T}$ on $\scr{B}^+_{G,T}$ yielding a dg-algebra. If two gardens $\Gamma$ and $\Gamma'$ on $G$ are homotopic and $T=T'$, then their associated dg-algebras are canonically identified via the identity on the underlying graded algebras.
	
	The group
	$$\scr{H}(G) := (\Z_2)^E \times F_{g+2} \times (\Z_2 * \Z_3)$$
	naturally acts transitively on the set of homotopy classes of gardens on $G$, where $E$ is the set of edges of $G$, and where $F_{g+2}$ is the free group on $g+2$ generators. Suppose $\zeta \in \scr{H}(G)$ with
	$$\Gamma' = \zeta \cdot \Gamma$$
	and suppose $\gamma$ is a homotopy class of paths from $*_T$ to $*_{T'}$ in $\Lambda_G$. Then there exist canonical dg-isomorphisms
	$$\Phi^{\Gamma',T'}_{\Gamma,T}(\zeta,\gamma) \colon (\scr{B}^+_{G,T},\partial^{\scr{B}^+}_{G,\Gamma,T}) \rightarrow (\scr{B}^{+}_{G,T'},\partial^{\scr{B}^{+}}_{G,\Gamma',T'})$$
	which satisfy the following functoriality properties:
	
	\begin{itemize}
		\item \textbf{Composition property:} The dg-isomorphisms compose naturally, in the sense that
		$$\Phi_{\Gamma',T'}^{\Gamma'',T''}(\zeta,\gamma_2) \circ \Phi_{\Gamma,T}^{\Gamma',T'}(\theta,\gamma_1) = \Phi_{\Gamma,T}^{\Gamma'',T''}(\zeta\theta,\gamma_1*\gamma_2).$$
		\item \textbf{Orientation changes:} The action of the factor $(\Z_2)^E$ (with $\gamma = \id$) is induced by ring automorphisms of $\Z[\Pi_T]$.
		\item \textbf{Fixing the trees and capping paths; modifying the garden:} If $T = T'$ and $\gamma = \id$ is the constant path at $*_T$, then for any $\zeta \in F_{g+2} \times (\Z_2 * \Z_3)$, the map $\Phi_{\Gamma,T}^{\Gamma',T}(\zeta,\id)$ is a regenerative tame dg-isomorphism. Furthermore, if $\zeta$ stabilizes $\Gamma$, i.e. $\zeta \cdot \Gamma = \Gamma$, then $\Phi_{\Gamma,T}^{\Gamma,T}(\zeta,\id)$ is homotopic to the identity, and hence acts by the identity in homology.
		\item \textbf{Fixing the garden; modifying the trees and capping paths:}	For any fixed garden $\Gamma$, any two trees $T$ and $T'$, and any homotopy class of path $\gamma$ from $*_T$ to $*_{T'}$, we have,
		$$\Phi_{\Gamma,T}^{\Gamma,T'}(1,\gamma) = \scr{R}_{T}^{T'}(\gamma) \circ \scr{C}_T^{T'}(\gamma)$$
		where $\scr{C}_T^{T'}(\gamma)$ acts on the coefficient ring (as discussed just before this theorem statement) and $\scr{R}_T^{T'}(\gamma)$ is a regeneration independent of the garden acting on each generator by conjugation by an element of $\Pi_{T'}$ (not necessarily the same for each generator). If $T = T'$, then $R_{T}^{T}(\id) = \id$.
		\item \textbf{Finite-type gardens:} If $\Gamma$ is a finite-type garden, then the dg-algebra $(\scr{B}_{G,f_{\Gamma},T},\partial^{\scr{B}}_{G,\Gamma,T})$ of Theorem \ref{thm:nc_CM} is obtained from $(\scr{B}^+_{G,T},\partial^+_{G,\Gamma,T})$ via a canonical composition of an elementary automorphism of $\scr{B}^+_{G,T}$ followed by a destabilization.
	\end{itemize}
\end{fthm}

Notice that we obtain significantly more functorial data than was originally obtained in the commutative coefficient setting by Casals and Murphy \cite{CM_DGA}. By functoriality, we mean that we may build for each connected nonempty trivalent graph $G$ a category whose objects are pairs $(\Gamma,T)$ and whose morphisms are pairs $(\zeta,\gamma)$; Theorem \ref{thm:nc_CM_functoriality} yields a functor from this category to the category of dg-algebras and dg-isomorphisms.

Any specified faces $f,f' \in F$ may arise as $f_{\Gamma},f_{\Gamma'}$ for some finite-type gardens $\Gamma,\Gamma'$, and accordingly, for any $\zeta \in \scr{H}(G)$ with $\zeta \cdot \Gamma = \Gamma'$ and homotopy class of path $\gamma$ from $*_T$ to $*_{T'}$, we obtain a canonical quasi-isomorphism from $(\scr{B}_{G,f_{\Gamma},T},\partial_{G,\Gamma,T})$ to $(\scr{B}_{G,f_{\Gamma'},T'},\partial_{G,\Gamma',T'})$ given by the composition
$$\scr{B}_{G,f_{\Gamma},T} \xrightarrow{\mathrm{stab}} \scr{B}^+_{G,T} \xrightarrow{\Phi_{\Gamma,T}^{\Gamma',T'}(\zeta,\gamma)} \scr{B}^+_{G,T'} \xrightarrow{\mathrm{destab}} \scr{B}_{G,f_{\Gamma},T'},$$
where the stabilization and destabilization maps include an extra elementary automorphism of $\scr{B}_{G,T}^+$ and $\scr{B}_{G,T'}^+$, respectively. Theorem \ref{thm:nc_CM} follows from Theorem \ref{thm:nc_CM_functoriality} so long as one proves that the above maps are dg-isomorphisms. Partial abelianization then yields all functoriality properties satisfied by the original Casals--Murphy dg-algebras $(\scr{A}_{G,f_{\Gamma},T},\partial^{\scr{A}}_{G,\Gamma,T})$.

Our construction of the differentials occuring in Theorem \ref{thm:nc_CM_functoriality} borrows from an idea of Casals and Murphy \cite{CM_DGA}. Instead of working with the coefficient ring $\Z[\Pi_T]$, we may work with instead a coefficient ring $\Z[\Pi]$, where $\Pi$ is a group coming with projections $\pi_T \colon \Pi \rightarrow \Pi_T$ and inclusions $i_T \colon \Pi_T \hookrightarrow \Pi$ with $\pi_T \circ i_T = \id$. In other words, we separate the dependence on the tree $T$, forming the dg-algebra
$$\wt{\scr{B}}^+_G := \Z[\Pi]*\Z \langle F,x,y,z,w \rangle,$$
and accordingly finding differentials $\wt{\partial}^{\scr{B}^+}_{G,\Gamma}$ for each garden $\Gamma$. (One may of course work without the $+$ enlargement, at the cost of restricting to finite-type gardens.) The differentials over this enlarged ring then induce the differentials $\partial^{\scr{B}^+}_{G,\Gamma,T}$ of Theorem \ref{thm:nc_CM_functoriality}, and the functoriality properties follow from a careful consideration of the algebraic structure of $\Pi$ with its various projections $\pi_T$. The functoriality properties for the dg-algebras $(\wt{\scr{B}}^+_G,\wt{\partial}^{\scr{B}^+}_{G,\Gamma})$ are given by Theorem \ref{thm:nc_CM_enlarged}.

\subsection{Representations and colors}

If $R$ is a ring, which we consider as a graded $\Z$-algebra concentrated in degree zero, then we obtain a dg-algebra $(R,0)$.
\begin{defn}
	An \textbf{$R$-augmentation} of a dg-algebra $(A,\partial)$  is a dg-homomorphism $\epsilon \colon (A,\partial) \rightarrow (R,0)$. In the case that $R = \mathrm{Mat}_r(\F)$ for a field $\F$, we call the corresponding $R$-augmentation a \textbf{rank $r$ representation over $\F$}. A rank $1$ representation is sometimes just called an \textbf{augmentation}. We denote the set of representations of rank $r$ by
	$$\mathrm{Rep}_r(A,\partial;\F) := \{\mathrm{rank~}r\mathrm{~representations~of~}(A,\partial)\mathrm{~over~}\F\}.$$
\end{defn}

Notice that there is a conjugation action of $\mathrm{GL}_r(\F)$ on $\mathrm{Rep}_r(A,\partial;\F)$: if $\epsilon$ is a rank $r$ representation and $g \in \mathrm{GL}_r(\F)$, then
$$(g \cdot \epsilon)(x) := g \cdot \epsilon(x) \cdot g^{-1}.$$
We shall prefer to consider two rank $r$ representations to be equivalent if they are related by the conjugation action.
\begin{defn}
	The \textbf{moduli space of rank $r$ representations over $\F$} of a dg-algebra $(A,\partial)$ is just the set of representations modulo equivalence by this conjugation action, denoted by
	$$\scr{M}^{\mathrm{Rep}}_r(A,\partial;\F) := \quot{\mathrm{Rep}_r(A,\partial;\F)}{\mathrm{GL}_r(\F)}.$$
\end{defn}
The conjugation $\mathrm{GL}_r(\F)$-action descends to a $\mathrm{PGL}_r(\F)$-action, so we could alternatively have written the quotient with respect to $\mathrm{PGL}_r(\F)$ instead. In the case $r=1$, since $\mathrm{PGL}_1(\F)$ is trivial, we therefore have
$$\scr{M}^{\mathrm{Rep}}_1(A,\partial;\F) = \mathrm{Rep}_1(A,\partial;\F).$$
For higher $r$, there is in general no such equality except for highly degenerate situations.

There is a general correspondence between the moduli space of rank $r$ representations of a Legendrian contact dg-algebra and a certain moduli space of constructible sheaves of microlocal rank $r$ depending upon the Legendrian.\footnote{In our case, representations are homotopic if and only if they are equal for degree reasons; the more general conjecture requires taking representations modulo homotopies.} As such, since we have stated throughout this introduction that the dg-algebras we are constructing are combinatorial models for the Legendrian contact dg-algebra associated to $\Lambda_G$, in order to verify this correspondence, we need also to understand the corresponding moduli space of constructbile sheaves. See work of Shende, Treumann, Williams, and Zaslow \cite{STWZ} for a precise definition of the desired constructible sheaves of microlocal rank $r$ and Treumann and Zaslow \cite{TZ} for a description in the case $r=1$ for the case we are considering. From these references, one may easily deduce a clean combinatorial description of this moduli space of sheaves purely in terms of the trivalent plane graph $G$, which we now describe.

Suppose we are given a trivalent plane graph $G$, and let $F$ be the faces of $G$. Consider the Grassmannian $\mathrm{Gr}(r,2r;\F)$ consisting of subspaces of $\F^{2r}$ of dimension $r$.
\begin{defn}
	A \textbf{rank $r$ face coloring of $G$ over $\F$} is an assignment
	$$\Phi \colon F \rightarrow \mathrm{Gr}(r,2r;\F)$$
	such that if $f$ and $g$ share an edge, then $\Phi(f)$ and $\Phi(g)$ are transverse. We denote the set of such rank $r$ face colorings by
	$$\mathrm{Col}_r(G;\F):= \{\mathrm{rank~}r~\mathrm{face~colorings~of~}G\mathrm{~over~}\F\}.$$
\end{defn}

In the case that $r=1$, this recovers the usual notion of a face coloring of the graph $G$ (or equivalently a vertex coloring of the dual graph $G^*$) by elements of the projective line $\F\PP^1$. We note that in the case that $\F = \F_q$ is a finite field of order $q$, then the number of face colorings of rank $1$ is therefore
$$|\mathrm{Col}_r(G;\F_q)| = \chi_{G^*}(q+1),$$
where $\chi_{G^*}$ is the famous chromatic polynomial associated to a trivalent plane graph. Here, we use the fact that $|\F_q\PP^1| = q+1$.

Notice that $\mathrm{PGL}_{2r}(\F)$ acts on $\mathrm{Gr}(r,2r;\F)$, and if $\Phi \in \mathrm{PGL}_{2r}(\F)$ and $X,Y \in \mathrm{Gr}(r,2r;\F)$ are transverse, then $\Phi(X)$ and $\Phi(Y)$ are also transverse. Therefore, $\mathrm{Col}_r(G;\F)$ is an invariant set for the $\mathrm{PGL}_{2r}(\F)$-action.

\begin{defn}
	The \textbf{moduli space of rank $r$ colorings over $\F$} is the set given by $$\scr{M}^{\mathrm{Col}}_r(G;\F) := \quot{\mathrm{Col}_r(G;\F)}{\mathrm{PGL}_{2r}(\F)}.$$
\end{defn}

In the case that $r=1$, so long as the trivalent plane graph $G$ is non-empty, the action of $\mathrm{PGL}_{2}(\F)$ is free. In particular, in the case that $\F = \F_q$ has order $q$, we have
$$|\scr{M}_1^{\mathrm{Col}}(G;\F_q)| = \frac{\chi_{G^*}(q+1)}{q^3-q}.$$
This description of $\scr{M}_1^{\mathrm{Col}}(G;\F_q)$ appeared in the work of Treumann and Zaslow \cite{TZ}. In the appendix to the paper of Casals and Murphy \cite{CM_DGA}, written by the present author, the `augmentations are sheaves' result was explicitly verified.
\begin{thm}[{S \cite[Appendix A]{CM_DGA}} ] \label{thm:augs_sheaves_finite_type}
	For any field $\F$, there is a natural bijection $$\scr{M}^\mathrm{Rep}_1(\scr{A}_{G,f_{\Gamma},T}, \partial^{\scr{A}}_{G,\Gamma,T};\F) \cong \scr{M}_1^{\mathrm{Col}}(G;\F).$$
\end{thm}

It is worth noting that in the case that $\F = \F_3$ is the field of order $3$, then the Four Color Theorem, famous for its eventual solution by Appel and Haken \cite{AH1,AH2} in the 1970s using computers, as well as the current non-existence of a non-computer-based proof, is the assertion that $\scr{M}^{\mathrm{Col}}_1(G;\F_3)$ is non-empty whenever $G$ is bridgeless, meaning there is no face which borders itself. Therefore, we obtain the equivalent result that for bridgeless $G$, the Legendrian $\Lambda_G$ admits an $\F_3$-augmentation.

Because $\mathrm{GL}_1(\F) = \F^*$ is commutative, we notice that there is a canonical identification
$$\scr{M}^{\mathrm{Rep}}_1(\scr{A}_{G,f_{\Gamma},T}, \partial^{\scr{A}}_{G,\Gamma,T};\F) \cong \scr{M}^{\mathrm{Rep}}_1(\scr{B}_{G,f_{\Gamma},T}, \partial^{\scr{B}}_{G,\Gamma,T};\F).$$
In particular, the previous theorem is a special case of `representations are sheaves' for the non-commutative setting, which we verify, along with a number of functoriality properties.

With regards to functoriality, notice that if we have a dg-homomorphism $\Phi \colon (A,\partial) \rightarrow (A',\partial')$, then the pullback maps
$$\Phi^* \colon \mathrm{Rep}_r(A',\partial';\F) \rightarrow \mathrm{Rep}_r(A,\partial;\F)$$
intertwine the conjugation action. Hence $\Phi^*$ descends to the moduli spaces
$$\Phi^* \colon \scr{M}^{\mathrm{Rep}}_r(A',\partial';\F) \rightarrow \scr{M}^{\mathrm{Rep}}_r(A,\partial;\F).$$
In our situation, $\Phi_{\Gamma,T}^{\Gamma',T'}(\zeta,\gamma)$ are all dg-isomorphisms yielding pullback isomorphisms
$$(\Phi_{\Gamma,T}^{\Gamma',T'}(\zeta,\gamma))^* \colon \mathrm{Rep}_r(\scr{B}^+_{G,T'},\partial^{\scr{B}^+}_{G,\Gamma',T'};\F) \rightarrow \mathrm{Rep}_r(\scr{B}^+_{G,T},\partial^{\scr{B}^+}_{G,\Gamma,T};\F).$$
which intertwine the conjugation action of $\mathrm{GL}_r(\F)$, and hence also induce isomorphisms at the level of the moduli space of rank $r$ representations.\\
\begin{fthm} \label{thm:reps_are_sheaves}
	Let $G$ be a trivalent plane graph, $\F$ a field, and $r$ a positive integer. For each garden $\Gamma$ and subtree $T \subset G$ spanning all but one vertex, there is a canonical map
	$$\Psi_{G,\Gamma,T} \colon \mathrm{Rep}_r(\scr{B}^+_{G,T},\partial^{\scr{B}^+}_{G,\Gamma,T};\F) \rightarrow \scr{M}^{\mathrm{Col}}_r(G;\F)$$
	with the following properties:
	\begin{itemize}
		\item \textbf{Functoriality:} We have $$\Psi_{G,\Gamma,T} \circ (\Phi_{\Gamma,T}^{\Gamma',T'}(\zeta,\gamma))^* = \Psi_{G,\Gamma',T'}.$$
		\item \textbf{Representations are sheaves:} $\Psi_{G,\Gamma,T}$ is invariant under the conjugation action, and induces a bijection
		$$\scr{M}_r^{\mathrm{Rep}}(\scr{B}^+_{G,T},\partial^{\scr{B}^+}_{G,\Gamma,T};\F) \cong \scr{M}^{\mathrm{Col}}_r(G;\F).$$
	\end{itemize}
\end{fthm}

If one wishes to work with finite-type gardens, as was done in the original commutative setting, one simply notices that, so long as all generators are in positive degree, the pullbacks of stabilization and destabilization maps are inverses on the sets of representations; see Proposition \ref{prop:stabilize_reps} below. It follows immediately that 
$$\mathrm{Rep}_r(\scr{B}^+_{G,T},\partial^{\scr{B}^+}_{G,\Gamma,T};\F) \cong \mathrm{Rep}_r(\scr{B}_{G,f_{\Gamma},T}, \partial^{\scr{B}}_{G,\Gamma,T};\F)$$
since the dg-algebras are related by a composition of (elementary) dg-isomorphisms together with stabilization/destabilization, where all generators involved are of positive degree. Furthermore, since pullbacks of dg-homomorphisms intertwine the conjugation action, we furthermore have
$$\scr{M}^{\mathrm{Rep}}_r(\scr{B}^+_{G,T},\partial^{\scr{B}^+}_{G,\Gamma,T};\F) \cong \scr{M}^{\mathrm{Rep}}_r(\scr{B}_{G,f_{\Gamma},T}, \partial^{\scr{B}}_{G,\Gamma,T};\F).$$
In the case $r=1$, one recovers Theorem \ref{thm:augs_sheaves_finite_type} from combining this observation with Theorem \ref{thm:reps_are_sheaves}.

\subsection{Contact geometry} \label{ssec:intro_contact}

We now give a few details as to why one might expect that the dg-algebras $(\scr{B}^{+}_{G,T},\partial^{\scr{B}^+}_{G,\Gamma,T})$ are indeed models for a Legendrian contact dg-algebra. To begin, we will explain why $\scr{B}^+_{G,T}$ is the correct graded algebra, which requires identifying the generators as the Reeb chords of the Legendrian. We will then partially discuss the differential, including a description of the types of objects which are counted. Upon reading Section \ref{ssec:enlarged}, the contact-geometric reader may wish to come back to this section to check their understanding of how each binary sequence (Definition \ref{defn:binary_seq}) corresponds to a term in the differential, which essentially finishes the translation of the rest of the article back into the language of contact geometry. A full proof of this correspondence would require also a description of the signs, which we do not discuss at all, as well as an understanding of gradient flow trees near the trivalent vertices, which requires a perturbation of certain singularities in the front.

Recall that $\R^5_{\mathrm{std}}$ comes with a so-called front projection $\pi_{f} \colon \R^5 \rightarrow \R^3$, and that for a generic Legendrian surface $\Lambda \subset \R^5$, it may be recovered from its projection to $\pi_{f}(\Lambda) \subset \R^3$, which is almost everywhere immersed except along some mild singular strata. More specifically, with respect to the typical coordinates $x_1,x_2,z$ for $\R^3$, this projection is never tangent to the vertical vector $\partial_z$, and the remaining two coordinates are determined by the slopes of the tangent planes to $\pi_{f}(\Lambda)$, i.e. $y_j = \frac{\partial z}{\partial x_j}$. The underlying graded algebra of the Legendrian contact dg-algebra is then of the form
$$\Z[\pi_1(\Lambda)] * \Z\langle \scr{R}\rangle,$$
where $\scr{R}$ is the collection of Reeb chords. From the perspective of the front projection, these consist of pairs of points $(p,q) \in \pi_{f}(\Lambda)$ such that $q$ lies directly above $p$ in the $z$-direction, and $T_p(\pi_{f}(\Lambda))$ and $T_q(\pi_{f}(\Lambda))$ are parallel, so that the $y_j$-coordinates of the corresponding points on $\Lambda$ are equal.

In order to verify that we have the correct underlying graded algebra, then, we need an explicit model where we identify $\scr{R} \cong F \sqcup \{w,x,y,z\}$. A front projection for the Legendrian satellite of $N$-weaves appears in work of Casals and Zaslow \cite{CZ}, though a little more care is needed to identify the Reeb chords. Assuming we have chosen a face at infinity $f_{\infty}$, we take the front as follows; see Figure \ref{fig:contact_front} for reference. First, we take an immersion of the double branched cover of a large disk containing $G_{f_{\infty}}$, branched at the vertices of $G$, and where the projection map $\R^3_{x_1,x_2,z} \rightarrow \R^2_{x_1,x_2}$ recovers the branching map, and where the immersion crosses itself along the edges and is given by a so-called $D_4^-$ singularity at vertices. We take this branched cover to be relatively flat, so that the corresponding portion of the Legendrian in $\R^{5}_{\mathrm{std}}$ has relatively small $y_1$ and $y_2$ coordinates, and accordingly, we call this the flat portion of our model. The boundary of the flat portion consists of two circles. We then glue in two more disk pieces, sewn together along the two boundary circles along cusp edges so that the projection lifts to a smooth Legendrian. We take these two disks to be $C^1$-close to each other but not intersect each other, and to be cambered upwards so that they are bell-shaped, rising mostly well above the flat portion, except that the bottom disk will pass the top sheet of the flat portion near the seam. If the lower disk is the graph of a bell-shaped function $f \colon \R^2_{x_1,x_2} \rightarrow \R_z$, then we may take the upper disk to be given by the graph of the nearby bell-shaped function $f + \epsilon + \epsilon^2 x_1$ for small enough $\epsilon > 0$.

Let us now find the Reeb chords. By construction, we see that there are no Reeb chords between the two cambered disks. For the flat portion, we find that for each finite face, there is a unique Reeb chord between the two sheets, given at the point where the two sheets are of maximal distance away. Because of the positioning of the seams according to the functions $f$ and $f+\epsilon+\epsilon^2x_1$, we will also find that there must be a Reeb chord appearing for the face at infinity, since the two sheets are not at maximal distance along the boundary. Finally, between the flat and cambered portions, we can choose $f$ so that there are no Reeb chords near the seams, and so the only other Reeb chords which occur are near the maximum of the function $f$, where we obtain $4=2\times 2$ more Reeb chords, given by picking one of the two cambered disks, and one of the two sheets below. The seed of the garden is meant to codify the approximate position of these Reeb chords.

\begin{figure}[h]
	\centering
	\includegraphics[width=.6\textwidth]{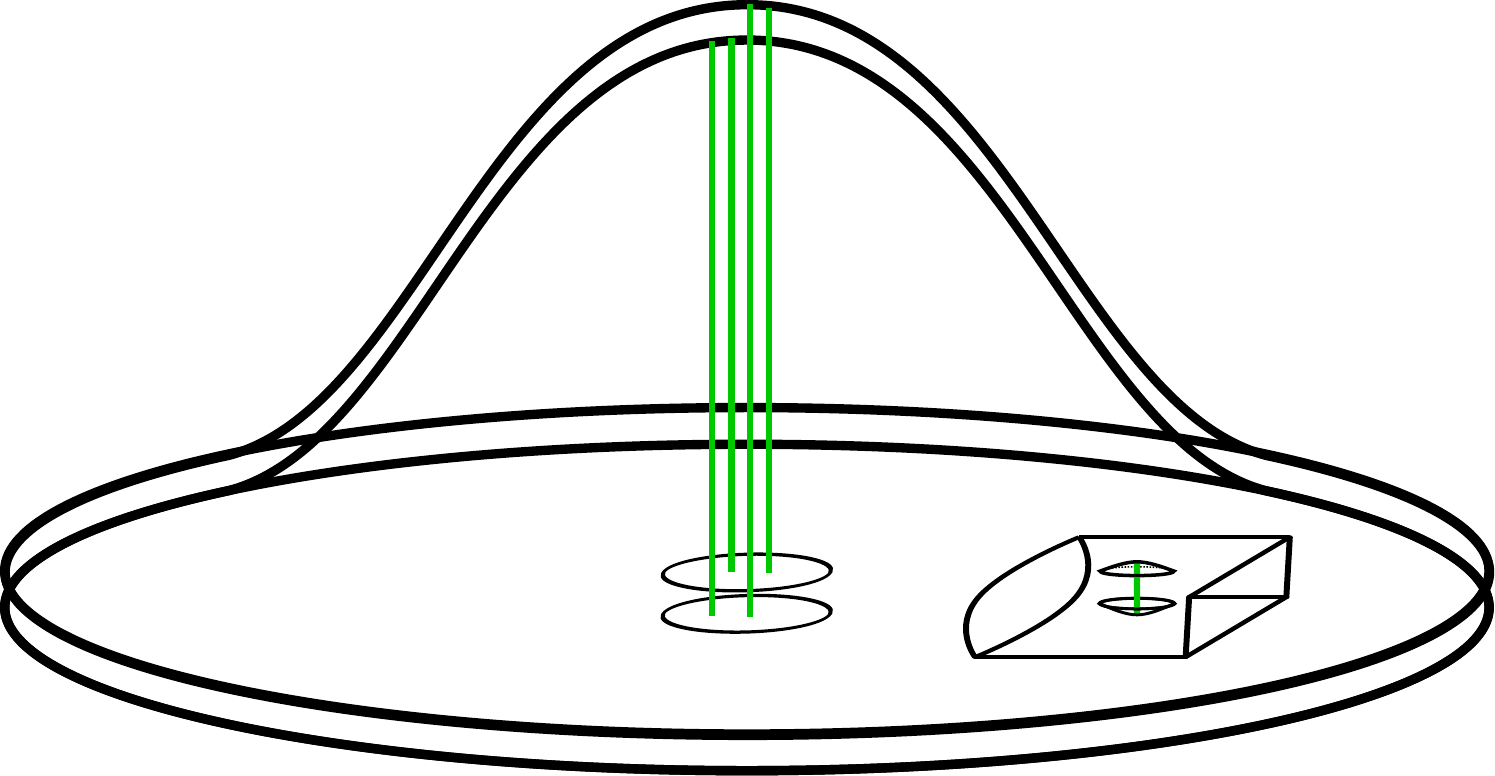}
	\caption{We have a sketch of the front of the Legendrian satellite of $\Lambda_G$. We have not drawn the details of the flat portion for simplicity, though we sketch in green one of the Reeb chords corresponding to a center of the garden inside one of the faces of the graph $G$. We have also drawn the four Reeb chords corresponding to the generators $w,x,y,z$.}
	\label{fig:contact_front}
\end{figure}

The grading of the generators in $\scr{R}$ may be computed using a formula of Ekholm, Etnyre, and Sullivan \cite[Lemma 3.4]{EES_noniso}. Namely, suppose $r \in \scr{R}$ is a Reeb chord from $r_-$ to $r_+$, lying on sheets in the front projection which are locally  the graphs of functions $f_-$ and $f_+$. Pick a generic path $\gamma$ on $\Lambda$ from $r_+$ to $r_-$. Let $U(\gamma)$ and $D(\gamma)$ be the number of times $\gamma$ crosses the cusp edges (where the cambered disks are sewn to the flat sheets) upwards or downwards, respectively. Let $\mathrm{Morse}(r)$ be the Morse index of the critical point of $f_+-f_-$ at the point over which the Reeb chord $r$ lies. Then the degree of $r$ is
$$|r| = D(\gamma)-U(\gamma) + \mathrm{Morse}(r)-1.$$
In our case, we see that $\mathrm{Morse}(r) = 2$ for all of the Reeb chords involved. For those Reeb chords in $F$, we may take $\gamma$ to stay within the flat portion, hence $D(\gamma) = U(\gamma) = 0$, and hence $|f| = 1$. For the Reeb chords $w,x,y,z$, we see that we can take a curve $\gamma$ which crosses one cusp downwards, so that $D(\gamma) = 1$ and $U(\gamma) = 0$, from which we find that $|w|=|x|=|y|=|z|=2$. We have therefore recreated the graded algebra.

As for the differential, we appeal to the gradient flow tree technology of Ekholm \cite{Ekholm}. Each sheet of $\pi_{f}(\Lambda)$ may be thought of as the graph of a function $\R^2_{x_1,x_2} \rightarrow \R_{z}$. The gradient flow trees counted by Ekholm may be thought of as certain trees in $\R^2_{x_1,x_2}$ such that each edge of the tree is labelled by two sheets, say sheets $i$ and $j$ given by graphs of functions $f_i < f_j$, and where the edges are tangent to and oriented with $-\nabla (f_j-f_i)$. The vertices of the tree are of only a handful of specified types. For example, there are $Y_0$-vertices, involving three sheets which are the graphs of functions $f_i < f_j < f_k$. Near such a vertex, the tree has one input edge directed as $-\nabla(f_k-f_i)$, and two output edges directed as $-\nabla(f_j-f_i)$ and $-\nabla(f_k-f_j)$. The differential of a Reeb chord corresponds to counts of gradient flow trees which emanate from the Reeb chord, i.e. such that the tree has a positive 1-valent puncture, meaning a node corresponding to the Reeb chord and with an outgoing edge oriented along $-\nabla(f_+-f_-)$ for the functions $f_- < f_+$ whose graphs are the sheets containing the endpoints of the Reeb chord. Many such flow trees exist, but only a finite number of them have the property that they cannot be deformed. It is these rigid flow trees which are counted in the differential.

For a Reeb chord corresponding to a face, there is a single gradient flow line to each of the vertices, and upon perturbing the front near the vertex (since the $D_4^-$ singularity is not allowed in Ekholm's technology), one finds each of these contributes precisely one term to the differential (though this requires additional proof via an understanding of the perturbation of the $D_4^-$ singularity, which we do not provide). These flow lines are codified as the threads in our garden, and they are easily seen to be the only terms which can possibly contribute to the differential of the Reeb chords at the faces.

Meanwhile, for each of the four Reeb chords between the cambered and flat portions, because the slopes of the cambered portion are $C^1$ close and the slopes in the flat portion are $C^1$ close and small, the gradient trajectories emanating from these Reeb chords are all essentially just radially outward from where these Reeb chords are positioned, approximately following the trajectory of the vector field $-\nabla f$ where $f$ is our bell-shaped function. The rigid flow trees starting at one of these four generators must pass over the face Reeb chords for degree reasons (as a 2-valent negative puncture in Ekholm's language), and there is one such trajectory of $-\nabla f$ for each face. These trajectories are codified combinatorially by the tines of the garden. But this is not the whole picture: for example, certain $Y_0$ vertices are allowed where the tines cross the threads. Indeed, suppose we have a gradient flow line along the tine, where the two sheets involved are the lower flat edge and one of the cambered edges. Then where the tine intersects the thread, this trajectory may split: one edge will still go along the tine, but following the upper flat edge (and the same cambered edge) instead, and the other edge will go along a thread until it ends at a $D_4^-$ vertex as we saw when looking at the differential of the chords at faces.

It is not hard to then enumerate all possible rigid gradient flow trees: essentially they are encoded by whether or not there is a $Y_0$-vertex where the tine crosses a thread. Figure \ref{fig:contact_flow_tree} demonstrates an example of a rigid flow tree which counts towards the differential of the generator $y$, which corresponds to the Reeb chord from the top sheet of the flat portion to the lower disk of the cambered portion. The figure also has labels $0,1$ corresponding to the fact that for each edge of the tree which is along the tine, one of the corresponding sheets is either the upper or lower of the flat sheets. Combinatorially, this data is encoded in the notion of a binary sequence, as in Definition \ref{defn:binary_seq}, with each binary sequence giving precisely one term of the differential. Verifying that these are gives all terms in the differential is not difficult, assuming again that we have understood how threads terminate at a perturbation of the $D_4^-$ singularities.

\begin{figure}[h]
	\centering
	\includegraphics[width=.6\textwidth]{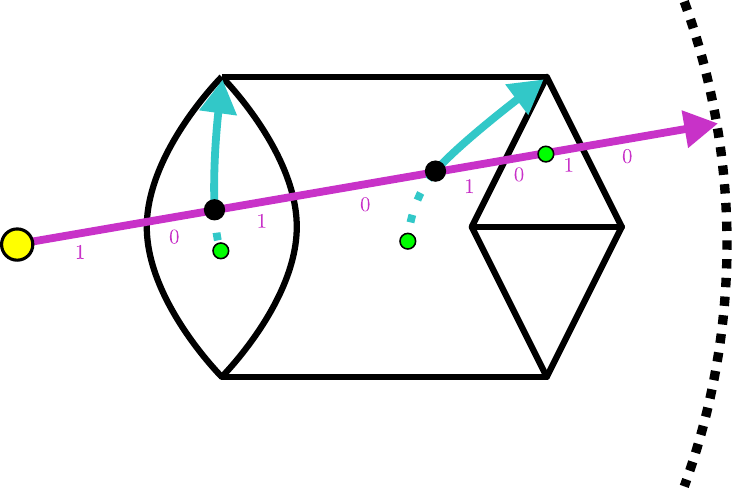}
	\caption{A sample rigid gradient flow tree which contributes to the differential of the degree 2 generator $y$. Each edge of the tree is either a thread or a tine. In Ekholm's language, the black dots are $Y_0$-vertices, the green dot contained in the tree is a $2$-valent (negative) puncture, the arrowheads occur at ends (after perturbing the $D_4^-$ singularities), and the yellow point is a $1$-valent (positive) puncture, which is just the Reeb chord $y$ in this case. Each edge of the tree may be labelled with the corresponding sheets for which it is following the negative gradient flow. Along the threads, the edge is labelled by the two flat sheets. Along the tine, each edge is labelled by one of the flat sheets, either $0$ or $1$ for the lower or upper flat sheet, and by the lower cambered sheet.}
	\label{fig:contact_flow_tree}
\end{figure}

\subsection{Acknowledgements}
The author, who was partially supported by the National Science Foundation under grant DMS-1547145, is also grateful to a number of people who influenced this paper. Emmy Murphy initially suggested this project, and provided helpful explanations of the underlying contact-geometric picture for the commutative case, especially in the early stages of this work. Discussions about Legendrian weaves with Roger Casals and Honghao Gao have influenced the expositional structure of this paper. Furthermore, encouragement from Honghao Gao helped to prevent this paper from entering development limbo. Almost all simplifications, clarifications, and corrections made from the first version of this article, as well as the inclusion of Section \ref{ssec:intro_contact}, were precipitated by the comments of two anonymous referees, whom I thank for reading the paper closely.

\section{Preliminaries on stable regenerative tame dg-homomorphisms} \label{sec:stable_tame}

In this short section, we prove a few algebraic preliminaries mentioned in the introduction about stable regenerative tame dg-homomorphisms. These proofs are somewhat technical, with techniques orthogonal to the main ideas of the paper, so we encourage the impatient reader to skip this section in a first read-through. They are presented for completeness.

We recall the following notions which may be found in the commutative coefficient setting in work of K\'{a}lm\'{a}n \cite{Kalman}, extended here to non-commutative coefficients.

\begin{defn}
	Suppose $(A,\partial^A)$ and $(B,\partial^B)$ are dg-algebras and $\phi,\psi \colon (A,\partial^A) \rightarrow (B,\partial^B)$ are dg-homomorphisms. A \textbf{$(\phi,\psi)$-derivation} is a $\Z$-linear map of degree $+1$, $K \colon (A,\partial^A) \rightarrow (B,\partial^B)$, such that if $x \in A_k$ and $y \in A$, then
	$$K(xy) = K(x)\psi(y) + (-1)^{k}\phi(x)K(y).$$
\end{defn}

\begin{lem}[{c.f. \cite[Lemma 2.18]{Kalman}}] \label{lem:extend_to_derivation}
	Suppose $(A = R*\Z\langle S \rangle,\partial^A)$ is a semi-free dg-algebra, $(B,\partial^B)$ is a dg-algebra, and $\phi,\psi \colon (A,\partial^A) \rightarrow (B,\partial^B)$ are dg-homomorphisms. If $K \colon S \rightarrow B$ is any map of degree $+1$, then there exists a unique $(\phi,\psi)$-derivation $\wt{K} \colon (A,\partial^A) \rightarrow (B,\partial^B)$ satisfying
	$$\wt{K}|_R = 0\qquad\mathrm{and}\qquad\wt{K}|_S = K.$$
\end{lem}

\begin{proof}
	All monomials in $A$ are of the form $r_0s_1r_1s_2r_2\cdots s_kr_k$ for some $k \geq 0$, $r_0,\ldots,r_k \in R$, and $s_1,\ldots,s_k \in S^+$. Because the derivation property gives $\wt{K}$ on products and by we wish $\wt{K}|_R = 0$ and $\wt{K}|_S = K$, we may inductively prove that if such a derivation exists, on monomials, it must satisfy
	$$\wt{K}(r_0s_1r_1s_2r_2\cdots s_kr_k) := \sum_{\ell=1}^{k}(-1)^{|s_1\cdots s_{\ell-1}|}\phi(r_0s_1r_1\cdots s_{j-1}r_{j-1})K(s_j)\psi(r_js_{j+1}r_{j+1}\cdots s_kr_k).$$
	This in turn determines $\wt{K}$ uniquely on all of $A$ upon extending by $\Z$-linearity, and is well-defined since the $\Z$-linear relations between these monomials are generated by those of the form
	$$r_0s_1r_1 \cdots s_{j}r_js_{j+1} \cdots s_kr_k+r_0s_1r_1 \cdots s_{j}r'_js_{j+1} \cdots s_kr_k=r_0s_1r_1 \cdots s_{j}(r_j+r'_j)s_{j+1} \cdots s_kr_k.$$
	To check this formula does indeed determine a $(\phi,\psi)$-derivation, it suffices by $\Z$-linearity to check the derivation equation on monomials, i.e. with $x = r_0s_1r_1s_2r_2\cdots s_kr_k$ and $y=r'_0s'_1r'_1s'_2r'_2\cdots s'_{k'}r'_{k'}$, which we leave for the reader to verify.
\end{proof}

\begin{lem}[{c.f. \cite[Lemma 2.18]{Kalman}}] \label{lem:htpy_of_derivation}
	Suppose $(A = R*\Z\langle S \rangle,\partial^A)$ is a semi-free dg-algebra, $(B,\partial^B)$ is a dg-algebra, $\phi,\psi \colon (A,\partial^A) \rightarrow (B,\partial^B)$ are dg-homomorphisms, and $K \colon (A,\partial^A) \rightarrow (B,\partial^B)$ is a $(\phi,\psi)$-derivation. Then the equation
	$$\phi-\psi = K\partial^A+\partial^BK$$
	is true if and only if it holds on the restriction to both $R$ and $S$.
\end{lem}
\begin{proof}
	Suppose that the equation is satisfied for elements $x \in A_k$ and $y \in A$, i.e.
	$$\phi(x)-\psi(x) = K\partial^Ax + \partial^BKx \qquad \mathrm{and} \qquad \phi(y)-\psi(y) = K\partial^Ay + \partial^BKy.$$
	Then the equation also holds for the product $xy$ by the following computation:
	\begin{align*}
	K\partial^A(xy) + \partial^BK(xy) & =K\left[(\partial^Ax)y + (-1)^kx(\partial^A y)\right] + \partial^B\left[K(x)\psi(y)+(-1)^k\phi(x)K(y)\right] \\
			&=K(\partial^Ax)\psi(y) - (-1)^{k}\phi(\partial^Ax)K(y) \\
			& \quad + (-1)^kK(x)\psi(\partial^Ay) + \phi(x)K(\partial^A y) \\
			& \quad \quad + \partial^BK(x)\psi(y) - (-1)^{k}K(x)\partial^B\psi(y) \\
			& \quad \quad \quad +(-1)^k(\partial^B\phi(x))K(y) + \phi(x)\partial^BK(y) \\
			& =(K\partial^A+\partial^BK)(x)\psi(y) + \phi(x)(K\partial^A+\partial^BK)(y)\\
			& \quad + (-1)^kK(x)(\psi\partial^A-\partial^B\psi)(y) - (-1)^{k}(\phi\partial^A - \partial^B\phi)(x)K(y) \\
			& =(\phi(x)-\psi(x))\psi(y) + \phi(x)(\phi(y)-\psi(y)) \\
			& =\phi(x)\psi(y)-\psi(xy)+\phi(xy)-\phi(x)\psi(y) \\
			& =\phi(xy)-\psi(xy).
	\end{align*}
	Since every element of $A$ can be written as a $\Z$-linear combination of products of elements of $R$ and $S$ (which are homogeneous), the result follows.
\end{proof}

We arrive now at the main results of this section.

\begin{prop}[c.f. {\cite[Corollary 3.11]{ENS}}] \label{prop:dg-homomorphism_is_quasi-isomorphism}
	A stable regenerative tame dg-homomorphism is a quasi-isomorphism.
\end{prop}

\begin{proof}
	Since elementary dg-automorphisms and regenerations are dg-isomorphisms, they obviously act by quasi-isomorphisms. It suffices to check that stabilizations and destabilizations are also quasi-isomorphisms. We will check that they are inverses up to homotopy, and hence inverses on homology.
	
	Consider a stabilization and destabilization. The composition $A_S \rightarrow A_{S^+} \rightarrow A_{S}$ is just the identity. It suffices to prove that the corresponding dg-homomorphism $\Phi \colon A_{S^+} \rightarrow A_{S} \rightarrow A_{S^+}$ induces the identity on homology. By definition, $\Phi$ is just the identity on $R$ and elements of $S$, but $\Phi(x) = 0 = \Phi(y)$. We aim now to build a chain homotopy between $\id$ and $\Phi$, i.e. a $\Z$-linear map $\wt{\Psi} \colon A_{S^+} \rightarrow A_{S^+}$ of degree $+1$ such that
	$$\id - \Phi = \wt{\Psi}\partial^+ + \partial^+\wt{\Psi}.$$
	Consider the map $\Psi \colon S^+ \rightarrow A_{S^+}$ given by
	$$\Psi(s) = \left\{\begin{matrix}0, & s \neq x \\ y, & s = x\end{matrix}\right.$$
	By Lemma \ref{lem:extend_to_derivation}, this extends to a $(\id,\Phi)$-derivation $\wt{\Psi}$. The chain homotopy equation is then easily verified on elements of $R$ and $S^+$, and so by Lemma \ref{lem:htpy_of_derivation}, $\wt{\Psi}$ is a chain homotopy as desired.
\end{proof}

\begin{prop} \label{prop:condition_htpc_identity}
	Suppose we have a dg-automorphism $\Phi \colon (A_S,\partial) \rightarrow (A_S,\partial)$ satisfying the following properties:
	\begin{itemize}
		\item $\Phi$ fixes the underlying ring $R$.
		\item For each $s \in S$, we have $\partial s \in A_{S'}$, where $S' = \{s \in S \mid \Phi(s)=s\}$.
		\item For each $s \in S$, $\Phi(s) - s$ is exact.
	\end{itemize}
	Then $\Phi$ is homotopic to the identity, and hence induces the identity on homology.
\end{prop}

\begin{proof}
	To fix notation, suppose for each generator $s \in S \setminus S'$ that $s-\Phi(s) = \partial \beta_s$ for some $\beta_s \in A_S$. Define a map $K \colon S \rightarrow A_{S}$ of degree $+1$ by
	$$K(s) = \begin{cases}\beta_s, & s \notin S' \\ 0, & s \in S'\end{cases}$$
	Then by Lemma \ref{lem:extend_to_derivation}, $K$ extends uniquely to an $(\id,\Phi)$-derivation $\wt{K} \colon A_S \rightarrow A_S$ with $\wt{K}|_{R} = 0$. Inductively from the defining equation for a derivation, one sees that $\wt{K}|_{A_{S'}}=0$. Therefore, for every $s \in S$, since $\partial s \in A_{S'}$, we have
	$$(\wt{K}\partial+\partial \wt{K})(s) = \wt{K}(\partial s) + \partial K(s) = 0+\partial \beta_s = (\id - \Phi)(s).$$
	We also have $(K\partial + \partial K)(r) = 0 = (\id-\Phi)(r)$ for every $r \in R$. By Lemma \ref{lem:htpy_of_derivation}, it follows that
	$$\id - \Phi = \wt{K}\partial + \partial \wt{K},$$
	and so $\wt{K}$ yields a chain homotopy between $\id$ and $\Phi$.
\end{proof}

\begin{prop} \label{prop:stabilize_reps}
	Suppose $(A_S,\partial)$ is a semi-free dg-algebra over a ring $R$ with stabilization $(A_{S^+},\partial^+)$, with the property that all generators of $S^+$ (and hence also of $S$) have strictly positive degree. Then for any positive integer $r$ and field $\F$, the stabilization and destabilization maps induce inverse bijections
	$$\mathrm{Rep}_r(A_S,\partial;\F) \cong \mathrm{Rep}_r(A_{S^+},\partial^+;\F)$$
	via pullback.
\end{prop}

\begin{proof}
	Let $i \colon (A_S,\partial) \rightarrow (A_{S^+},\partial^+)$ and $\pi \colon (A_{S^+},\partial^+) \rightarrow (A_S,\partial)$ be the stabilization and destabilization dg-homomorphisms. Since $\pi \circ i$ is the identity, we have
	$$i^* \circ \pi^* \colon \mathrm{Rep}_r(A_S,\partial;\F) \rightarrow \mathrm{Rep}_r(A_S,\partial;\F)$$
	is the identity. On the other hand, $i \circ \pi$ is the identity on the coefficient ring, and representations are only nontrivial on the coefficient ring since all of the generators of $S^+$ have positive degree. Hence, we also have that its pullback
	$$\pi^* \circ i^* \colon \mathrm{Rep}_r(A_{S^+},\partial^+;\F) \rightarrow \mathrm{Rep}_r(A_{S^+},\partial^+;\F)$$
	is the identity.
\end{proof}

\section{The coefficient ring} \label{sec:coef}

Recall from the introduction that, in generalizing the Casals--Murphy graded algebra (even before constructing the differential), there were two algebraic steps to perform, the first introducing noncommutativity between the coefficient ring and the other generators, and the second in introducing non-commutativity to the coefficient ring itself. The first will be discussed when we define the differential. But before we can even write the differential, we need a combinatorial model for the coefficient ring. We will develop a family of such combinatorial models, depending upon a choice of tree $T \subset G$ spanning all but one vertex which we will use as a base point. We will then see how to relate these models using homotopy classes of paths between the base points, a step which should feel well-motivated by the statement of Theorem \ref{thm:nc_CM_functoriality}.

\begin{rmk}
	Throughout this section, we will assume we have fixed a garden (to be defined later), and hence remove the corresponding decorations from our dg-algebras. Indeed, we are interested in the coefficient ring in this section, so it is better to keep the notation free from unnecessary clutter.
\end{rmk}

\subsection{The commutative case} \label{ssec:comm_coef}

To begin, we consider the commutative case, as in the work of Casals and Murphy \cite{CM_DGA}. Recall from the introduction that in the commutative case, we work over the coefficient ring $R = \Z[H_1(\Lambda_G)]$. The surface $\Lambda_G$ is a genus $g$ surface, so $H_1(\Lambda_G)$ is non-canonically isomorphic to $\Z^{2g}$. In order to have a combinatorial description for $R$, we therefore need to choose an isomorphism of $H_1(\Lambda_G)$ with $\Z^{2g}$. To do this, we make a choice of a tree $T \subset G$ spanning all but one of the vertices, as in Figure \ref{fig:Tree}.
\begin{figure}[h]
\centering
\includegraphics[width=.6\textwidth]{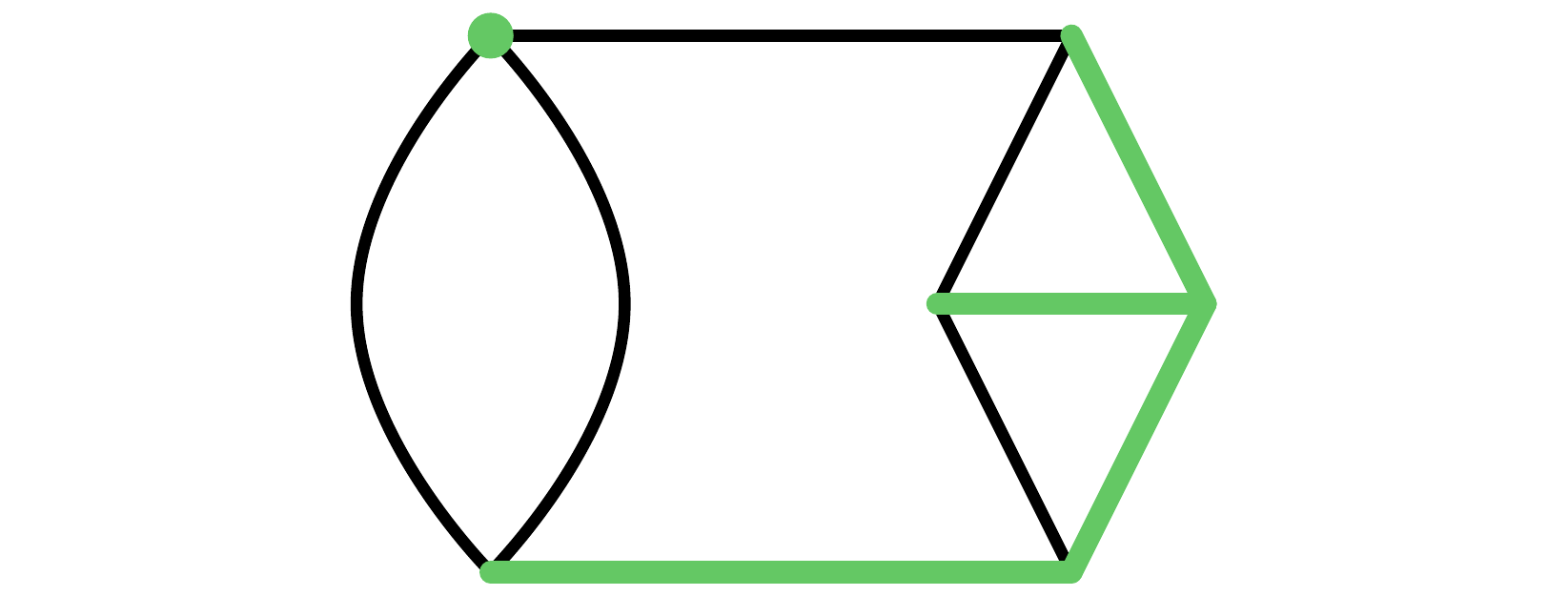}
\caption{The green edges form a tree spanning all but one vertex, singled out also in green.}
\label{fig:Tree}
\end{figure}
Such a tree automatically contains $2g$ edges. Let $\Z^T$ be the free abelian group on the set of edges of $T$. Each edge in the graph $G$ has preimage in $\Lambda_G$ given by a circle, and that circle represents a corresponding $1$-cycle in $H_1(\Lambda_G)$. (There is a minor issue of orientations which we ignore, as the geometry of $\Lambda_G$ makes this choice canonical.) Furthermore, one may check that for a tree as above, these $1$-cycles form a basis for homology. One may then take the Poincar\'e dual basis of $H_1(\Lambda_G)$. Hence we obtain for each edge $e \in T$ the unique primitive cycle in $H_1(\Lambda_G)$ which intersects $e$ once (upon fixing orientations properly) and all other edges of the tree zero times. Thinking of this association as a map from $T$ to $H_1(\Lambda_G)$, we obtain, from the universal property of free abelian groups, a group isomorphism
$$\varphi_T \colon \Z^T \xrightarrow{\sim} H_1(\Lambda_G).$$
Setting $R_T = \Z[\Z^T]$, upon applying the group ring functor, we obtain ring isomorphisms
$$\varphi_T \colon R_T \xrightarrow{\sim} R,$$
denoted by the same symbol by abuse of notation. Hence, for the underlying Casals--Murphy graded algebra $\scr{A}$, we have for each such tree $T$ a combinatorial model $\scr{A}_T$ where
$$\scr{A}_T = R_T \langle F_{\mathrm{fin}},x,y,z \rangle.$$
The isomorphisms at the level of rings hence induce isomorphisms at the level of graded algebras
$$\varphi_T \colon \scr{A}_T \xrightarrow{\sim} \scr{A},$$
again notated by the same symbol by abuse of notation. The differential $\partial^{\scr{A}}$ may be computed in any of these models to yield dg-algebras $(\scr{A}_T,\partial^{\scr{A}}_T)$, and we hence obtain canonical dg-isomorphisms
$$\varphi_T \colon (\scr{A}_T,\partial^{\scr{A}}_T) \xrightarrow{\sim} (\scr{A},\partial^{\scr{A}}).$$
In other words, we may choose any of these combinatorial models as defining the same object. In order to change between combinatorial models, we simply apply the transition maps $\Phi_T^{T'} = \varphi_{T'}^{-1} \circ \varphi_T$.

The perspective of Casals and Murphy \cite{CM_DGA} is purely combinatorial, and hence, they work by constructing the various $(\scr{A}_T,\partial^{\scr{A}}_T)$ and proving that they may all be canonically identified via canonical isomorphisms. This is surprisingly tricky, as the transition maps between $R_T$ and $R_{T'}$ look somewhat complicated, and accounts for quite a bit of effort to achieve. In the non-commutative setting, we are automatically confronted with this issue; as described in the introduction, different choices of trees yield different choices of capping paths, and as is clear from the statement of Theorem \ref{thm:nc_CM_functoriality}, this makes the naturality of the corresponding transition dg-isomorphisms a bit harder. We follow the Casals--Murphy approach of working with a slightly larger dg-algebra, from which the models for each $T$ are induced naturally. Transitioning between different choices of $T$ in the non-commutative setting is then discussed in Section \ref{ssec:transition}.

To this end, consider the collection $E$ of all of the edges of the graph. Then we may form the free abelian group $\Z^E$ and consider the group ring $\wt{R} := \Z[\Z^E].$ We have natural projection maps
$$\pi_T \colon \Z^E \twoheadrightarrow \Z^T,$$
which upon applying the group ring functor yields projections
$$\pi_T \colon \wt{R} \twoheadrightarrow R_T.$$
If we fix an explicit indexing of $E$ such that the edges of $T$ are listed first, this last map is just the map
$$\Z[e_1^{\pm 1},\ldots,e_{3g+3}^{\pm 1}] \twoheadrightarrow \Z[e_1^{\pm 1},\ldots,e_{2g}^{\pm 1}]$$
which evaluates a given Laurent polynomial at $e_{2g+1} = \cdots = e_{3g+3} = 1$.

The strategy of Casals and Murphy \cite{CM_DGA} is to work instead over the coefficient ring $\wt{R}$ and use the projections $\pi_T$ to understand the transition maps. That is, on the graded algebra
$$\wt{\scr{A}} = \wt{R} \langle F_{\mathrm{fin}},x,y,z \rangle,$$
they find a differential $\wt{\partial}^{\scr{A}}$. Reduction of coefficients then yields a graded algebra homomorphism
$$\pi_T \colon \wt{\scr{A}} \rightarrow \scr{A}_T$$
given by projecting the coefficients along $\pi_T \colon \wt{R} \rightarrow R_T$. This allows them to construct the differentials $\partial^{\scr{A}}_T$ via the following easy but important lemma.

\begin{lem}\label{lem:induced_morphisms}
	Suppose $A_1$, $A_2$, $B_1$, and $B_2$ are objects in a category, with morphisms $\Phi^A \colon A_1 \rightarrow A_2$, $p_1 \colon A_1 \rightarrow B_1$, and $p_2 \colon A_2 \rightarrow B_2$. If $p_1$ is an epimorphism, then there exists at most one morphism $\Phi^B \colon B_1 \rightarrow B_2$ making the following diagram commute:
	$$\xymatrix{A_1 \ar@{->>}[d]_-{p_1} \ar[r]^{\Phi^A} & A_2 \ar[d]^-{p_2} \\ B_1 \ar@{-->}[r]^{\Phi^B} & B_2}$$
	If $p_1$ has a right inverse $i_1 \colon B_1 \rightarrow A_1$ (i.e. $p_1 \circ i_1 = \mathrm{id}_{B_1}$), then such $\Phi^B$ exists if and only if
	$$p_2 \circ \Phi^A \circ i_1 \circ p_1 = p_2 \circ \Phi^A,$$
	in which case $\Phi^B = p_2 \circ \Phi^A \circ i_1$.
\end{lem}

\begin{proof}
	Uniqueness of $\Phi^B$ follows from the epimorphism property of $p_1$. Assuming $p_1$ has a right inverse, then we obtain a formula for such $\Phi^B$, should it exist, as
	$$\Phi^B = \Phi^B \circ p_1 \circ i_1 = p_2 \circ \Phi^A \circ i_1.$$
	Commutativity of the diagram is just the desired equation.
\end{proof}

As a special case, let us use the category of graded algebras to see how a differential may descend. Suppose $(A,\partial^A)$ is a dg-algebra, $B$ is a graded algebra, $p \colon A \rightarrow B$ is a surjective morphism of graded algebras, and $i \colon B \rightarrow A$ is a right inverse for $p$. Then so long as $p \circ \partial^A \circ i \circ p = p \circ \partial^A$, by  applying Lemma \ref{lem:induced_morphisms} we obtain $\partial^B = p \circ \partial^A \circ i$ on $B$ fitting into the following diagram:
$$\xymatrix{A \ar@{->>}[d]_-{p} \ar[r]^{\partial^A} & A \ar[d]^-{p} \\ B \ar[r]^{\partial^B} & B}$$
We claim that $(B,\partial_B)$ is a dg-algebra, so that $p \colon (A,\partial^A) \rightarrow (B,\partial^B)$ is a dg-homomorphism. The differential property follows because
$$(\partial^B)^2 = p \circ \partial^A \circ i \circ p \circ \partial^A \circ i = p \circ (\partial^A)^2 \circ i = 0,$$
where we have used that $p \circ \partial^A \circ i \circ p = p \circ \partial^A$. The Leibniz rule is easily checked from the fact that $p$ and $i$ are homomorphisms of graded algebras with $p \circ i = \mathrm{id}$.

In the Casals--Murphy setting, one obtains the desired differentials $\partial^{\scr{A}}_T$ uniquely descending from $\wt{\partial}^{\scr{A}}$ as above, so that we have induced dg-homomorphisms
$$\pi_T \colon (\wt{\scr{A}},\wt{\partial}^{\scr{A}}) \rightarrow (\scr{A}_T,\partial^{\scr{A}}_T).$$
To see this, one upgrades the natural right inverse $i_T \colon \Z^T \hookrightarrow \Z^E$ to $\pi_T \colon \Z^E \rightarrow \Z^T$ to the graded algebra level and checks that Lemma \ref{lem:induced_morphisms} applies. In this manner, $\partial^{\scr{A}}_T$ is actually very easy to compute. One simply looks at the combinatorial formulas for $\wt{\partial}^{\scr{A}}$ (which are very explicit), and ignores any edge not in $T$ which appears as an output, replacing it with the value $1$.

\subsection{The non-commutative case} \label{ssec:nc-coef}

In the non-commutative case, we follow the same procedure. Hence, we will need to define the ring $\wt{R}^{nc}$ as well as the rings $R^{nc}_T$ for each choice of tree $T$ spanning all but one vertex $v_T$ of $G$. We will prove in Section \ref{ssec:nc-coef_geom} that each of our rings $R^{nc}_T$ is canonically isomorphic to the group ring $R^{nc}_v = \Z[\pi_1(\Lambda_G, \wt{v})]$ where $v = v_T$ is the vertex missed by $T$.

Notice, in this case, that it is not really correct to talk about just one ring $R^{nc} = \Z[\pi_1(\Lambda_G)]$, since we need to choose a base point in order to make this a true object. Indeed, there is no canonical isomorphism between the various $R^{nc}_v$ as $v$ varies in $V$. Although $R$ has no non-commutative analogue, we will find that $\wt{R}$ does. Herein lies our point of entry into the theory.

Let $\wt{E}$ be the set of all edges together with an orientation. We denote an element by $(e,\mathfrak{o}) \in \wt{E}$, where $e \in E$, and $\mathfrak{o}$ is an orientation on $e$. We denote the opposite orientation by $\overline{\mathfrak{o}}$. Consider the free group
$$*^{\wt{E}}\Z = \underbrace{\Z * \Z * \cdots * \Z}_{|\wt{E}| = 6g+6}$$
generated by elements of $\wt{E}$. For each oriented edge $(e,\mathfrak{o}) \in \wt{E}$, we denote by $\lambda_{e,\mathfrak{o}} \in *^{\wt{E}}\Z$ the corresponding generator.

It is worthwhile to make a visual picture for the generators $\lambda_{e,\mathfrak{o}}$. For every edge, we draw four oriented transversals, two oriented each direction, and where we think of them either from going from a state $+$ to a state $-$ or vice versa. Our $\lambda_{e,\mathfrak{o}}$ generators correspond to a transversal going from $+$ to $-$, and oriented via a right-hand rule, so that the transversal orientation followed by the edge orientation give a positive orientation with respect to the sphere $S^2$. This convention is easier illustrated than stated, as is done in Figure \ref{fig:Generator_Convention}.

\begin{figure}[h]
\centering
\includegraphics[width=.6\textwidth]{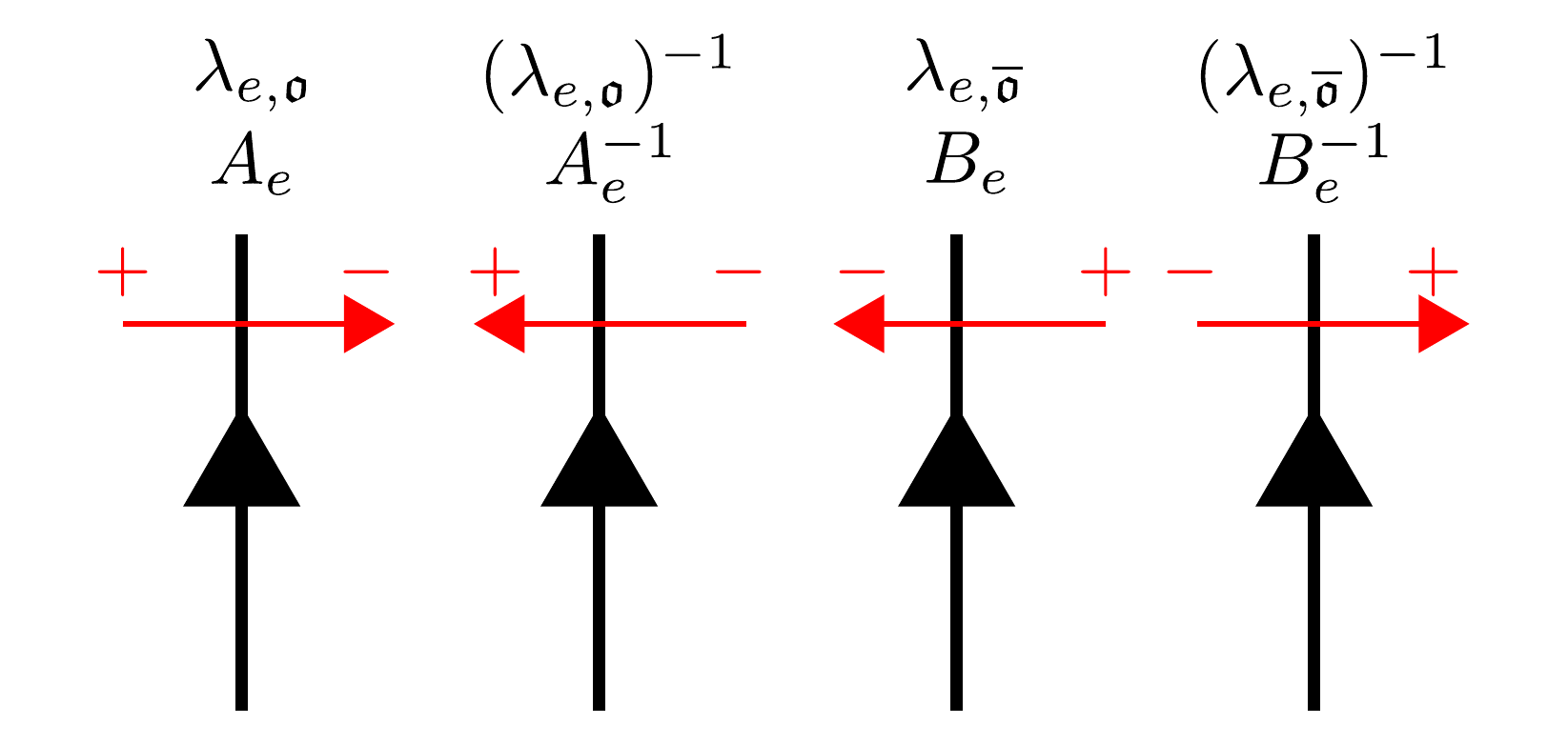}
\caption{A pictorial representation of the generators $\lambda_{e,\mathfrak{o}}$ and their inverses, in red, where the orientation $\mathfrak{o}$ on the edge $e$ is indicated by the upwards arrow. We have also written these generators in terms of $A_e$ and $B_e$, if we assume $\mathfrak{o}$ is a constituent of a fixed orientation on the entire graph.}
\label{fig:Generator_Convention}
\end{figure}

Elements of $*^{\wt{E}}\Z$ are just arbitrary words in the various $\lambda_{e,\mathfrak{o}}^{\pm 1}$, though for the most part, we only care about words formed from paths on $S^2$ as in the following definition.

\begin{defn}
	A \textbf{transverse signed path} with respect to a trivalent plane graph $G$ on $S^2$ is a pair $(\gamma,\sigma)$ consisting of a curve $\gamma \colon [0,1] \rightarrow S^2$ together with a sign $\sigma \in \{+,-\}$ such that:
	\begin{itemize}
		\item the image of $\gamma$ avoids the set of all vertices $V$
		\item if $\gamma(t) \in E$ lies on an edge, then $t \neq 0,1$ and the path passes through the edge transversely
	\end{itemize}
\end{defn}

A transverse signed path $(\gamma,\sigma)$ defines a unique lower-semicontinuous map $\wt{\sigma} \colon [0,1] \rightarrow \{+,-\}$ (we take $- < +$) with initial value $\wt{\sigma}(0) = \sigma$ and which is:
\begin{itemize}
	\item locally constant away from values of $t$ where $\gamma$ intersects an edge
	\item discontinuous at values of $t$ where $\gamma$ intersects an edge, alternating either from $+$ to $-$ or from $-$ to $+$
\end{itemize}
In other words, we may label the path of $\gamma$, starting with the label $\sigma$, so that every time $\gamma$ crosses an edge, it alternates its label. Given a transverse signed path, the map $\wt{\sigma}$ in turn can be used to obtain an element of $*^{\wt{E}}\Z$ given by reading out, in order, the generators corresponding to every time we pass through an edge, using the convention as in Figure \ref{fig:Generator_Convention}.

\begin{defn}
	Any element of $*^{\wt{E}}\Z$ obtained in this way from a transverse signed path will be called a \textbf{geometric element}.
\end{defn}

We provide not only a clarifying example, but a necessary example. Suppose that $v$ is a vertex of the graph $G$, and let $e_1$, $e_2$, and $e_3$ be the incident edges, labelled cyclically counter-clockwise, as in Figure \ref{fig:Vertex_Relation}.
\begin{figure}[h]
\centering
\includegraphics[width=.6\textwidth]{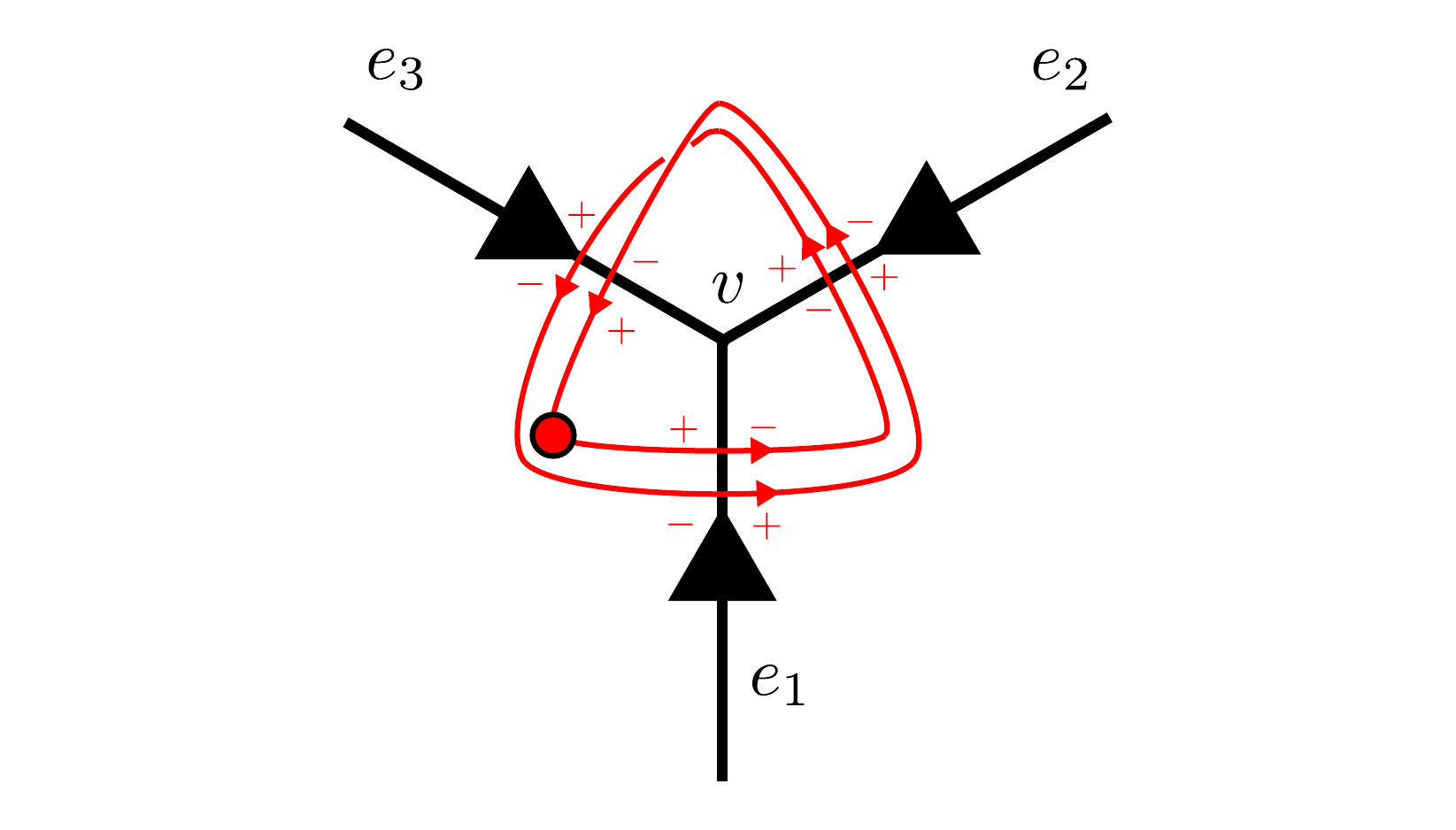}
\caption{The loop around a vertex which corresponds to the word $w_v$. If the graph is oriented as indicated by the arrows, this word is $w_v = A_{e_1}B_{e_2}^{-1}A_{e_3}B_{e_1}^{-1}A_{e_2}B_{e_3}^{-1}$.}
\label{fig:Vertex_Relation}
\end{figure}
For each $e_i$, let $\mathfrak{o}_i$ be the orientation along each edge pointing towards the vertex. (Note that it may be the case that $e_1 = e_2$, for example, in which case automatically $\mathfrak{o}_1 = \overline{\mathfrak{o}_2}$ since every edge has one source and one target.) Consider for each vertex the conjugacy class $\omega_v = [w_v]$ of the word
$$w_v := \lambda_{e_1,\mathfrak{o}_1} \lambda_{e_2,\overline{\mathfrak{o}_2}}^{-1}\lambda_{e_3,\mathfrak{o}_3} \lambda_{e_1,\overline{\mathfrak{o}_1}}^{-1}\lambda_{e_2,\mathfrak{o}_2} \lambda_{e_3,\overline{\mathfrak{o}_3}}^{-1},$$
which corresponds to a path going counterclockwise around the vertex twice. We see that changing our starting basepoint will change the word, but only by cyclic permutation of its six constituent letters, and so we will not change the conjugacy class $\omega_v$. Let $N$ be the normal subgroup of $*^{\wt{E}}\Z$ generated by these conjugacy classes over all of the vertices. We hence obtain a group
$$\Pi := \quot{\left(*^{\wt{E}}\Z\right)}{N}.$$
In other words, $\Pi$ has a presentation as $\Pi = \langle \{\lambda_{e,\mathfrak{o}}\} \mid \{w_v\} \rangle$, with $6g+6$ generators, one for each element of $\wt{E}$, and $2g+2$ relations, one for each vertex $v \in V$. The image of the geometric elements of $*^{\wt{E}}\Z$ in $\Pi$ will be called the \textbf{geometric elements} of $\Pi$.

\begin{rmk} There is nothing special about taking our word $w_v$ counter-clockwise. We could have instead considered the word corresponding to the composable loop going clockwise twice around the vertex
$$w'_v = \lambda_{e_1,\mathfrak{o}_1}^{-1} \lambda_{e_3,\overline{\mathfrak{o}_3}}\lambda_{e_2,\mathfrak{o}_2}^{-1} \lambda_{e_1,\overline{\mathfrak{o}_1}}\lambda_{e_3,\mathfrak{o}_3}^{-1} \lambda_{e_2,\overline{\mathfrak{o}_2}}.$$
In such a case, we have $(w'_v)^{-1} \in \omega_v$, so we could use any $w'_v$ in place of $w_v$ in the presentation of the group $\Pi$ since the same normal subgroup will be generated.
\end{rmk}

Finally, taking the group ring, we arrive at our coefficient ring $\wt{R}^{nc} := \Z[\Pi].$ There is a natural forgetful map
$$U \colon *^{\wt{E}}\Z \twoheadrightarrow \Z^E,$$
given by abelianizing and forgetting about the orientation. We note that the word $w_v$ is killed by $U$, yielding an induced surjective group homomorphism
$$U \colon \Pi \twoheadrightarrow \Z^E$$
also denoted $U$ by abuse of notation. At the level of rings, we arrive at a homomorphism
$$U \colon \wt{R}^{nc} \twoheadrightarrow \wt{R}.$$

Now suppose that $T \subset G$ is a tree spanning all but one vertex, and let $\wt{T}$ be the collection of edges of the tree together with orientations. We have natural group homomorphisms
$$\pi_T \colon *^{\wt{E}}\Z \twoheadrightarrow *^{\wt{T}}\Z.$$
The map $\pi_T$ simply sends $\lambda_{e,\mathfrak{o}}$ to the identity if $e \notin T$. Let $N_T$ be $\pi_T(N)$, which is automatically a normal subgroup of $*^{\wt{T}}\Z$ since $\pi_T$ is surjective. We set
$$\Pi_T := \quot{*^{\wt{T}}\Z}{N_T}$$
which has a presentation of the form $\Pi_T = \langle \{\lambda_{e,\mathfrak{o}}\} \mid \{\pi_T(w_v)\} \rangle$, with $4g$ generators corresponding to the elements $(e,\mathfrak{o}) \in \wt{T}$, and $2g+1$ relations, one for each vertex $v \in V - \{v_T\}$. (Notice that $\pi_T(w_{v_T})) = 1$, which explains why we have not included this trivial relation.) Setting $R_T = \Z[\Pi_T]$, we therefore have induced maps
$$\pi_T \colon \wt{R}^{nc} \twoheadrightarrow R^{nc}_T.$$
As before, forgetting the orientations and abelianizing gives a surjective group homomorphism
$$U_T \colon \Pi_T \twoheadrightarrow \Z^T$$
inducing a surjective ring homomorphism
$$U_T \colon R^{nc}_T \twoheadrightarrow R_T.$$
It is clear that these forgetful maps commute with the projection maps, as per the following commutative diagram.
$$\xymatrix{ \Pi \ar@{->>}[r]^{\pi_T} \ar@{->>}[d]_{U} & \Pi_T \ar@{->>}[d]_{U_T} \\  \Z^E \ar@{->>}[r]^{\pi_T} & \Z^T & }$$

\begin{rmk}
We mentioned that we are mostly interested in the geometric elements of $\Pi$. The main reason is that every element in $\Pi_T$ is a geometric element, in the sense that it lies in the image of the geometric elements of $\Pi$ under the map $\pi_T \colon \Pi \rightarrow \Pi_T$. This point is further elucidated by Section \ref{ssec:nc-coef_geom} and especially Theorem \ref{thm:geometric_coefficient_ring}, upon the realization that transverse signed paths lift uniquely to paths on $\Lambda_G$. It is probably more mathematically optimal to use a certain path algebra instead of $\Pi$, where non-geometric elements are killed, but it is not necessary for our purposes.
\end{rmk}

Admittedly, the notation $\lambda_{e,\mathfrak{o}}$ is clumsy, but is useful in that it indicates that $\wt{R}^{nc}$ is completely independent of choices. However, in the case that our graph has an orientation chosen on each of its edges, which will be part of the data of a garden for other reasons, there is simpler notation which is useful for computations. Suppose each edge $e$ of $E$ comes with a choice of orientation $\mathfrak{o}_e$. We obtain a bijection
$$\wt{E} \cong E_+ \sqcup E_-$$
where $E_+ \cong E$ consists of those edges matching the given orientation and $E_- \cong E$ consists of those edges which go against the given orientation. We shall write
$$A_e := \lambda_{e,\mathfrak{o}_e},~~~~~~~~~~B_e := \lambda_{e,\overline{\mathfrak{o}_e}}.$$
We have traded the clumsiness of the notation $\lambda_{e,\mathfrak{o}}$ for the clumsiness of the asymmetry between $A_e$ and $B_e$ which is very much dependent upon the orientation chosen on each edge of $G$. Swapping the orientation of a single edge interchanges the roles of $A$ and $B$. See Figure \ref{fig:Generator_Convention} or Figure \ref{fig:Vertex_Relation} with its caption for examples. Note that the words $w_v$ no longer necessarily alternate between $A$ and $B$ terms: this occurs only if all the edges are oriented away from the vertex (a source) or towards the vertex (a sink); except for very special graphs, it is impossible to choose an orientation so that every vertex is either a source or a sink.

\subsection{Matching the coefficients to the geometry} \label{ssec:nc-coef_geom}

We aim now to fill in the underlying geometric details of the combinatorics outlined in Section \ref{ssec:nc-coef}. The only geometric input we need from the contact picture is that we have a double cover
$$p \colon \Lambda_G \rightarrow S^2$$
branched over the set of vertices $V$ such that for each (open) face $f \in F$, the two connected components of $p^{-1}(f)$ are canonically labelled $+$ and $-$ respectively, with the property that if $\gamma$ is a curve on $\Lambda_G$ such that the curve $p \circ \gamma$ crosses precisely one edge transversely, then $\gamma$ either goes from a region labelled $+$ to a region labelled $-$, or exactly the opposite, from $-$ to $+$.

It is clear from this description that transverse signed paths are identified canonically with paths on $\Lambda_G$ by lifting along $p$ and matching the starting label $\sigma$. Conversely, if we have a path in $\Lambda_G$ which avoids the preimages of the vertices and passes transversely through the preimages of the edges, then its projection together with its initial label yield a transverse signed path. We call such a path on $\Lambda_G$ \emph{generic}, so that transverse signed paths correspond precisely to generic paths on $\Lambda_G$.

On the other hand, transverse signed paths give the geometric elements of $\Pi$. Hence, we may think of the geometric elements as those which arise from generic paths on $\Lambda_G$. Suppose now that two generic paths with the same endpoints are homotopic to each other. We may ask if they define the same geometric element in $\Pi$. The point is that a homotopy from a generic path to another generic path may in general pass through non-generic paths. However, so long as the homotopy is itself generic, it will only pass through non-generic paths a finite number of times, and the non-generic behavior is quite mild. Either:
\begin{itemize}
	\item The family passes through a vertex $\wt{v} \in \Lambda_G$ (corresponding to $v \in V$). In this case, so long as the homotopy is generic, this has the effect of changing the geometric element by the vertex relation $w_{v}$, which is trivial in $\Pi$, and hence there is actually no effect.
	\item The family passes through a tangency to an edge. A good picture of the behavior for a generic homotopy is given by Move II in Figure \ref{fig:Moves_I-IV}, but where the purple curve is thought of as our path at the critical moment (instead of a \emph{tine}). In this case, the word is modified by insertion of a product $\lambda_{e,\mathfrak{o}}\lambda_{e,\mathfrak{o}}^{-1}$, which is just the identity, and hence the geometric element remains unchanged.
\end{itemize}

We have therefore proved the following.

\begin{thm} \label{thm:geometric_elements}
	Suppose $q,q' \in \Lambda_G$ are points with projection to $S^2$ lying in an open face of $G$. Suppose $\gamma_1$ and $\gamma_2$ are two generic paths in $\Lambda_G$ starting at $q$ and ending at $q'$ such that $[\gamma_1] = [\gamma_2] \in \pi_1(\Lambda_G,q,q')$, the homotopy classes of paths from $q$ to $q'$. Then the geometric elements $w_1, w_2 \in \Pi$ corresponding to $\gamma_1$ and $\gamma_2$ are equal, $w_1 = w_2$. We therefore obtain canonical group homomorphisms
	$$\Psi_{q,q'} \colon \pi_1(\Lambda_G,q,q') \rightarrow \Pi.$$
	Furthermore, if $\gamma \in \pi_1(\Lambda_G,q,q')$ and $\gamma' \in \pi_1(\Lambda_G,q',q'')$, then
	$$\Psi_{q,q''}(\gamma * \gamma') = \Psi_{q,q'}(\gamma) \cdot \Psi_{q',q''}(\gamma').$$
\end{thm}

\begin{rmk}
	When $q=q'$, we write simply $\Psi_q := \Psi_{q,q}$.
\end{rmk}

\begin{rmk}
	Suppose $q_1$ and $q_2$ are connected by a path in $\Lambda_G$ which remains completely in the preimage of an open face $f \in F$, i.e. they lie in the same connected component of the preimage of $f$ in $\Lambda_G$. Since this component is just one of the sheets lying over $f$, it is a topological disk, the groups $\pi_1(\Lambda_G,q_1,q')$ and $\pi_1(\Lambda_G,q_2,q')$ are canonically identified, and under this identification, $\Psi_{q_1,q'} = \Psi_{q_2,q'}$. Similarly, we have an equality $\Psi_{q,q_1'} = \Psi_{q,q_2'}$ upon identifying the domains of these maps.
\end{rmk}

Theorem \ref{thm:geometric_elements} has given us an algebraic output in $\Pi$ for paths on $\Lambda_G$. We will now find a natural way to recognize $\Pi_T$ as identified with the fundamental group of $\Lambda_G$, with basepoint given by the unique preimage $*_T = p^{-1}(v_T)$ in $\Lambda_G$ singled about the vertex $v_T$. Towards this purpose, notice that $p^{-1}(S^2 \setminus T)$ is a branched double cover of a topological disk, branched at a single point, so $p^{-1}(S^2 \setminus T)$ is a topological disk. Therefore, if $q \in p^{-1}(S^2 \setminus T)$, then there is a unique homotopy class of path in $\gamma_{q,T} \in \pi_1(\Lambda_G,*_T,q)$ which admits a representative whose projection avoids $T$. With this notation, we have the following theorem.

\begin{thm} \label{thm:geometric_coefficient_ring}
There is a canonical isomorphism
$$\varphi_T \colon \Pi_T \xrightarrow{\sim} \pi_1(\Lambda_G,*_T)$$
such that for any $q,q' \in \Lambda_G$ lying above an open face of $G$, the composition
$$\varphi_T \circ \pi_T \circ \Psi_{q,q'} \colon \pi_1(\Lambda_G,q,q') \rightarrow \pi_1(\Lambda_G,*_T)$$
is the isomorphism given by
$$[\gamma] \mapsto [\gamma_{q,T} * \gamma * (\gamma_{q',T})^{-1}].$$
In particular, for $\lambda_{e,\mathfrak{o}} \in \Pi$, writing $\lambda_{e,\mathfrak{o}} = \Psi_{q_f^+,q_g^-}(\gamma_{e,\mathfrak{o}})$ for a short path $\gamma_{e,\mathfrak{o}}$ starting at $q_{f}^+$ on the $+$ sheet above face $f$, ending at $q_{g}^-$ on the $-$ sheet above face $g$, and with projection crossing only the edge $e$ once, we have
$$(\varphi_T\circ \pi_T)(\lambda_{e,\mathfrak{o}}) = [\gamma_{q_{f}^+,T} * \gamma_{e,\mathfrak{o}} * (\gamma_{q_g^-,T'})^{-1}].$$
\end{thm}

\begin{proof}
For notational convenience, around each vertex $v \in V$ (including $v_T$), we draw a small open disk $D_v$ in $S^2$, all of which are disjoint, and we pick any point $q_v \in \partial \overline{D_v}$ not lying on an edge of $G$. Let $\nu(V) = \bigsqcup D_v$ be the union of these disks. Also for notational convenience, we fix on the edges of $T$ an orientation. To compute the fundamental group $\pi_1(\Lambda_G,*_T)$, we proceed in three steps.

\noindent\emph{Step 1:} First, we compute the fundamental group $\pi_1(X,q_{v_T})$ where $X := S^2 \setminus \nu(V)$.

The space $X$ is just a copy of $S^2$ with $2g+2$ disks removed, which deformation retracts to a wedge of $2g+1$ circles, so $\pi_1(X;q_{v_T})$ is isomorphic (non-canonically) to a free group of rank $2g+1$. We aim to find a convenient system of generators. Since $T$ is contractible, for each edge $e \in T$, we may pick a loop $\overline{A}_e \in \pi_1(X,q_{v_T})$ which intersects the tree only along the edge $e$ transversely so that it is oriented in the same way that we described for $\lambda_{e,\mathfrak{o}} = A_e$ with respect to the chosen orientation on $e$ (see again Figure \ref{fig:Generator_Convention}). Also, let $\gamma \in \pi_1(X;q_{v_T})$ be the loop which follows $\partial \overline{D_{v_T}}$ around once, say in the counter-clockwise direction (although this choice will end up being irrelevant). These $\overline{A_e}$, together with $\gamma$, form the generators we are after.

\begin{lem}
The collection $\{\overline{A}_e \mid e \in T\} \cup \{\gamma\}$ forms a free system of generators for $\pi_1(X;q_{v_T})$.
\end{lem}

Though the choice of $\overline{A_e}$ is not canonical, it is easy to encode the ambiguity in this choice, and this ambiguity will happily disappear in the end of our computation of $\pi_1(\Lambda_G,*_T)$. The following lemma, which makes this ambiguity explicit, follows easily from the fact that $S^2 \setminus (T \sqcup \{v_T\})$ is a punctured disk with fundamental group generated by the loop $\gamma$.

\begin{lem}
Suppose $\overline{A}_e$ and $\overline{A}'_e$ are two choices of paths transversely intersecting $T$ only along the edge $e$ with the same orientation. Then there exist (unique) integers $m,n \in \Z$ such that
$$\overline{A}_e = \gamma^{m}\overline{A}'_e\gamma^{n}.$$
\end{lem}

\noindent\emph{Step 2:} We compute $\pi_1(Y,q_{v_T}^+)$, where $Y = p^{-1}(X)$ and $q_{v_T}^+$ is the lift of $q_{v_T}$ to the face labelled $+$.

The projection $p \colon Y \rightarrow X$ induces a map
$$p_* \colon \pi_1(Y,q_{v_T}^+) \rightarrow \pi_1(X,q_{v_T}).$$
Since $p \colon Y \rightarrow X$ is a connected double cover, we automatically have a short exact sequence of groups
$$0 \rightarrow \pi_1(Y,q_{v_T}^+) \xrightarrow{p_*} \pi_1(X,q_{v_T}) \xrightarrow{\mathrm{par}} \Z_2 \rightarrow 0.$$
For each class in $\pi_1(X,q_{v_T})$, a generic representative will intersect the edges transversely, and the map $\mathrm{par}$ is simply the parity of the number of edges that the curve intersects.

Recall the following facts.
\begin{thm}
Suppose $G$ is a finitely generated group, and $H \leq G$ is a finite index subgroup. Then:
\begin{itemize}
	\item (\textbf{Schreier Subgroup Lemma}) Consider a right inverse $\sigma \colon H\backslash G \rightarrow G$ of the quotient $\pi \colon G \rightarrow H \backslash G$. Let $S$ be a set of generators for $G$. Then $H$ is generated by the $|S| \cdot [G \colon H]$ elements
	$$g_{c,s} := \sigma(c) \cdot s \cdot (\sigma(\pi(\sigma(c) \cdot s)))^{-1}$$
	for pairs $(c,s) \in H \backslash G \times S$.
	\item (\textbf{Nielsen-Schreier Theorem}) If $G$ is a free group, then $H$ is a free group. Furthermore, if $G$ has rank $n$ and $H$ is of finite index $k$, then $H$ has rank $1+k(n-1)$.
\end{itemize}
\end{thm}
Since $\pi_1(X,q_{v_T})$ is a free group, these two results automatically apply to our setting. By the Nielsen-Schreier Theorem, we see that $\pi_1(Y,q_{v_T}^+)$ is a free group of rank $1+2((2g+1)-1) = 1+4g$. On the other hand, we may use the generating set $S = \{\overline{A}_e \mid e \in T\} \cup \{\gamma\}$ and the section
$$\sigma \colon \Z_2 \rightarrow \pi_1(X,q_{v_T})$$
with $$\sigma(0) = e,~~~~~~~~\sigma(1) = \gamma$$
as input into the Schreier subgroup lemma. If we let $m_e := \mathrm{par}(\overline{A}_e) \in \{0,1\}$ be the parity of the loops $\overline{A}_e$ that we have chosen, then
\begin{empheq}[left=\empheqlbrace]{align*}
g_{0,\overline{A}_e} &= \overline{A}_e \cdot \gamma^{-m_e} \\
g_{0,\gamma} &= \gamma \cdot \gamma^{-1} = e\\
g_{1,\overline{A}_e} &= \gamma \cdot \overline{A}_e \cdot \gamma^{m_e-1}  \\
g_{1,\gamma} &= \gamma^2
\end{empheq}
Forgetting the non-identity element, we see that we are left with precisely $4g+1$ generators. As this matches the rank of the free group it generates, there are no relations between these generators by a result of Nielsen \cite{Nielsen} (or its English translation \cite{Nielsen_trans}), c.f. work of Federer and J\'onsson \cite[Theorem 3.13]{FJ}.

Of the two generators $\overline{A}_e \cdot \gamma^{-m_e}$ and $\gamma \cdot \overline{A}_e \cdot \gamma^{m_e-1}$, precisely one of them crosses the edge $e$ in such a way that it goes from a $+$ region to a $-$ region. We let $\wt{A}_e$ be this generator, and we let the other be $\wt{B}_e^{-1}$ (which matches our intuition from Figure \ref{fig:Generator_Convention}). We set $\wt{\gamma} = \gamma^2$. Hence, we have
$$\pi_1(Y,q_{v_T}^+) = \langle \wt{\gamma}, \{\wt{A}_e \mid e \in T\}, \{\wt{B}_e \mid e \in T\} \rangle.$$
Again, $\wt{A}_e$ and $\wt{B}_e$ are not canonically defined, since our choice of $\overline{A}_e$ was only defined up to multiplying on the left and right by powers of $\gamma$. Nonetheless, it is easy to trace through how these are allowed to change. The following lemma is an easy exercise for the reader if it is not obvious.

\begin{lem}
For two different choices $\overline{A}_e$ and $\overline{A}'_e$, the corresponding generators $\wt{A}_e$ and $\wt{A}'_e$ are related by
$$\wt{A}_e = \wt{\gamma}^m\wt{A}'_e \wt{\gamma}^{n}$$
for (unique) integers $m,n \in \Z$, and similarly for $\wt{B}_e$ and $\wt{B}'_e$.
\end{lem}

We finish Step 2 by noting not only that we have a nice presentation for $\pi_1(Y,q_{v_T}^+)$, but also that for a given element in this fundamental group, it is quite easy to read off, up to powers of $\wt{\gamma}$, the corresponding word in these generators simply by looking at the projection to $X$ of a generic representative. In particular, for a generic representative, this projection will cross the edges of $G$ transversely and be disjoint from all of the vertices. From this, we have that our loop starts on the $+$ sheet, and every time it crosses an edge, it swaps between the $+$ and $-$ sheets. Suppose this loop crosses the edges of $T$ in the order $e_1, e_2, \ldots, e_n$. Then the corresponding word representing this loop is given by
$$\wt{\gamma}^{m_1} \cdot g_1 \cdot \wt{\gamma}^{m_2} \cdot g_2 \cdots \wt{\gamma}^{m_n} \cdot g_n \cdot \wt{\gamma}^{m_{n+1}}$$
for some integers $m_i \in \Z$ and where each $g_i \in \{\wt{A}_{e_i},\wt{A}_{e_i}^{\pm 1}, \wt{B}_{e_i}, \wt{B}_{e_i}^{\pm 1}\}$ is chosen according to the convention of Figure \ref{fig:Generator_Convention}.

\noindent\emph{Step 3:} Finally, we glue back in $p^{-1}(\nu(V))$ and compute $\pi_1(\Lambda_G,*_T)$.

Each $p^{-1}(D_v)$ is itself a disk, so we may attach them one by one and apply the Seifert-van Kampen theorem to compute the fundamental group. We begin with the disk $p^{-1}(D_{v_T})$, which has boundary given by the loop $\wt{\gamma}$, so attaching this disk simply kills this generator. In particular, we have that all of the choices $\wt{A}_e$ and $\wt{B}_e$ already correspond to well-defined classes $A_e$ and $B_e$. So at this point, we have
$$\pi_1(Y \cup p^{-1}(D_{v_T}), q_{v_T}^+) = \langle \{A_e \mid e \in T\}, \{B_e \mid e \in T\} \rangle$$
In fact, we see already from the description at the end of Step 2 that it is quite easy to read off the corresponding word from the projection of the loop to $X \cup D_{v_T}$, with no ambiguity in choices since we have $\wt{\gamma} = 1$.

Now suppose we glue in a single disk $p^{-1}(D_v)$. This corresponds to killing the loop $\partial \overline{D_v}$, which, we associate with an element of $\pi_1(Y \cup p^{-1}(D_{v_T}), q_{v_T}^+)$ by choosing an arbitrary path from $q^+_{v_T}$ to $q^+_v$ which does not intersect the tree. We see that the word $w_v$ (appropriately evaluated so that all non-tree edges are given the value $1$) is precisely the word generating this loop. Therefore,
$$\pi_1(\Lambda_G,q_{v_T}^+) = \langle \{A_e \mid e \in T\}, \{B_e \mid e \in T\} \mid \{\pi_T(w_v)\}\rangle,$$
which is precisely the group $\Pi_T$ we described, i.e. $\Pi_T \cong \pi_1(\Lambda_G,q_{v_T}^+)$. Finally, all paths between $*_T$ and $q_{v_T}^+$ completely contained in $p^{-1}(D_{v_T})$ are homotopic, and this allows us to move our base point from $q_{v_T}^+$ to $*_T$ canonically.

Finally, to prove the expression for $\varphi_T \circ \pi_T \circ \Psi_{q,q'}$, we note that by the composition property of Theorem \ref{thm:geometric_elements}, it suffices to prove the result for the short generators $\gamma_{e,\mathfrak{o}}$. But this is given by tracing through the proof and identifying the generators $A_e$ and $B_e$ with these explicit loops.
\end{proof}

\begin{cor}\label{cor:right_inverse}
	There exists a group homomorphism $i_T \colon \Pi_T \rightarrow \Pi$ which is right inverse to $\pi_T \colon \Pi \rightarrow \Pi_T$.
\end{cor}

\begin{proof}
	Consider any point $q \in \Lambda_G$ lying above an open face of $G$. We obtain an identification $\scr{C}(\gamma_{q,T})\pi_1(\Lambda_G,*_T) \cong \pi_1(\Lambda_G,q)$ by conjugation by the path $\gamma_{q,T} \in \pi_1(\Lambda_G,*_T,q)$. Set
	$$i_T := \Psi_q \circ \scr{C}(\gamma_{q,T}) \circ \varphi_T \colon \Pi_T \rightarrow \Pi.$$
	This is a composition of group homomorphisms, hence a group homomorphism itself, and satisfies
	\begin{align*}
		\pi_T \circ i_T &= \pi_T \circ \Psi_q \circ \scr{C}(\gamma_{q,T}) \circ \varphi_T \\
			&= \varphi_T^{-1} \circ (\varphi_T \circ \pi_T \Psi_q) \circ \scr{C}(\gamma_{q,T}) \circ \varphi_T \\
			&= \varphi_T^{-1} \circ \scr{C}(\gamma_{q,T})^{-1} \circ \scr{C}(\gamma_{q,T}) \circ \varphi_T \\
			&= \id,
	\end{align*}
	where we have used Theorem \ref{thm:geometric_coefficient_ring} to identify $\varphi_T \circ \pi_T \circ \Psi_q$ with $\scr{C}(\gamma_{q,T})^{-1}$.
\end{proof}

To end the section, we discuss how Theorem \ref{thm:geometric_coefficient_ring} relates to the commutative coefficient setting. From the universal property for abelianization of groups, the map $U_T \circ \varphi_T^{-1} \colon \pi_1(\Lambda_G,*_T) \twoheadrightarrow \Z^T$ must factor through the abelianization map $\pi_1(\Lambda_G,*_T) \twoheadrightarrow H_1(\Lambda_G)$, yielding a surjective morphism $H_1(\Lambda_G) \twoheadrightarrow \Z^T$, as in the following commutative diagram:
$$\xymatrix{\Pi_T \ar[r]^-{\varphi_T}_-{\sim} \ar@{->>}_{U_T}[d] & \pi_1(\Lambda_G,*_T) \ar@{->>}[d]^{\mathrm{Ab}} \\ \Z^T & H_1(\Lambda_G) \ar@{-->>}[l]}$$
But this surjective morphism is between two free abelian groups of the same rank $2g$, and hence is automatically an isomorphism. In other words, $\Z^T$ is itself the abelianization of $\Pi_T$, $U_T$ is the abelianization map, and $\varphi_T$ descends to the isomorphism $\varphi_T \colon \Z^T \rightarrow H_1(\Lambda_G)$ which we used in the commutative case. We leave it to the reader to check that this definition of the isomorphism $\varphi_T \colon \Z^T \xrightarrow{\sim} H_1(\Lambda_G)$ matches our original description at the beginning of Section \ref{ssec:comm_coef}.

\subsection{Transition maps} \label{ssec:transition} 

In order to define the Casals--Murphy dg-algebra, it sufficed to perform the following steps:
\begin{itemize}
	\item define the tilde-version $(\wt{\scr{A}},\wt{\partial}^{\scr{A}})$
	\item plug the tilde-version into Lemma \ref{lem:induced_morphisms} by using the maps $\pi_T$ and $i_T$ to obtain the dg-algebras $(\scr{A}_T,\partial_T^{\scr{A}})$
	\item verify that the various versions with respect to different trees $T$ are naturally identified together
\end{itemize}
The first two steps hold over in exactly the same manner, with no changes in the non-commutative setting. It is the last step which we have not yet discussed. We will work with the (only slightly more difficult) non-commutative setting since the strategies are essentially the same; see Remark \ref{rmk:transition_commutative} below for a discussion of the commutative case.

Suppose $\gamma$ is a homotopy class of path on $\Lambda_G$ from $*_T$ to $*_{T'}$. Then there is an induced isomorphism $\scr{C}(\gamma) \colon \pi_1(\Lambda_G,*_T) \xrightarrow{\sim} \pi_1(\Lambda_G,*_{T'})$ given by
$$\scr{C}(\gamma)(\eta) := \gamma^{-1} * \eta * \gamma.$$
We aim to find, for each pair of trees $T$ and $T'$ and homotopy class of path $\gamma$ from $*_T$ to $*_{T'}$, an isomorphism $\wt{\scr{C}}_T^{T'}(\gamma) \colon \Pi \rightarrow \Pi$ fitting into the following commutative diagram:
$$\xymatrix{\Pi \ar[d]_{\pi_T} \ar[r]^{\wt{\scr{C}}_{T}^{T'}(\gamma)} & \Pi \ar[d]^{\pi_{T'}}\\ \Pi_T \ar@{-->}[r]^{\scr{C}_{T}^{T'}(\gamma)} \ar[d]_-{\varphi_T}  & \Pi_{T'} \ar[d]^-{\varphi_{T'}}\\ \pi_1(\Lambda_G,*_T) \ar[r]^{\scr{C}(\gamma)}  & \pi_1(\Lambda_G,*_{T'}) }$$
The middle arrow is dashed to indicate that it arises from the bottom square automatically as $\scr{C}_T^{T'}(\gamma) := \varphi_{T'}^{-1} \circ \scr{C}(\gamma) \circ \varphi_T$, since all morphisms in this composition are isomorphisms. We see then that by Lemma \ref{lem:induced_morphisms}, the isomorphism $\scr{C}_T^{T'}(\gamma)$ is the unique one fitting into the top square.

Let us define $\wt{\scr{C}}_T^{T'}(\gamma)$ on generators. Suppose we have some element $\lambda_{e,\mathfrak{o}} \in \Pi$ for $e \in E$. As in the statement of Theorem \ref{thm:geometric_coefficient_ring}, it is identified with a short transverse path $\gamma_{e,\mathfrak{o}}$ from $q_f^+$ on the $+$ sheet over the face $f$ to $q_g^-$ on the $-$ sheet over $g$, with projection crossing the edge $e$, i.e. $\lambda_{e,\mathfrak{o}} = \Psi_{q_{f}^+,q_g^-}(\gamma_{e,\mathfrak{o}})$. For any pairs of trees $T$ and $T'$, and homotopy class of path $\gamma$ from $*_T$ to $*_{T'}$, let
$$\beta_{T}^{T'}(f,\gamma,\pm):=\Psi_{q_{f}^{\pm}}\left[(\gamma_{q_f^{\pm},T'})^{-1} * \gamma^{-1} * \gamma_{q_f^{\pm},T}\right]$$
be the geometric element in $\Pi$ corresponding to the homotopy class of the loop $(\gamma_{q_f^{\pm},T'})^{-1} * \gamma^{-1} * \gamma_{q_f^{\pm},T}$. We define
\begin{align*}
	(\wt{\scr{C}}_T^{T'}(\gamma))(\lambda_{e,\mathfrak{o}}) &:= \beta_T^{T'}(f,\gamma,+) \cdot \lambda_{e,\mathfrak{o}} \cdot  \beta_T^{T'}(g,\gamma,-)^{-1} \\
	&= \Psi_{q_f^+,q_g^-}(\gamma_{q_f^+,T'}^{-1} * \gamma^{-1} * \gamma_{q_f^+,T} * \gamma_{e,\mathfrak{o}} *  \gamma_{q_g^-,T}^{-1} * \gamma * \gamma_{q_g^-,T'}).
\end{align*}
In order to check this is well-defined, we must check that it preserves the normal subgroup $N$ of $*^{\wt{E}}\Z$ generated by the words $w_v$. But we see that because $w_v$ is given as a geometric element, the inner factors of $\beta$ come in pairs $\beta^{-1} \cdot \beta$ which cancel, and we are left with a conjugate of $w_v$. Hence, this is a well-defined group homomorphism of $\Pi$.

\begin{lem} \label{lem:coefficient_tree_change}
	The diagram
	$$\xymatrix{\Pi \ar[d]_{\varphi_T \circ \pi_T} \ar[r]^{\wt{\scr{C}}_{T}^{T'}(\gamma)} & \Pi \ar[d]^{\varphi_{T'} \circ \pi_{T'}}\\ \pi_1(\Lambda_G,*_T) \ar[r]^{\scr{C}(\gamma)}  & \pi_1(\Lambda_G,*_{T'}) }$$
	is commutative.
\end{lem}
\begin{proof}
	It suffices to check the result for generators. Suppose we have $\lambda_{e,\mathfrak{o}}$, corresponding to the short transverse signed path $\gamma_{e,\mathfrak{o}}$ from $q_f^+$ to $q_g^-$. Using Theorem \ref{thm:geometric_coefficient_ring} to compute $\varphi_T \circ \pi_T$ and $\varphi_{T'} \circ \pi_{T'}$, we have
	\begin{align*}
		(\varphi_{T'} \circ \pi_{T'} \circ \wt{\scr{C}}_T^{T'}(\gamma))(\lambda_{e,\mathfrak{o}}) &=(\varphi_{T'} \circ \pi_{T'})\left(\Psi_{q_f^+,q_g^-}(\gamma_{q_f^+,T'}^{-1} * \gamma^{-1} * \gamma_{q_f^+,T} * \gamma_{e,\mathfrak{o}} *  \gamma_{q_g^-,T}^{-1} * \gamma * \gamma_{q_g^-,T'})\right) \\
		&=\gamma_{q_f^+,T'} *\gamma_{q_f^+,T'}^{-1} * \gamma^{-1} * \gamma_{q_f^+,T} * \gamma_{e,\mathfrak{o}} *  \gamma_{q_g^-,T}^{-1} * \gamma * \gamma_{q_g^-,T'} * \gamma_{q_g^-,T'}^{-1} \\
		&=\gamma^{-1} * \gamma_{q_f^+,T} * \gamma_{e,\mathfrak{o}} *  \gamma_{q_g^-,T}^{-1} * \gamma\\
		&= \scr{C}(\gamma) (\gamma_{q_f^+,T} * \gamma_{e,\mathfrak{o}} * \gamma_{q_g^-,T}^{-1})\\
		&=(\scr{C}(\gamma) \circ \varphi_T \circ \pi_T)(\lambda_{e,\mathfrak{o}}).
	\end{align*}
	Hence, the diagram does indeed commute.
\end{proof}

Recalling the various functoriality properties of Theorem \ref{thm:nc_CM_functoriality}, we will require the following lemma, where we recall again that without a tilde, $\scr{C}_T^{T'}(\gamma) = \varphi_{T'}^{-1} \circ \scr{C}(\gamma) \circ \varphi_T$.

\begin{lem}
	The following two properties hold:
	\begin{itemize}
		\item For all trees $T$ and constant paths $\id$ ad $*_T$, the maps $\wt{C}_T^T(\id)$ and $C_T^T(\id)$ are the identity on $\Pi$ and $\Pi_T$ respectively.
		\item For any triples of trees $T$, $T'$, and $T''$ and homotopy classes of paths $\gamma$ from $*_T$ to $*_{T'}$ and $\gamma'$ from $*_{T'}$ to $*_{T''}$, we have
			$$\wt{\scr{C}}_{T}^{T''}(\gamma * \gamma') = \wt{\scr{C}}_{T'}^{T''}(\gamma') \circ \wt{\scr{C}}_{T}^{T'}(\gamma) \qquad \mathrm{and} \qquad \scr{C}_{T}^{T''}(\gamma * \gamma') = \scr{C}_{T'}^{T''}(\gamma') \circ \scr{C}_{T}^{T'}(\gamma).$$
	\end{itemize}
\end{lem}
\begin{proof}
	The versions without the tildes follow from the ones with the tildes by Lemma \ref{lem:induced_morphisms}. The first property is clear by definition. The second property follows from the fact that
	$$\beta_{T'}^{T''}(f,\gamma',\pm) \cdot \beta_T^{T'}(f,\gamma,\pm) = \beta_{T}^{T''}(f,\gamma,\pm)$$
	because the corresponding paths are homotopic:
	$$\left[(\gamma_{q,T''})^{-1} * (\gamma')^{-1} * \gamma_{q,T'}\right] * \left[(\gamma_{q,T'})^{-1} * \gamma^{-1} * \gamma_{q,T}\right] \sim (\gamma_{q,T''})^{-1} * (\gamma * \gamma')^{-1} * \gamma_{q,T}$$
	where $q = q_f^{\pm}$.
\end{proof}

Unfortunately, it is not the case that the naively induced graded algebra morphisms $\wt{\Phi}_{T}^{T'}(\gamma) \colon \wt{\scr{B}} \rightarrow \wt{\scr{B}}$ preserve the differential $\wt{\partial}^{\scr{B}}$. We will need to further act on generators via regenerations (as suggested by Theorem \ref{thm:nc_CM_functoriality}). However, since we have not yet constructed the differentials, it is too early to discuss such matters. We relegate a treatment to Section \ref{ssec:tree_time}.

\begin{rmk} \label{rmk:transition_commutative}
	The story is much the same in the commutative case, and can be recovered from the non-commutative case upon abelianization. Namely, one finds a collection of isomorphisms
	$$\wt{\Phi}_T^{T'} \colon \Z^E \rightarrow \Z^E$$
	with $\wt{\Phi}_T^T = \id$ and $\wt{\Phi}_T^{T'} \circ \wt{\Phi}_{T'}^{T''} = \Phi_{T}^{T''}$ and fitting into the following commutative diagram:
	$$\xymatrix{\Z^E \ar[d]_{\pi_T} \ar[r]^{\wt{\Phi}_T^{T'}} & \Z^E \ar[d]^{\pi_{T'}} \\ \Z^T \ar@{-->}[r]^{\Phi_T^{T'}} \ar[d]_-{\varphi_T} & \Z^{T'} \ar[d]^-{\varphi_{T'}}\\ H_1(\Lambda_G) \ar@{=}[r]  & H_1(\Lambda_G) }$$
	Our construction of a combinatorial model for $\Phi_T^{T'}$ is a little different from the one implicitly presented in Casals and Murphy, who use a group action of $\Z^F$ on $\wt{\scr{A}}[t^{\pm 1}]$ to obtain $(\scr{A},\partial^{\scr{A}})$ as a GIT quotient. Along the way, they construct maps $s_T \colon \Z^T \rightarrow \Z^E$ \cite[Proof of Theorem 6.6]{CM_DGA} such that one may take $\Phi_T^{T'} = \pi_{T'} \circ s_T$. It is not much harder to construct the lift $\wt{\Phi}_T^{T'}$ as well, though this is not done in their paper.
\end{rmk}

\section{The enlarged non-commutative Casals--Murphy dg-algebra}

In this section, we will define a dg-algebra associated to a trivalent plane graph $G$ and a choice of garden $\Gamma$. This dg-algebra is enlarged from the version appearing in Theorem \ref{thm:nc_CM_functoriality}. The functoriality properties of this enlarged dg-algebra are then proved as Theorem \ref{thm:nc_CM_enlarged} in Section \ref{sec:functoriality_enlarged}.

\subsection{Gardens, binary sequences, and associated words}

In order to define our dg-algebra, we first need to make a number of auxiliary choices, which are grouped together in the following convenient definition.

\begin{defn} \label{defn:garden} Suppose $G$ is a trivalent plane graph. A \textbf{garden} on $G$ consists of the data of
	\begin{itemize}
		\item an \textbf{orientation}, by which we mean an orientation of all the edges of $G$
		\item a \textbf{centering}, consisting of, for each $f \in F$, a point $c_f$ contained in the interior of the face $f$, called the \textbf{center} of $f$
		\item a \textbf{web}, consisting of, for each face-vertex adjacency $(f,v)$, a \textbf{thread} $\tau_f(v)$, which is a connected embedded arc oriented from $c_f$ to $v$ completely contained in $f$ except for its endpoint at $v$, such that no two threads intersect except at their endpoints at the centers and vertices, and such that their tangent vectors at $c_f$ are all distinct
		\item a \textbf{seed}, given by a covector in the unit cotangent bundle $\xi \in S_s^*S^2$ for a \textbf{base point} $s \in S^2$
		\item a \textbf{rake}, consisting of oriented embedded loops $\gamma_f$, called \textbf{tines}, one for each $f \in F$, which start and end at the base point $s$, pass through $c_f$ and no other center, have initial and terminal tangent vectors positive with respect to $\xi$, and intersect each other only at the base point $s$
	\end{itemize}
The face $f_{\Gamma} \in F$ to which the base point $s$ belongs is called the \textbf{face at infinity with respect to $\Gamma$}.

An oriented path $\gamma \colon [a,b] \rightarrow S^2$ is said to be \textbf{non-degenerate} with respect to a garden $\Gamma$ if:
\begin{itemize}
	\item its endpoints do not lie on a vertex, edge, center, or thread
	\item it passes through edges and threads transversely at interior points
	\item it passes through centers so that it is not tangent to any thread exiting that center
\end{itemize}
The garden is said to be \textbf{non-degenerate} if the tines of the rake are all non-degenerate, and \textbf{degenerate} otherwise.

A \textbf{homotopy} of gardens is a smoothly varying path of gardens $\Gamma_t$, $0 \leq t \leq 1$. We always allow our homotopies to pass through degenerate gardens in which the tines may be degenerate; on the other hand, homotopies must still be through gardens. Hence, for example, the base point of the seed cannot pass through a thread or edge, since then it would not be a garden at that time. Similarly, a tine $\gamma_f$ cannot homotope to nontrivially intersect another tine nor to pass through a face other than $c_f$.
\end{defn}

\begin{rmk} \label{rmk:finite-type}
	Our definition of a garden is slightly different from the version of Casals--Murphy \cite{CM_DGA}; their version is essentially the same, except they take the base point $s$ of the seed to lie in a chosen face at infinity, and they exclude a tine through the face at infinity. We will handle this slightly differently when we define finite-type gardens in Definition \ref{defn:finite-type_garden}. Our definition is essentially equivalent to theirs, but places it in our context in which all of the faces are treated equally, and where the base-point is free to move around.
\end{rmk}

See Figure \ref{fig:Garden} for an example of a non-degenerate garden.
\begin{figure}[h]
	\centering
	\includegraphics[width=\textwidth]{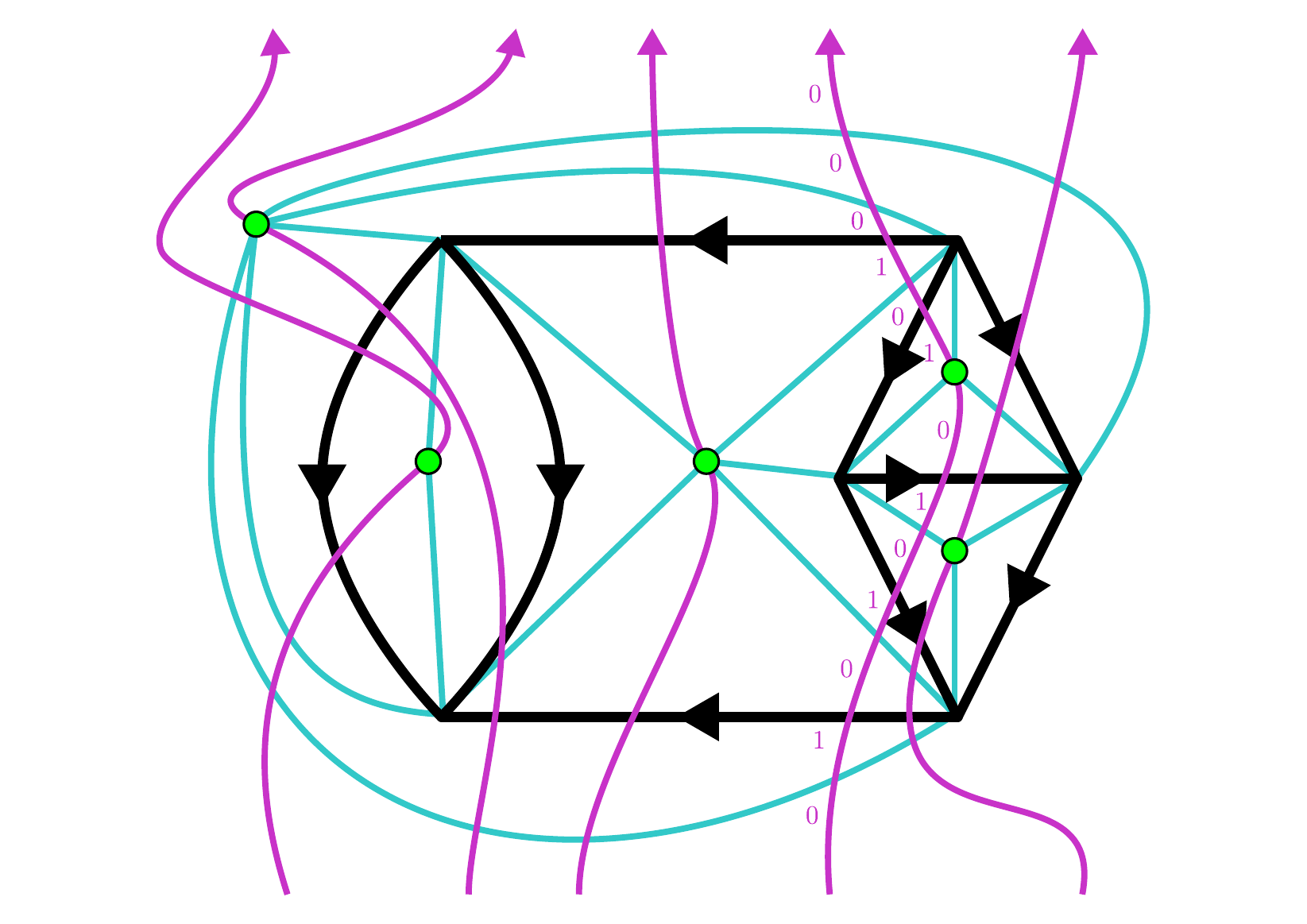}
	\caption{A non-degenerate garden. We have elected, for simplicity, to stereographically project from the base point so that $s$ is the point at infinity; the seed $\xi$ is chosen so that all tines, represented by the magenta lines are oriented from the bottom of the picture to the top. (See Figure \ref{fig:Moves_VI-VII} for how we picture the seed, when needed.) The graph, together with the orientations of the edges, are notated in black. The green circles are the centers, and the web is given by the blue threads. Since the threads are always oriented from the centers to the vertices, we exclude their orientations. On the fourth tine is a sample binary sequence from $0$ to $0$, which contributes a term to the differential of $z$ of the form $(H_1)(A_2)(H_3)(-B_4)(-H_5)(-B_6)(f_7)(A_8)(H_9)(A_{10})$, where the $A$ and $B$ terms correspond to where the tine passes an edge and are short for the corresponding elements $A_e$ and $B_e$ in $W$, the $H$ terms are shorthand for the element $H_f(v)$ for where the tine switches at a thread, and $f_7$ is the face through which the tine passes.}
	\label{fig:Garden}
\end{figure}

Up to homotopy, any two non-degenerate gardens are equivalent via swapping the orientation on edges together with finite sequences of seven classes of moves, Moves I-VII. The first five, Moves I-V, were studied by Casals and Murphy \cite{CM_DGA} (our version of Move V is presented in a slightly modified form). Remark \ref{rmk:finite-type} explains why Moves VI and VII appear here but not in the original study of Casals and Murphy.

Moves I-IV are illustrated in Figure \ref{fig:Moves_I-IV}. Move I forms or removes a bigon between a tine and a thread, whereas Move II likewise forms or removes a bigon between a tine and an edge. Move III passes a tine through a vertex of the graph. Move IV swings the tangent vector to a thread at a vertex past the tangent vector the tine. If $\Gamma$ and $\Gamma'$ are any two homotopic non-degenerate gardens, then one may always find a homotopy between them which is a composition of isotopies of the garden on the graph and a composition of Moves I-IV.

\begin{figure}[h]
	\centering
	\includegraphics[width=\textwidth]{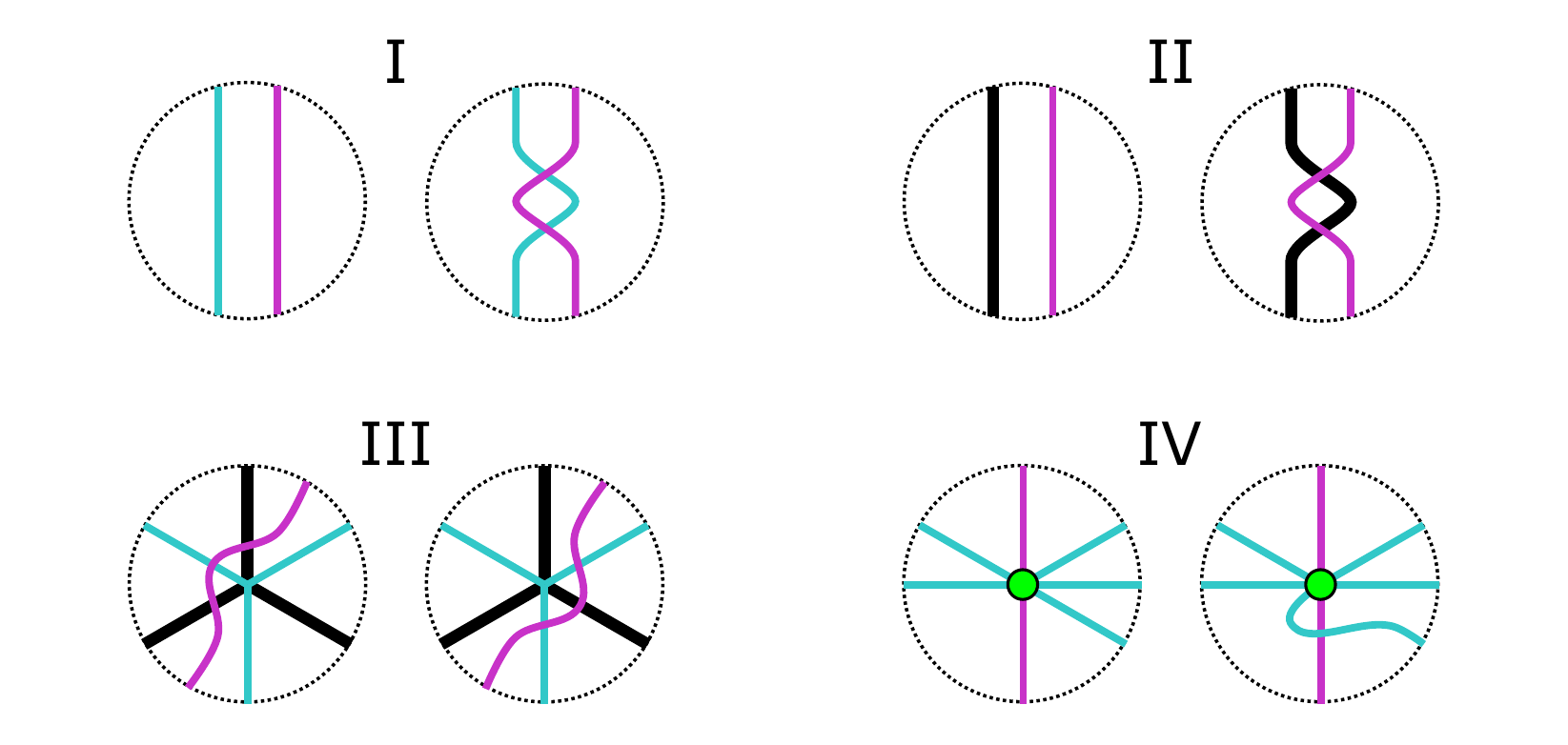}
	\caption{Moves I-IV. Implicit are all possible orientations. For example, for Move IV, we can also swing the tangent vector to a thread around in the clockwise direction as opposed to the counter-clockwise direction drawn.}
	\label{fig:Moves_I-IV}
\end{figure}

Move V is demonstrated in Figure \ref{fig:Move_V}. Move V interchanges centers along a generic curve connecting the two centers. A single application of Move V will change the homotopy type of the garden, though there are also non-trivial sequences of applications of Move V which preserve the homotopy type. These matters will be discussed in great detail later.

\begin{figure}[h]
	\centering
	\includegraphics[width=\textwidth]{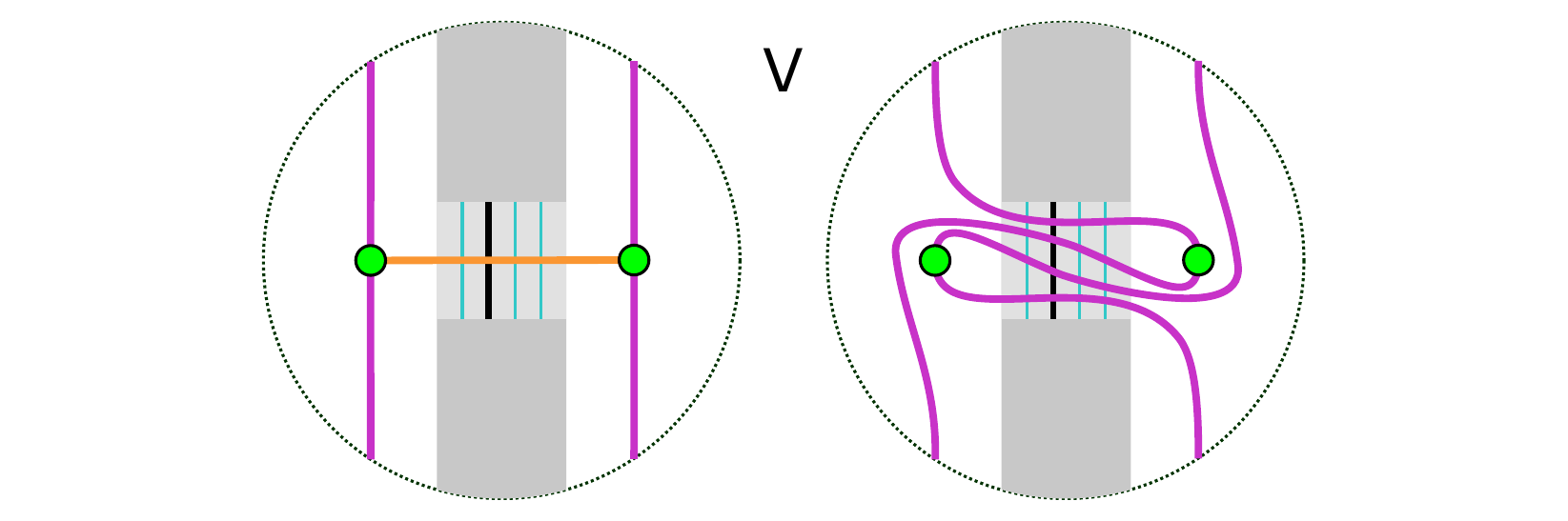}
	\caption{Move V. Suppose we have two adjacent tines, and consider a curve from one center to the next which intersects no tines, but may of course intersect edges or threads. This is depicted as an orange arc in the left diagram. Note that in the dark gray region (and outside of the local picture), the garden may be arbitrary. We have removed the threads emanating from the centers for simplicity: they may be moved arbitrarily by Move IV. Then we may take the tines and perform a positive half-twist of the two centers supported in a neighborhood of a connecting curve. In the diagram, the tines tunnel through the light gray rectangle so that they swap centers.}
	\label{fig:Move_V}
\end{figure}

Finally, Moves VI and VII are illustrated in Figure \ref{fig:Moves_VI-VII}. They involve passing a seed through a thread or an edge. Again, a single application of Move VI or VII will in general change the homotopy type of the garden, as such a homotopy restricts the base point to lie in the same region of $S^2$ bounded by the graph and the web (just a triangle). However, a nontrivial sequence of applications of Moves VI and VII may in general preserve a garden, as will also be discussed in great detail later.

\begin{figure}[h]
	\centering
	\includegraphics[width=\textwidth]{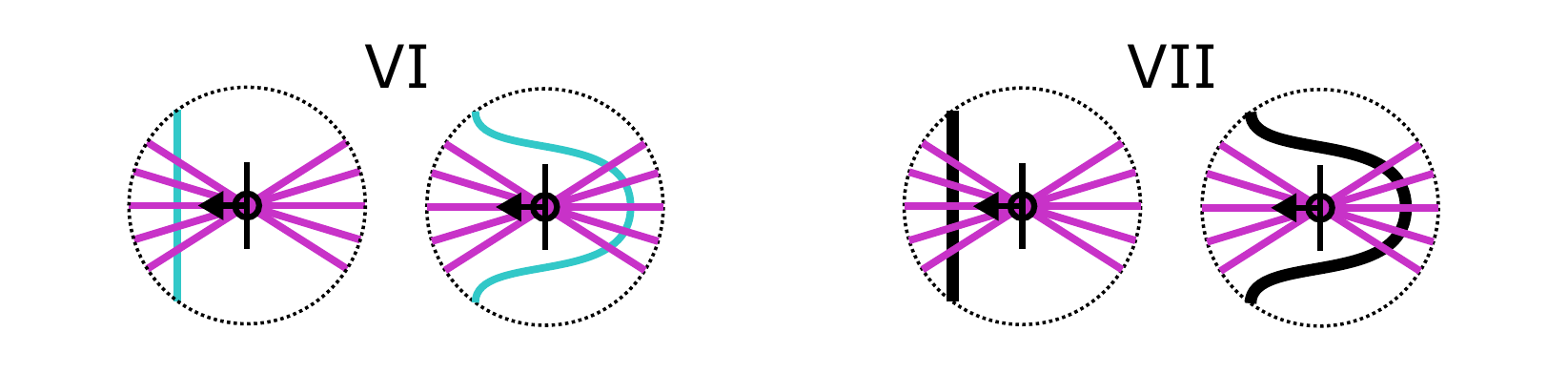}
	\caption{Moves VI and VII. We pass the seed of the garden through a thread or an edge, respectively. The base point of the seed is located at the center of each diagram, and the seed itself is represented by the small vertical line and arrow, representing the kernel of the seed $\xi$ and its coorientation. Hence, in the diagrams presented, all the tines are oriented from right to left.}
	\label{fig:Moves_VI-VII}
\end{figure}

Every non-degenerate path, has three types of critical behaviors, passing through either a center $c_f$, an edge of the graph, or the interior of a thread. Note that we are not interested in intersections of such paths with the tines, which do not appear in the definition of a non-degenerate path. Also notice that the tines of a non-degenerate garden are themselves non-degenerate paths. For each type of intersection point, we study a certain type of combinatorial interaction by way of binary sequences. The following definition appears in the commutative setting of Casals and Murphy \cite{CM_DGA} where it was originally defined only for tines. We make the same definition more generally for non-degenerate paths.

\begin{defn} \label{defn:binary_seq}
	Suppose $\gamma$ is an embedded path in $S^2$ which is non-degenerate with respect to a trivalent plane graph $G$ and garden $\Gamma$. A \textbf{binary sequence along $\gamma$} is a lower semicontinuous function $B \colon [a,b] \rightarrow \{0,1\}$ with the properties that:
	\begin{itemize}
		\item $B$ is locally constant away from the three types of intersection points
		\item at a center, $B$ is discontinuous, and must switch from $0$ to $1$
		\item at an edge, $B$ is discontinuous and must switch, either from $0$ to $1$ or from $1$ to $0$
		\item at the interior of a thread, $B$ is either locally constant or switches from $0$ to $1$
	\end{itemize}
	For $i,j \in \{0,1\}$, the set of binary sequences along $\gamma$ starting at $i$ and ending at $j$ is denoted by $\scr{BS}_{i,j}(\gamma)$.
\end{defn}

To each binary sequence $B$ along $\gamma$, we encode a monomial
$$w_B \in \wt{R}^{nc} * \Z\langle F \rangle = \Z[\Pi] * \Z\langle F \rangle,$$
with coefficient $\pm 1$, determined as follows. To determine $w_B$, we consider the set $D_B$ of discontinuities of the binary sequence, labelled in order that they appear on the tine. Each element of $D_B$ corresponds either to an edge, a thread, or a face. Note that $D_B$ does \emph{not} in general contain every thread the path $\gamma$ intersects, but only those where the binary sequence is discontinuous. To each element of $d \in D_B$, we determine a monomial in $\wt{R}^{nc}*\Z\langle F\rangle$ with coefficient $\pm 1$, described in the next paragraphs, and set $w_B$ to be the product of these monomials in order. Let us now describe the monomial associated to each discontinuity. As we proceed, the reader should verify that our conventions match with the example in Figure \ref{fig:Garden}, which shows a sample binary sequence along a tine, and which has the associated word described in the caption; recall again that the notion we are defining works for all non-degenerate paths and not just tines.

For an edge, the binary sequence always switches, either from $0$ to $1$, or vice versa. Associating the label $0$ with the label $-$ and the label $1$ with the label $+$, we have that the path crosses the edge in a manner which matches one of the four diagrams of Figure \ref{fig:Generator_Convention}, where our path $\gamma$ is taken in place of the tine. Accordingly, we associate to this discontinuity the corresponding generator $A_e$, $A_e^{-1}$, $-B_e$, or $-B_{e}^{-1}$, where we note that we have included a negative sign in two of these cases.

For a discontinuity with a thread $\tau_{f}(v)$, we associate with it a monomial
$$H(\tau_{f}(v)) \in \wt{R}^{nc} = \Z[\Pi]$$
defined as follows. Consider a counterclockwise loop around the vertex $v$ based at $c_f$ which goes from the label $-$ to the label $+$, and where the $B$ and $B^{-1}$ terms come with negative signs (as we did for edges). In other words, if we look at Figure \ref{fig:Vertex_Relation}, if the center $c_f$ is in the region with the red circle and the thread flows from the red circle to the vertex $v$, then we have $H(\tau_{f}(v)) = B_{e_1}^{-1}A_{e_2}B_{e_3}^{-1}$ by following a single loop counterclockwise around $v$ from $-$ to $+$, and where the sign is positive because there are two $B$ terms in this case. Notice that the sign of $H(\tau_{f}(v))$ is automatically $(-1)^{r_v}$, where $r_v$ is the number of outgoing edges. As for our construction of $w_B$, we include another sign, using $\pm H(\tau_f(v))$ for a discontinuity at $\tau_f(v)$, where the sign is determined by the following right-hand rule: the sign is $+1$ or $-1$ according to whether the orientation of the thread followed by the orientation of the path gives a frame matching the underlying orientation on the plane. The signs are indicated in Figure \ref{fig:Sign_Rule_Thread_Crossings}.

\begin{figure}[h]
	\centering
	\includegraphics[width=\textwidth]{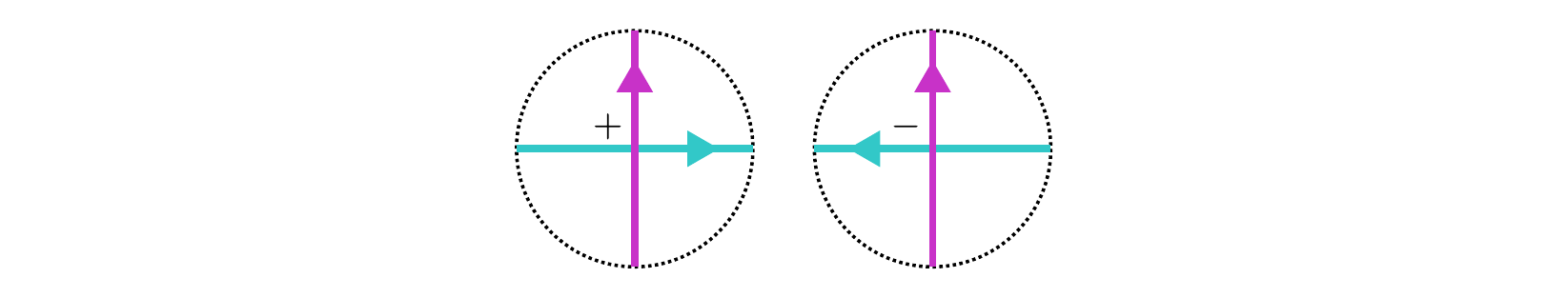}
	\caption{The sign placed in front of $H(\tau_f(v))$ for a binary sequence which changes from $0$ to $1$ at a thread. The horizontal lines are threads and the vertical lines represent our path $\gamma$ to which we are associating our monomial.}
	\label{fig:Sign_Rule_Thread_Crossings}
\end{figure}

\begin{rmk}
	In our convention for $H(\tau_f(v))$, we have used the geometric element associated to a counterclockwise loop around $v$. We could have alternatively used clockwise loops in our convention, which has the effect of simply negating $H(\tau_f(v))$, which would be equivalent to using the counterclockwise convention but for the opposite orientation on $G$.
\end{rmk}

Finally, if we have a discontinuity at a center $c_f$, we simply replace it with $f \in F$, considered as a monomial in $\wt{R}^{nc} * \Z \langle F\rangle$. This finishes the definition of $w_B$; we repeat, as we said at the beginning of this discussion, that an example is found in Figure \ref{fig:Garden} and its caption.

\begin{rmk}
	If $D_B = \emptyset$, then we set $w_B = 1$.
\end{rmk}

To each non-degenerate immersed path $\gamma$ and each $i,j \in \{0,1\}$, we may therefore define four associated words $w_{i,j} \in \wt{R}^{nc} = \Z[\Pi]$, given by
$$w_{i,j}(\gamma) := \sum_{B \in \scr{BS}_{i,j}(\gamma)} w_B.$$

\begin{rmk}
	If $\scr{BS}_{i,j}(\gamma) = \emptyset$, then $w_{i,j}(\gamma) = 0$.
\end{rmk}

We end this subsection by establishing two fundamental properties about the words $w_{i,j}(\gamma)$. The first is a statement about locality, showing that breaking up non-degenerate paths into smaller non-degenerate paths plays nicely with the associated word construction. The second involves invariance of the associated word under mild homotopies. Both of these are used in the commutative coefficient case by Casals and Murphy \cite{CM_DGA} within various proofs; here we single them out in separate lemmas.

\begin{lem} \label{lem:locality_of_word}
	Suppose $\gamma_0$ and $\gamma_1$ are non-degenerate paths such that the terminal point of $\gamma_0$ is the initial point of $\gamma_1$. Let $\gamma = \gamma_0 * \gamma_1$ be the composition of paths, smoothed along where they glue. Then
	$$w_{i,j}(\gamma) = \sum_{k = 0}^{1} w_{i,k}(\gamma_0) \cdot w_{k,j}(\gamma_1).$$
\end{lem}

\begin{proof}
	We have obvious bijections
	$$\scr{BS}_{i,j}(\gamma) = \left(\scr{BS}_{i,0}(\gamma_0) \times \scr{BS}_{0,j}(\gamma_1)\right) \sqcup \left(\scr{BS}_{i,1}(\gamma_0) \times \scr{BS}_{1,j}(\gamma_1)\right),$$
	and the two sides of the equation in the lemma are just given by writing out the corresponding words according to these two associations.
\end{proof}

With this lemma proved, it is convenient from now on, for a non-degenerate path $\gamma$, to write
$$W(\gamma) := \begin{pmatrix}w_{0,0}(\gamma) & w_{0,1}(\gamma) \\ w_{1,0}(\gamma) & w_{1,1}(\gamma)\end{pmatrix} \in \mathrm{Mat}_{2 \times 2}(\wt{R}^{\mathrm{nc}}*\Z\langle F \rangle).$$
We will call this the \textbf{matrix of words} for $\gamma$. Lemma \ref{lem:locality_of_word} implies that a composition of non-degenerate paths is obtained by matrix multiplication of the corresponding matrices of words, i.e. $W(\gamma_1 * \gamma_2) = W(\gamma_1) \cdot W(\gamma_2).$ Even more is true; Lemma \ref{lem:invariance_of_word} below implies that the matrix of words is an invariant of non-degenerate paths relative to homotopies (through possibly degenerate paths) not crossing the centering of the garden.

\begin{lem} \label{lem:invariance_of_word}
	Suppose $\gamma_0$ and $\gamma_1$ are two non-degenerate paths with the same endpoints which are homotopic via a homotopy $\gamma_t$ (of not necessarily non-degenerate paths) which never intersects the centering of the garden. Then the matrices of words are equal,
	$$W(\gamma_0) = W(\gamma_1).$$
	Additionally, if $\gamma_0$ and $\gamma_1$ are two non-degenerate paths differing by an application of Move IV, then again $W(\gamma_0) = W(\gamma_1)$.
\end{lem}
\begin{proof}
	Any such homotopy maybe generically perturbed so that it is given (up to isotopy of $S^2$) by a finite sequence of Moves I-III as in Figure \ref{fig:Moves_I-IV} and creation/removal of self-intersections, with the understanding that our paths take the place of the tines. (Indeed, this lemma will primarily be applied to tines later on.) Since self-intersections do not affect the matrix of words, it suffices to check invariance under any of Moves I-IV. By Lemma \ref{lem:locality_of_word}, this is just a local check. To be precise, let $\gamma$ and $\gamma'$ be the before and after for each curve. Then we may decompose
	$$\gamma = \eta * \gamma_{\ell} * \mu$$
	$$\gamma' = \eta * \gamma'_{\ell} * \mu,$$
	where $\gamma_{\ell}$ and $\gamma'_{\ell}$ represent the local pieces on the left and right of the corresponding figure. It suffices to prove that
	$$W(\gamma_{\ell}) = W(\gamma'_{\ell}),$$
	i.e. that $w_{i,j}(\gamma_{\ell}) = w_{i,j}(\gamma'_{\ell})$ for each $i,j \in \{0,1\}$, since then
	$$W(\gamma) = W(\eta)W(\gamma_{\ell})W(\mu) = W(\eta)W(\gamma'_{\ell})W(\mu) = W(\gamma').$$
	
	For Move I, notice on the right that because there are only thread crossings, the only local binary sequences for which there is a discontinuity are from $0$ to $1$, and hence
	\begin{empheq}[left=\empheqbiglbrace]{align*}
		w_{0,0}(\gamma_{\ell}) &= 1 = w_{0,0}(\gamma'_{\ell})\\
		w_{1,1}(\gamma_{\ell}) &= 1 = w_{1,1}(\gamma'_{\ell})\\
		w_{1,0}(\gamma_{\ell}) &= 0 = w_{1,0}(\gamma'_{\ell})
	\end{empheq}
	On the other hand, whereas $\scr{BS}_{0,1}(\gamma_{\ell}) = \emptyset$, we have $\scr{BS}_{0,1}(\gamma'_{\ell})$ consists of two elements, depending upon which thread crossing is chosen to give a discontinuity. However, those two binary sequences come with opposite thread orientations, and hence
	$$w_{0,1}(\gamma'_{\ell}) = \pm H(\tau) \mp H(\tau) = 0 = w_{0,1}(\gamma_{\ell}).$$
	
	For Move II, we have each $W(\gamma_{\ell}) = I$ is the identity matrix since there are no crossings in the left diagram. On the other hand, on the right, both crossings must be discontinuities, and must switch the label. Clearly $w_{1,0}(\gamma'_{\ell}) = w_{0,1}(\gamma'_{\ell}) = 1$ since there are no corresponding binary sequences. But also, it is easy to see that we have chosen our orientations so that also $w_{0,0}(\gamma'_{\ell}) = w_{1,1}(\gamma'_{\ell})=1$. For example, if both the edge and the tine are oriented upwards in the diagram, then $w_{1,1}(\gamma'_{\ell}) = (-B_e)(-B_e^{-1}) = 1$, where $e$ is the edge.
	
	For Move III, suppose without loss of generality that our paths are directed from bottom left to top right. Starting from the bottom left endpoint of the local path and working clockwise, suppose the edges and tines are labeled $e_1,\tau_1,e_2,\tau_2,e_3,\tau_3$. Suppose the orientations of all the edges are inwards towards the vertex; we leave the other orientations to the reader. (The orientations of the threads are always inwards pointing, so we need not specify them.) One checks that with our orientations, $H(\tau_i) = B_{e_i}^{-1}A_{e_{i-1}}B_{e_{i-2}}^{-1}$, where the indices are taken modulo $3$. The following relations are easily checked, since the corresponding local binary sequence sets are singletons.
	\begin{empheq}[left=\empheqbiglbrace]{align*}
		w_{0,0}(\gamma_{\ell}) = (A_{e_1}^{-1})(-B_{e_2}) = (-H(\tau_3))A_{e_3} = w_{0,0}(\gamma'_{\ell})\\
		w_{1,1}(\gamma_{\ell}) = (-B_{e_1})(A_{e_2}^{-1}) = A_{e_3}(-H(\tau_2)) = w_{1,1}(\gamma'_{\ell})\\
		w_{1,0}(\gamma_{\ell}) = (-B_{e_1})(H(\tau_1))(-B_{e_2}) = A_{e_3} = w_{1,0}(\gamma'_{\ell})
	\end{empheq}
	For the final case, we have $\scr{BS}_{0,1}(\gamma_{\ell}) = \emptyset$, so that $w_{0,1}(\gamma_{\ell}) = 0$. On the other hand, $\scr{BS}_{0,1}(\gamma'_{\ell})$ consists of two elements, one which crosses only at the edge, and the other which also crosses at both threads. We see
	$$w_{0,1}(\gamma'_{\ell}) = -B_{e_3}^{-1}+(-H(\tau_3))(A_{e_3})(-H(\tau_2)) = 0.$$
	
	For Move IV, the fact that the binary sequence must change from $0$ to $1$ at the center severely restricts the possible local binary sequences. In particular, on the right hand side, if the path is oriented upwards, then the binary sequence cannot switch at the thread because it would have to change from $0$ to $1$, but must be $0$ when it enters the face. Similar reasoning applies if the path is oriented downwards. Regardless, we find
	\begin{empheq}[left=\empheqbiglbrace]{align*}
		w_{0,0}(\gamma_{\ell}) = 0 = w_{0,0}(\gamma'_{\ell})\\
		w_{1,1}(\gamma_{\ell}) = 0 = w_{1,1}(\gamma'_{\ell})\\
		w_{1,0}(\gamma_{\ell}) = 0  = w_{1,0}(\gamma'_{\ell})\\
		w_{0,1}(\gamma_{\ell}) = f = w_{0,1}(\gamma'_{\ell})
	\end{empheq}
	where $f$ is the face corresponding to the center $c_f$ in the diagram through which the path passes.
\end{proof}

\subsection{The enlarged dg-algebra} \label{ssec:enlarged}

Define
$$\wt{\scr{B}}^+_{G} := \wt{R}^{\mathrm{nc}} * \Z\langle F,x,y,z,w\rangle,$$
the free associative graded algebra over the ring $\wt{R}^{\mathrm{nc}}$ generated by all faces $F$ (including any chosen face at infinity) of degree $1$ and four extra generators $x,y,z,w$ of degree $2$. This is of course a slightly different object than appears in Theorems \ref{thm:nc_CM} and \ref{thm:nc_CM_functoriality}, but we will be able to use Lemma \ref{lem:induced_morphisms} together with the transition maps of Section \ref{ssec:transition} to eventually prove our desired results. The tilde in the notation refers to the coefficient ring $\wt{R}^{\mathrm{nc}}$ (as opposed to $R_T$), and the plus in the notation refers to the inclusion of the face at infinity and the generator $w$ which do not appear in the original formulation of Casals and Murphy in the commutative coefficient setting \cite{CM_DGA}. We will discuss removing these enlargements in Section \ref{sec:remove_enlargements}.

For a given garden $\Gamma$ on $G$, we define a linear degree $-1$ endomorphism $\wt{\partial}^{\scr{B}^+}_{G,\Gamma}$ by specifying its action on generators and extending via the Leibniz rule
$$\wt{\partial}(a \cdot b) = (\wt{\partial}a) \cdot b + (-1)^{|a|}a \cdot (\wt{\partial} b)$$
for homogeneous elements $a,b \in \wt{\scr{B}}^{+}_{G}$. To simplify notation, let
$$X = \begin{pmatrix} z & w \\ y & x \end{pmatrix}$$
be the matrix consisting of the four non-face generators of $\wt{\scr{B}}^+_G$. Then $\wt{\partial}^{\scr{B}^+}_{G,\Gamma}$ acts on the generators by 
\begin{empheq}{align*}
	\wt{\partial}^{\scr{B}^+}_{G,\Gamma}f &= \sum_{v \in f} H(\tau_f(v)), \quad \forall~f \in F\\
	\wt{\partial}^{\scr{B}^+}_{G,\Gamma}X &= \sum_{f \in F}W(\gamma_f)
\end{empheq}

\begin{rmk}
	A complete example of an enlarged dg-algebra for a specific pair $(G,\Gamma)$ is included in Appendix \ref{appx:example}.
\end{rmk}

It is not immediately clear that $\wt{\partial}^{\scr{B}^+}_{G,\Gamma}$ is actually a differential. Since we have defined it so that it is of degree $-1$ and is extended via the Leibniz rule, it suffices to check that $\wt{\partial}^{\scr{B}^+}_{G,\Gamma}$ squares to zero. We begin with a lemma which allows us to extend Lemma \ref{lem:invariance_of_word} to the setting in which a homotopy of non-degenerate paths passes through a center.

\begin{lem}\label{lem:htpy_through_center}
	Suppose $\gamma_0$ and $\gamma_1$ are two non-degenerate paths with the same endpoints which are homotopic via a homotopy $\gamma_t$ (of not necessarily non-degenerate paths) such that there is a single such path $\gamma_{t_0}$ (with $t_0 \neq 0,1$) passing through a single center so that frame given by the vector $\frac{d}{dt}|_{t=t_0}\gamma_t$ followed by the tangent vector to $\gamma_{t_0}$ forms a positively oriented frame at this center. Then for any garden $\Gamma$ on $G$,
	$$W(\gamma_1) - W(\gamma_0) = \wt{\partial}^{\scr{B}^+}_{G,\Gamma}W(\gamma_{t_0}).$$
\end{lem}

\begin{proof}
	Let us write
	$$\gamma_{t_0} := \gamma^s * \gamma^c * \gamma^t,$$
	where $\gamma^c$ represents a small portion of $\gamma^{t_0}$ which crosses through the center and none of the edges or threads, and $\gamma^s$ and $\gamma^t$ represent the starting and terminal portions of the curve. By homotoping the homotopy $\gamma_t$, we may assume without loss of generality that for some small $\epsilon > 0$, we may write
	$$\gamma_{t_0-\epsilon} = \gamma^s * \gamma^{c,L} * \gamma^t \qquad \mathrm{and} \qquad \gamma_{t_0+\epsilon} = \gamma^s * \gamma^{c,R} * \gamma^t,$$
	where $\gamma^{c,L}$ and $\gamma^{c,R}$ are small arcs with the same endpoints at $\gamma^c$ but passing on opposite sides of the corresponding center. We notice that the curves $\gamma^{c,L}$ and $\gamma^{c,R}$ have corresponding matrices of words of the form
	$$W(\gamma^{c,L}) = \begin{pmatrix}1 & w_{0,1}(\gamma^{c,L}) \\ 0 & 1\end{pmatrix}, \qquad \qquad W(\gamma^{c,R}) = \begin{pmatrix}1 & w_{0,1}(\gamma^{c,R}) \\ 0 & 1\end{pmatrix}$$
	since the corresponding arcs only interact with the threads of the center. In fact, each thread $\tau$ around this center contributes a unique local binary sequence from $0$ to $1$ for either $\gamma^{c,L}$ or $\gamma^{c,R}$, and the corresponding monomial $H(\tau)$ comes with a minus sign for $\gamma^{c,L}$ and a plus sign for $\gamma^{c,R}$. That is, if the center is $c_f$, then
	$$w_{0,1}(\gamma_k^{c,R}) - w_{0,1}(\gamma_k^{c,L}) = \sum_{v \in f}H(\tau_{f}(v)) = \wt{\partial}^{\scr{B}^+}_{G,\Gamma}f.$$
	Using Lemma \ref{lem:invariance_of_word}, we have $W(\gamma_0) = W(\gamma_{t_0-\epsilon})$ and $W(\gamma_1) = W(\gamma_{t_0+\epsilon})$. Therefore, from our computation, and from the fact that $\gamma^s$ and $\gamma^t$ pass through no centers so that
	$$\wt{\partial}^{\scr{B}^+}_{G,\Gamma}W(\gamma^s) = 0 = \wt{\partial}^{\scr{B}^+}_{G,\Gamma}W(\gamma^t),$$
	we find
	\begin{align*}
		W(\gamma_1) - W(\gamma_0) & =W(\gamma_{t_0+\epsilon}) - W(\gamma_{t_0-\epsilon}) \\
		&= W(\gamma^s)W(\gamma^{c,R})W(\gamma^t) - W(\gamma^s)W(\gamma^{c,L})W(\gamma^t) \\
		&= W(\gamma^s)\left[W(\gamma^{c,R})-W(\gamma^{c,L})\right]W(\gamma^t) \\
		&= W(\gamma^s)\begin{pmatrix}0 & \wt{\partial}^{\scr{B}^+}_{G,\Gamma} f\\ 0 & 0\end{pmatrix}W(\gamma^t) \\
		&= W(\gamma^s) \left[\wt{\partial}^{\scr{B}^+}_{G,\Gamma}\begin{pmatrix}0 & f \\ 0 & 0 \end{pmatrix}\right]W(\gamma^t)\\
		&=W(\gamma^s)\left[\wt{\partial}^{\scr{B}^+}_{G,\Gamma}W(\gamma^c)\right]W(\gamma^t) \\
		&=\wt{\partial}^{\scr{B}^+}_{G,\Gamma}\left[W(\gamma^s)W(\gamma^c)W(\gamma^t)\right]\\
		&= \wt{\partial}^{\scr{B}^+}_{G,\Gamma}W(\gamma_{t_0}).
	\end{align*}
\end{proof}

We may now prove the differential property.

\begin{prop} \label{prop:differential} For any non-degenerate garden $\Gamma$ for a trivalent plane graph $G$, the linear maps $\wt{\partial}^{\scr{B}^+}_{G,\Gamma}$ are differentials, i.e. $$\left(\wt{\partial}^{\scr{B}^+}_{G,\Gamma}\right)^2 = 0.$$ Hence, $(\wt{\scr{B}}^+_G,\wt{\partial}^{\scr{B}^+}_{G,\Gamma})$ is a dg-algebra with respect to any choice of non-degenerate garden.
\end{prop}

\begin{proof}
	It suffices to check that $\left(\wt{\partial}^{\scr{B}^+}_{G,\Gamma}\right)^2 = 0$ on generators. For degree reasons, this holds on faces, and it suffices to check the relation on the generators $x$, $y$, $z$, and $w$, which we conveniently group as the matrix $X$.
	
	For the given garden $\Gamma$, let $f_1,\ldots,f_{g+3}$ be the enumeration of the faces given by enumerating the tines in clockwise order as they emanate from the seed. For example, in Figure \ref{fig:Garden}, the face at infinity is given as $f_2$; alternatively, in any of the diagrams in Figure \ref{fig:Moves_VI-VII}, the tines are ordered from bottom to top, both as they leave the seed, and as they exit the seed. Choose non-degenerate loops $\eta_0,\eta_1,\ldots,\eta_{g+3}$ based at the seed such that $\gamma_{f_k}$ lies between $\eta_{k-1}$ and $\eta_{k}$. By Lemma \ref{lem:htpy_through_center}, we find
	$$W(\eta_k) - W(\eta_{k-1}) =\wt{\partial}^{\scr{B}^+}_{G,\Gamma}W(\gamma_{f_k}).$$
	Summing all of these up, we find
	$$W(\eta_{g+3}) - W(\eta_0) = \wt{\partial}^{\scr{B}^+}_{G,\Gamma}\sum_{k=1}^{g+3} W(\gamma_{f_k}) = \left(\wt{\partial}^{\scr{B}^+}_{G,\Gamma}\right)^2X.$$
	But notice that both $\eta_0$ and $\eta_{g+3}$ may be contracted to a point, so that
	$$W(\eta_0) = I = W(\eta_{g+3})$$
	is just the identity matrix. Hence indeed $\left(\wt{\partial}^{\scr{B}^+}_{G,\Gamma}\right)^2X = 0$.
\end{proof}

\begin{rmk}
	Our proof of Proposition \ref{prop:differential} is simpler than the one presented by Casals and Murphy \cite{CM_DGA}. They build an involution of pairs consisting of binary sequences and threads, each of which contributes a term to $(\wt{\partial}^{\scr{A}}_{G,\Gamma})^2$ (recall they do not include the enlargement in which we include two extra generators, so we have not included a $+$ in the notation), and check that the signs of these terms are opposite of each other. One could reason in precisely the same way in the non-commutative setting with no complication. Their argument is closer to the contact-geometric reason why the differential should square to zero: cancelling pairs identified via this involution are rigid gradient flow trees at the boundaries of moduli spaces of (non-rigid) gradient flow trees of dimension $1$.
\end{rmk}

\section{Functoriality for the enlarged dg-algebra} \label{sec:functoriality_enlarged}

At this point, we have defined an enlarged version $(\scr{B}^+_{G},\wt{\partial}_{G,\Gamma}^{\scr{B}^+})$, of the non-commutative Casals--Murphy dg-algebra. Some of the functoriality properties of Theorem \ref{thm:nc_CM_functoriality} will not appear at the enlarged level, since there is no choice of tree, and since we are not restricting to finite-type gardens. However, the composition property, changing the orientation, and changing the garden do make perfect sense at the enlarged level. The goal in this section is to prove the following theorem.

\begin{thm} \label{thm:nc_CM_enlarged}
	For any choice of garden $\Gamma$ on a trivalent plane graph $G$ of genus $g$, there exists a differential $\wt{\partial}_{G,\Gamma}^{\scr{B}^+}$ on the graded algebra $\wt{\scr{B}}^+_G$, making it into a dg-algebra. For any two gardens which are homotopic, the differentials are canonically identified.
	
	There exists a group action of $\scr{H}(G) := (\Z_2)^E \times F_{g+2} \times (\Z_2 * \Z_3)$ on the space of gardens up to homotopy. For each $\zeta \in \scr{H}(G)$ and homotopy class of garden $\Gamma$ on $G$, with $\Gamma' = \zeta \cdot \Gamma$, there is a dg-isomorphism
	$$\Phi^{\Gamma'}_{\Gamma}(\zeta) \colon (\wt{\scr{B}}^+_G,\wt{\partial}_{G,\Gamma}^{\scr{B}^+}) \rightarrow (\wt{\scr{B}}^+_G,\wt{\partial}_{G,\zeta \cdot \Gamma}^{\scr{B}^+})$$
	satisfying the following properties:
	\begin{itemize}
		\item \textbf{Composition property:}The dg-isomorphisms compose naturally, in the sense that
		$$\Phi_{\Gamma'}^{\Gamma''}(\zeta) \circ \Phi_{\Gamma}^{\Gamma'}(\theta) = \Phi_{\Gamma}^{\Gamma''}(\zeta\theta).$$
		\item \textbf{Orientation changes:} The action of the factor $(\Z_2)^E$ is induced by ring automorphisms of $\Z[\Pi]$.
		\item \textbf{Modifying the garden:} For any $\zeta \in F_{g+2} \times (\Z_2 * \Z_3)$ (i.e. trivial in the $(\Z_2)^E$ factor), the map $\Phi_{\Gamma}^{\Gamma'}(\zeta)$ is a regenerative tame dg-isomorphism. Furthermore, if $\zeta$ stabilizes $\Gamma$, i.e. $\zeta \cdot \Gamma = \Gamma$, then $\Phi_{\Gamma}^{\Gamma}(\zeta)$ is homotopic to the identity, and hence acts by the identity in homology.
	\end{itemize}	 
\end{thm}

We have already verified that $\wt{\partial}^{\scr{B}^+}_{G,\Gamma}$ is a differential which is unaffected by Moves I-IV, and hence, our dg-isomorphisms will be constructed for orientation changes along with each of Moves V, VI, and VII. We proceed in the following order, subsection by subsection:
\begin{itemize}
	\item \textbf{Changing Orientations:} We construct a dg-isomorphism every time we change the orientation of an edge simply by negating the corresponding generators in the coefficient ring.
	\item \textbf{Invariance under Move V:} We construct a tame dg-isomorphism for each application of Move V. Further, we show that Move V yields a natural free transitive braid group action $\mathrm{Br}_{g+3}$ on the homotopy classes of the space of gardens with a fixed seed. We have a surjection $F_{g+2} \twoheadrightarrow \mathrm{Br}_{g+3}$, the kernel of which acts by dg-automorphisms homotopic to the identity.
	\item \textbf{Exactness of the full twist:} The full twist in the braid group preserves the homotopy class of the garden, as it may be realized by just twisting the seed by 360 degrees. We prove that the full twist acts by a dg-automorphism homotopic to the identity.
	\item \textbf{Invariance under Moves VI and VII:} We construct a dg-isomorphism for each application of Moves VI and VII. We show these commute with Move V, and further show that Moves VI and VII yield a $\Z_2 * \Z_3$ action on the homotopy classes of gardens, and understand the corresponding regenerative tame dg-isomorphisms.
	\item \textbf{Proof of Theorem \ref{thm:nc_CM_enlarged}:} In this short subsection, we simply summarize the results we have obtained, which proves the desired theorem.
\end{itemize}

\subsection{Changing orientations} \label{ssec:orientation_change}

We have already seen as a consequence of Lemma \ref{lem:invariance_of_word} that Moves I-IV preserve the differential on the nose. In other words, if two gardens $\Gamma$ and $\Gamma'$ differ by a composition of isotopies and applications of Moves I-IV, then $\wt{\partial}_{\Gamma} = \wt{\partial}_{\Gamma'}$, and hence we may take $\Phi_{\Gamma}^{\Gamma'} = \id$. We turn our attention now to the remaining ways in which we may change a garden, by changing the orientation of an edge, or by Moves V-VII. In this subsection, we focus on changing orientations.

There is an obvious $(\Z_2)^E$ action on the homotopy classes of gardens given by swapping orientations at corresponding edges. Let us consider the simplest case, in which we swap a single edge, $e_0 \in E$ is an edge. Let us define a ring automorphism $\Phi_{e_0} \colon \Z[\Pi] \rightarrow \Z[\Pi]$ on generators as follows. Recall that $\Pi$ is generated by elements of the form $\lambda_{e,\mathfrak{o}}$, where $(e,\mathfrak{o}) \in \wt{E}$. Define $\Phi_{e_0}$ on these generators by
$$\Phi_{e_0}(\lambda_{e,\mathfrak{o}}) = \left\{\begin{matrix}\lambda_{e,\mathfrak{o}}, & e \neq e_0 \\ -\lambda_{e,\mathfrak{o}}, & e = e_0\end{matrix}\right.$$
In turn, this induces an graded algebra automorphism of $\wt{\scr{B}}^+_G$, acting by the identity on generators. Furthermore, it is obvious that $(\Phi_{e_0})^2 = \id$, and that $\Phi_{e_0}$ and $\Phi_{e_1}$ commute for any $e_0,e_1 \in E$. Hence, we also have a homomorphism
$$(\Z_2)^E \rightarrow \mathrm{Aut}(\wt{\scr{B}}^+_G).$$
For $\iota \in (\Z_2)^E$ (the character $\iota$ is used because we have an \emph{involution} on the orientations of edges), we write $\Phi_{\iota}$ for the induced automorphism of $\wt{\scr{B}}^+_G$.

\begin{lem} \label{lem:orientation_change}
	For any homotopy class of garden $\Gamma$ and $\iota \in (\Z_2)^E$, the graded algebra automorphism $\Phi_{\iota} \colon \wt{\scr{B}}^+_G \rightarrow \wt{\scr{B}}^+_G$ intertwines the differentials $\wt{\partial}^{\scr{B}^+}_{G,\Gamma}$ and $\wt{\partial}^{\scr{B}^+}_{G,\iota \cdot \Gamma}$, yielding a dg-isomorphism
	$$\Phi_{\iota} \colon (\wt{\scr{B}}^+_G,\wt{\partial}^{\scr{B}^+}_{G,\Gamma}) \rightarrow (\wt{\scr{B}}^+_G,\wt{\partial}^{\scr{B}^+}_{G,\iota \cdot \Gamma}).$$
\end{lem}

\begin{proof}
	It suffices to check the equation
	$$\Phi_{\iota} \circ \wt{\partial}^{\scr{B}^+}_{G,\Gamma} = \wt{\partial}^{\scr{B}^+}_{G,\iota \cdot \Gamma} \circ \Phi_{\iota}$$
	on generators. If we apply both sides to a face $f$, then this reads
	$$\sum_f \Phi_{\iota}(H^{\Gamma}(\tau_f(v))) = \sum_f H^{\iota \cdot \Gamma}(\tau_f(v)),$$
	where we have augmented the notation $H(\tau_f(v))$ to point out that it depends upon a choice of orientation of the edges. But this relation is clear term by term for each $f$ by definition of $H(\tau_f(v))$, i.e.
	$$\Phi_{\iota}(H^{\Gamma}(\tau_f(v))) = H^{\iota \cdot \Gamma}(\tau_f(v)),$$
	since both sides are $\pm H^{\Gamma}(\tau_f(v))$ depending upon the parity of the number of edges incident to $v$ which change orientation. Similarly, applying both sides to one of the generators $x$, $y$, $z$, or $w$, it suffices to check that
	$$\Phi_{\iota}(W^{\Gamma}(\gamma_f)) = W^{\iota \cdot \Gamma}(\gamma_f)$$
	for each tine $\gamma_f$. In turn, this resorts to checking that for each binary sequence $B$ along $\gamma_f$, we have $\Phi_{\iota}(w^{\Gamma}_B) = w^{\iota \cdot \Gamma}_B$. Each $w_B$ is a product of monomials for each discontinuity of $B$, and it suffices to check for each such discontinuity. For threads, we have already verified the relation. For centers, there is no sign. For an edge, the corresponding term $\pm \lambda_{e,\mathfrak{o}}^{\pm 1}$ in $w^{\Gamma}_B$ and $w^{\iota \cdot \Gamma}_{B}$ are the same if $\iota$ leaves the orientation of the edge unfixed and negatives of each other if $\iota$ reverses the orientation of the edge. But this negation is the same as the action of $\Phi_{\iota}$.
\end{proof}

\subsection{Invariance under Move V} \label{ssec:Move_V_invariance}

We turn now to understanding Move V. We refer to Figure \ref{fig:Move_V} again, with the convention that the tines are oriented upwards, $c_f$ is the center on the left, and $c_g$ is the center on the right. Let and $\gamma_f$ and $\gamma_g$ be the tines in the left figure, and $\gamma_f'$ and $\gamma_g'$ the tines in the right figure. Recall in the proof of Proposition \ref{prop:differential} that we wrote each tine as a composition of starting, central, and terminal portions, e.g.
$$\gamma_f = \gamma_f^s * \gamma_f^c * \gamma_f^t.$$
Both $\gamma_f^s$ and $(\gamma')_f^s$ have the same start point (the seed) and endpoint (just under $c_g$), and we see that they are in fact homotopic with fixed endpoints without passing through any centers, since Move V is applied to neighboring tines. Hence, by Lemma \ref{lem:invariance_of_word}
$$W(\gamma_f^s) = W((\gamma')_f^s).$$
On the other hand, $\gamma_f^t$ and $(\gamma')_f^t$ have the same start point (just above $c_f$) and endpoint (the seed), and we see that there is a natural homotopy between them which passes through the center $c_g$. We saw in Lemma \ref{lem:htpy_through_center} how a homotopy passing through a center affects the corresponding words. In particular, if $\eta$ is the orange curve in the figure, oriented from $c_f$ to $c_g$, then
$$W((\gamma')_f^t)-W(\gamma_f^t) = W(\eta)\begin{pmatrix}0 & \wt{\partial}^{\scr{B}^+}_{G,\Gamma}g \\ 0 & 0 \end{pmatrix}W(\gamma_g^t).$$
By similar reasoning,
$$W(\gamma_g^t) = W((\gamma')_g^t)$$
and
$$W(\gamma_g^s)-W((\gamma')_g^t) = W(\gamma_f^s)\begin{pmatrix}0 & \wt{\partial}^{\scr{B}^+}_{G,\Gamma}f \\ 0 & 0 \end{pmatrix}W(\eta).$$
Now if we compute the differentials of $X$ with respect to the two gardens $\Gamma$ and $\Gamma'$, we see the contributions only differ along these two tines. Therefore,
$$\left(\wt{\partial}^{\scr{B}^+}_{G,\Gamma'}-\wt{\partial}^{\scr{B}^+}_{G,\Gamma}\right)X = W((\gamma'_f))+W((\gamma')_g) - W(\gamma_f) - W(\gamma_g).$$
But we have by our relations above that
\begin{align*}
	W((\gamma')_f) - W(\gamma_f) &= W((\gamma')_f^s)\begin{pmatrix}0 & f \\ 0 & 0 \end{pmatrix} W((\gamma')_f^t) - W(\gamma_f^s)\begin{pmatrix}0 & f \\ 0 & 0 \end{pmatrix} W(\gamma_f^t) \\
		&= W(\gamma_f^s)\begin{pmatrix}0 & f \\ 0 & 0 \end{pmatrix} W(\eta)\begin{pmatrix}0 & \wt{\partial}^{\scr{B}^+}_{G,\Gamma}g \\ 0 & 0 \end{pmatrix}W(\gamma_g^t)
\end{align*}
Similarly,
$$W((\gamma')_g) - W(\gamma_g) = -W(\gamma_f^s)\begin{pmatrix}0 & \wt{\partial}^{\scr{B}^+}_{G,\Gamma}f \\ 0 & 0 \end{pmatrix} W(\eta)\begin{pmatrix}0 & g \\ 0 & 0 \end{pmatrix}W(\gamma_g^t).$$
Adding these together, we find that
$$\left(\wt{\partial}^{\scr{B}^+}_{G,\Gamma'}-\wt{\partial}^{\scr{B}^+}_{G,\Gamma}\right)X = -\wt{\partial}^{\scr{B}^+}W(\gamma^{\Gamma}_{f,g}),$$
where the differential on the right hand side is given no subscript since it may be taken with respect to either $\Gamma$ or $\Gamma'$, and where $\gamma^{\Gamma}_{f,g}$ is the path given by following $\gamma^s_f$, then passing through $c_f$, then following $\eta$, then passing through $c_g$, and finally following $\gamma^t_g$, so that
$$W(\gamma^{\Gamma}_{f,g}) = W(\gamma_s^f)\begin{pmatrix} 0 & f \\ 0 & 0 \end{pmatrix}W(\eta) \begin{pmatrix} 0 & g \\ 0 & 0\end{pmatrix}W(\gamma^t_g).$$
This path is actually well defined up to homotopy, assuming the garden has been fixed up to homotopy. Suppose $\Gamma$ is a garden, and suppose $\gamma_f$ is tine adjacent and to the left of $\gamma_g$ (as we have been assuming). Then $\gamma^{\Gamma}_{f,g}$ may be defined as the (unique homotopy class of) path from the base point to itself, with initial and tangent vectors positive with respect to the seed, which stays between $\gamma_f$ and $\gamma_g$, and such that it passes through the centers $c_f$ and $c_g$ once in that order.

Our discussion has proved the following proposition.

\begin{prop} \label{prop:Move_V_action}
	Suppose $\Gamma$ and $\Gamma'$ are gardens which differ from each up to homotopies fixing the seed by an application of Move V as in Figure \ref{fig:Move_V} applied to neighboring tines $\gamma_f$ and $\gamma_g$ in $\Gamma$. Let $\gamma^{\Gamma}_{f,g}$ be the path described above. Then the graded algebra homomorphism $\Phi_{\Gamma}^{\Gamma'} \colon \wt{\scr{B}}^+ \rightarrow \wt{\scr{B}}^+$ defined on generators by
	\begin{align*}
		\Phi_{\Gamma}^{\Gamma'}(f) &= f\\
		\Phi_{\Gamma}^{\Gamma'}(X) &= X + W(\gamma^{\Gamma}_{f,g})
	\end{align*}
	is a tame dg-isomorphism with respect to $\wt{\partial}^{\scr{B}^+}_{G,\Gamma}$ and $\wt{\partial}^{\scr{B}^+}_{G,\Gamma'}$.
\end{prop}

\begin{proof}
	This is just a summary of our discussion. Tameness is clear, as this dg-isomorphism is just a composition of four elementary dg-automorphisms, one for each of the generators $x$, $y$, $z$, and $w$ which have been collected in the matrix $X$.
\end{proof}

\begin{rmk}
	Notice that Move V comes with a handed-ness, in which we perform a clockwise half-twist on the centers. We notice that after Move V, the new tine $\gamma'_g$ is now to the left of $\gamma'_f$, and one might be tempted to say that the inverse of Move V is given by taking $X \mapsto X + W(\gamma^{\Gamma'}_{g,f})$. But this is not correct, since this second application of Move V actually yields a different garden, up to homotopy, given by a full twist of the two centers. Instead, we artificially define the dg-isomorphism for the counter-clockwise twist to just be the inverse map. In particular, it has the slightly funky form as follows: suppose $\Gamma'$ is given from $\Gamma$ by an inverse application of Move V, with a counter-clockwise half twist. If $f$ is to the left of $g$ in $\Gamma$, then $g$ is to the left of $f$ in $\Gamma'$, and our dg-isomorphism takes $X \mapsto X-W(\gamma^{\Gamma'}_{f,g})$. It is worth noting that $\gamma^{\Gamma'}_{f,g}$ may be defined with respect to $\Gamma$ as well: it is the (unique homotopy class of) path from the base point to itself, positive with respect to the seed, which remains between the two tines $\gamma'_g$ and $\gamma'_f$, and crosses the faces $c_f$ and $c_g$ once, in that order. Notice here that it crosses $c_f$ first, but now $c_f$ is to the right of $c_g$.
\end{rmk}

Proposition \ref{prop:Move_V_action} demonstrates that Move V preserves the tame dg-isomorphism type of our dg-algebra. But actually, there is more structure. Suppose we have trivalent plane graph $G$. Notice that up to isotopy, the centering and web of any two gardens are the same, and hence we may assume this data is fixed. We will also fix an underlying orientation of the edges of $G$, as we have described the effect of changing this orientation in Section \ref{ssec:orientation_change}. Let $\scr{G}(G)$ denote the space of (not necessarily non-degenerate) gardens for this fixed centering, web, and orientation, all of which we omit from the notation. If $\xi \in S^*S^2$ is a possible seed, meaning its base point $s$ is not on the vertices, edges, centering, or web, then denote by $\scr{G}(G;\xi)$ the space of gardens with this fixed centering, web, and orientation, but with $\xi$ as its seed. We may realize this as the fiber of a Serre fibration.

Let $\scr{S}_n(\xi)$ denote the space of $n$-tuples of embedded loops on $S^2$ starting and ending at $\xi$ with initial and terminal tangent vectors positive with respect to $\xi$, such that the strands are disjoint except at the base point $s$, and with a chosen point on each loop other than $s$. We note that $\scr{S}_1(\xi)$ is contractible with respect to the $C^{\infty}$ topology, c.f. work of Eliashberg and Polterovich \cite[Proposition 1.4.A]{EP} for the contractibility of a slightly different space, from which contractibility of our $\scr{S}_1(\xi)$ easily follows. (Their space is a little more rigid than ours near the seed, and does not include the choice of point.) Furthermore, for each element of $\scr{S}_{n}(\xi)$, because $S^2$ is oriented, the loops are ordered, and we note that the map $\scr{S}_{n+1}(\xi) \rightarrow \scr{S}_n(\xi)$ given by forgetting the first loop is a Serre fibration with fiber given by a space equivalent to $\scr{S}_1(\xi)$. Hence, inductively, each $\scr{S}_n(\xi)$ is contractible. Additionally, for any such element of $\scr{S}_n(\xi)$, we may extract from it an (unordered) $n$-tuple of points, one for each loop. This is a Serre fibration
$$\pi \colon \scr{S}_n(\xi) \rightarrow \mathrm{UConf}_n(S^2 \setminus \{s\})$$
to the $n$th unordered configuration space. When $n = g+3$, the space $\scr{G}(G;\xi)$ is just the fiber over the centering. Hence, from the long exact sequence of homotopy groups, and since $\scr{S}_n(\xi)$ is contractible, we find that the connected components of $\scr{G}(G;\xi)$ are in bijection with the braid group on $g+3$ elements,
$$\pi_0(\scr{G}(G;\xi)) = \pi_1(\mathrm{UConf}_{g+3}(S^2 \setminus \{s\})) \cong \mathrm{Br}_{g+3}.$$
As Moves I-IV leave the differential invariant, we see that the dg-algebra is constant on each connected component of $\scr{G}(G;\xi)$. On the other hand, we see that Move V consists of taking two neighboring centers and twisting them around each other: this is precisely a braid group element. That is, $\pi_0(\scr{G}(G;\xi))$ naturally forms a $\mathrm{Br}_{g+3}$-torsor.

To be more explicit, recall that as a group, the braid group on $n$ strands has the natural presentation
$$\mathrm{Br}_n = \left\langle \sigma_1,\ldots,\sigma_{n-1} \mid \begin{matrix}\sigma_i\sigma_{i+1}\sigma_i = \sigma_{i+1}\sigma_i\sigma_{i+1}, & 1\leq i \leq n-2 \\ \sigma_i\sigma_j = \sigma_j\sigma_i, & |i-j| \geq 2\end{matrix} \right\rangle.$$
Suppose we have fixed a garden with seed $\xi$. Then the tines may be ordered $f_1,\ldots,f_{g+3}$ based upon their initial tangent vectors at the base point. The element $\sigma_i$ in $\mathrm{Br}_{g+3}$ then acts via Move V applied to the tines $\gamma_{f_i}$ and $\gamma_{f_{i+1}}$. In Proposition \ref{prop:Move_V_action}, we have already discussed how Move V acts by dg-isomorphism on the dg-algebra. In fact, these dg-isomorphisms are coherent, in that the braid relations are satisfied up to exact tame dg-isomorphims, as we now verify.

Let $F_{g+2}$ be the free group on $g+2$ generators, denoted $\wt{\sigma}_1,\ldots,\wt{\sigma}_{g+2}$. Let
$$\pi \colon F_{g+2} \twoheadrightarrow \mathrm{Br}_{g+3}$$
be the natural projection, with $\pi(\wt{\sigma}_k) = \sigma_k$. Then the action of $\mathrm{Br}_{g+3}$ on $\pi_0(\scr{G}(G;\xi))$ yields in turn an action of $F_{g+2}$. Under this action, we have the following lemma.

\begin{lem} \label{lem:Move_V_invariance}
	Suppose $\Gamma \in \pi_0(\scr{G}(G;\xi))$ and $\zeta \in F_{g+2}$. Then there exist tame dg-isomorphisms
	$$\Phi_{\Gamma}^{\zeta} \colon (\wt{\scr{B}}_G^+,\wt{\partial}^{\scr{B}^+}_{G,\Gamma}) \rightarrow (\wt{\scr{B}}_G^+, \wt{\partial}^{\scr{B}^+}_{G,\zeta \cdot \Gamma})$$
	such that
	$$\Phi_{\Gamma}^{1} = \id$$
	and
	$$\Phi_{\Gamma}^{\zeta\theta} = \Phi_{\theta \cdot \Gamma}^{\zeta} \circ \Phi_{\Gamma}^{\theta}.$$
	Furthermore, if $\pi(\zeta) = \id \in \mathrm{Br}_{g+3}$, then $\Phi^\zeta_{\Gamma}$ is homotopic to the identity, and hence induces the identity on homology.
\end{lem}

\begin{proof}
	The construction of the dg-isomorphisms $\Phi_{\Gamma}^\zeta$ is clear: the element $\wt{\sigma}_k$ acts via the action of Move V as applied to the faces $f_k$ and $f_{k+1}$ via the formula of Proposition \ref{prop:Move_V_action}. The composition rule is given by construction since we are using a free group and have defined the action on a minimal generating set.
	
	The kernel $K$ of the projection $\pi \colon F_{g+2} \twoheadrightarrow \mathrm{Br}_{g+3}$ is the normal subgroup generated by the braid relations. We check the extent to which the braid relations are satisfied by the morphisms constructed by Move V. Notice that it suffices to prove that $\Phi_{\Gamma}^{\zeta}$ is homotopic to the identity for these braid relations which normally generate $K$. Indeed, suppose for $\zeta$ one such generator we have $\Psi^{\zeta}_{\Gamma}$ is our desired chain homotopy, i.e. with
	$$\Psi^{\zeta}_{\Gamma} \circ \wt{\partial}^{\scr{B}^+}_{G,\Gamma} - \wt{\partial}^{\scr{B}^+}_{G,\zeta \cdot \Gamma} \circ  \Psi^{\zeta}_{\Gamma}= \id-\Phi_{\Gamma}^{\zeta}.$$
	Then for $\zeta,\theta \in K$, we obtain a chain homotopy for $\Phi^{\zeta\theta^{-1}}_{\Gamma}$ by taking
	$$\Psi_{\Gamma}^{\zeta\theta^{-1}} = \left(\Psi_{\theta^{-1}\cdot\Gamma}^{\zeta} - \Psi_{\Gamma}^{\theta}\right) \circ \Phi_{\theta^{-1}\cdot\Gamma}^{\theta^{-1}},$$
	and if $\theta \in F_{g+2}$ and $\zeta \in K$, then we may take
	$$\Psi^{\theta^{-1}\zeta\theta}_{\Gamma} = (\Phi_{\Gamma}^\theta)^{-1}\Psi^\zeta_{\theta \cdot \Gamma}\Phi_{\Gamma}^{\theta}.$$

	First, consider the braid relations $\sigma_i\sigma_j = \sigma_j\sigma_i$ for $|i-j| \geq 2$. Notice that Move V applied to $f_i$ and $f_{i+1}$ does not affect $\gamma_{f_j,f_{j+1}}$ at all, and vice versa. In other words,
	$$\gamma^{\sigma_i \cdot \Gamma}_{f_j,f_{j+1}} = \gamma^{\Gamma}_{f_j,f_{j+1}}$$
	$$\gamma^{\sigma_j \cdot \Gamma}_{f_i,f_{i+1}} = \gamma^{\Gamma}_{f_i,f_{i+1}}$$
	We find that both $\wt{\sigma}_i\wt{\sigma}_j$ and $\wt{\sigma}_j\wt{\sigma}_i$ therefore both act on the degree 2 generators by
	$$X \mapsto X + W(\gamma^{\Gamma}_{f_i,f_{i+1}}) + W(\gamma^{\Gamma}_{f_j,f_{j+1}}).$$
	Meanwhile, they act by the identity on the degree 1 generators. Hence, since they act the same way on all generators, they are identical. In particular, we have $\Phi_{\Gamma}^{\wt{\sigma}_i\wt{\sigma}_j(\wt{\sigma}_j\wt{\sigma}_i)^{-1}} = \id$.

	The second braid relation $\sigma_i\sigma_{i+1}\sigma_i = \sigma_{i+1}\sigma_i\sigma_{i+1}$ is a little more tricky, and we will need to choose nonzero $\Psi$. To simplify notation, we will suppose $i=1$. The argument is the same for any other triple of adjacent tines, so this may be assumed without loss of generality. The tines will simply be referred to as $\gamma_1$, $\gamma_2$, and $\gamma_3$ passing through the centers $c_1$, $c_2$, and $c_3$, respectively, and we will use paths $\eta_1$ from $c_1$ to $c_2$ and $\eta_2$ from $c_2$ to $c_3$ completely contained between neighboring tines. As usual, for any tine we use the decomposition
	$$\gamma_f = \gamma_f^s * \gamma_f^c * \gamma_f^t.$$
	See Figure \ref{fig:Move_V_Braid_Relation_Labels}.	
	
	\begin{figure}[h]
		\centering
		\includegraphics[width=\textwidth]{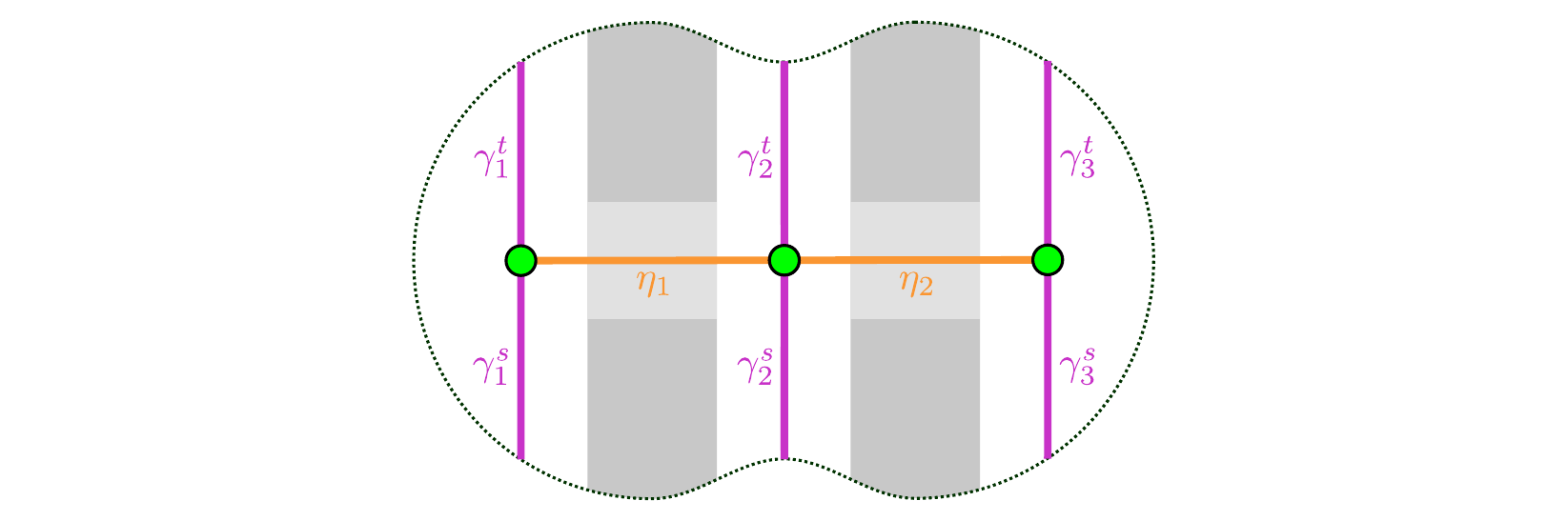}
		\caption{All of the tines are oriented upwards, and the curves $\eta_1$ and $\eta_2$ are oriented from left to right.}
		\label{fig:Move_V_Braid_Relation_Labels}
	\end{figure}

	Let us consider the action of $\sigma_1 \sigma_2 \sigma_1$. As discussed, the action of $\sigma_1$ is given on the generators $x$, $y$, $z$, and $w$, by adding the corresponding associated words for $\gamma_{1,2}$, which goes along $\gamma_1^s$ through $f_1$, then along $\eta_1$, through $f_2$, and ends with $\gamma_2^t$. Once we have performed this operation, we are now in a position to apply $\sigma_{2}$, as in Figure \ref{fig:Move_V_Braid_Relation_Step_1}.
	
	\begin{figure}[h]
		\centering
		\includegraphics[width=\textwidth]{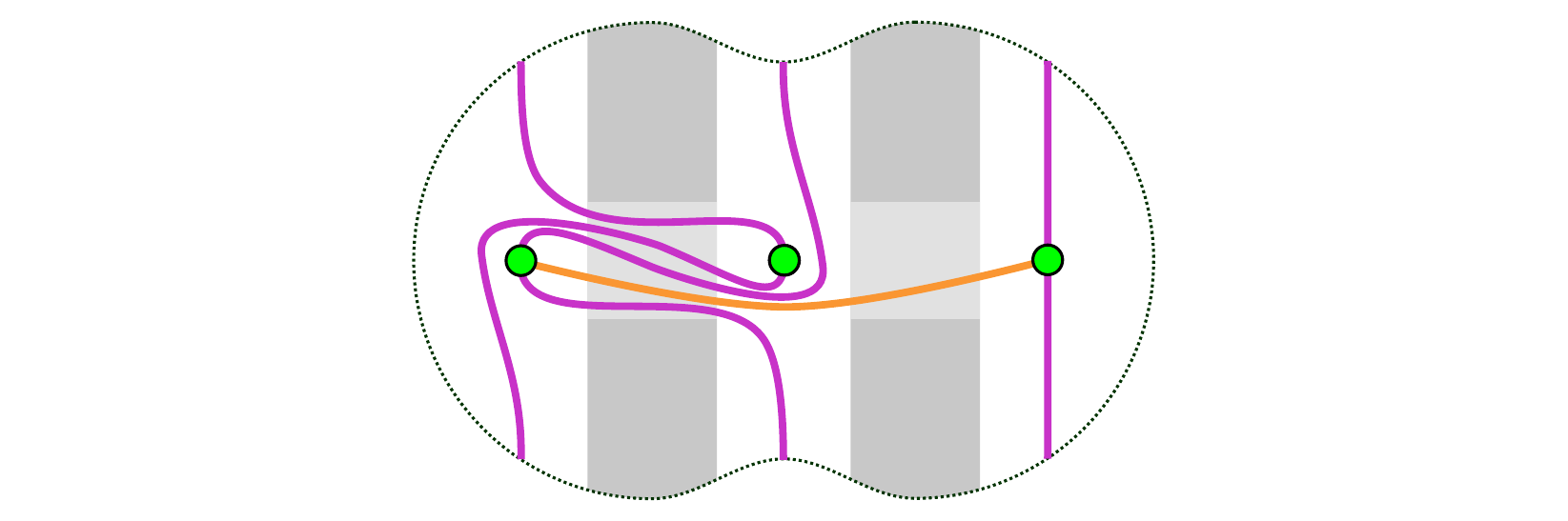}
		\caption{After performing Move V according to $\sigma_1$, we are now in a position to apply Move V along the orange curve, i.e. by $\sigma_2$.}
		\label{fig:Move_V_Braid_Relation_Step_1}
	\end{figure}

	We see in this case that the connecting curve goes under the second center $f_2$. By homotoping the curves using Moves I-IV, we see that Move V acts at this stage by the word associated to $\gamma_1^s * (\eta_{1,2}^D) * \gamma_3^t$, where the superscript $D$ denotes that the curve goes down, under $f_2$. Finally, after this application of Move V, the picture now looks as in the left diagram of Figure \ref{fig:Move_V_Braid_Relation_Step_2}, which we homotope to the right diagram.
	\begin{figure}[h]
		\centering
		\includegraphics[width=\textwidth]{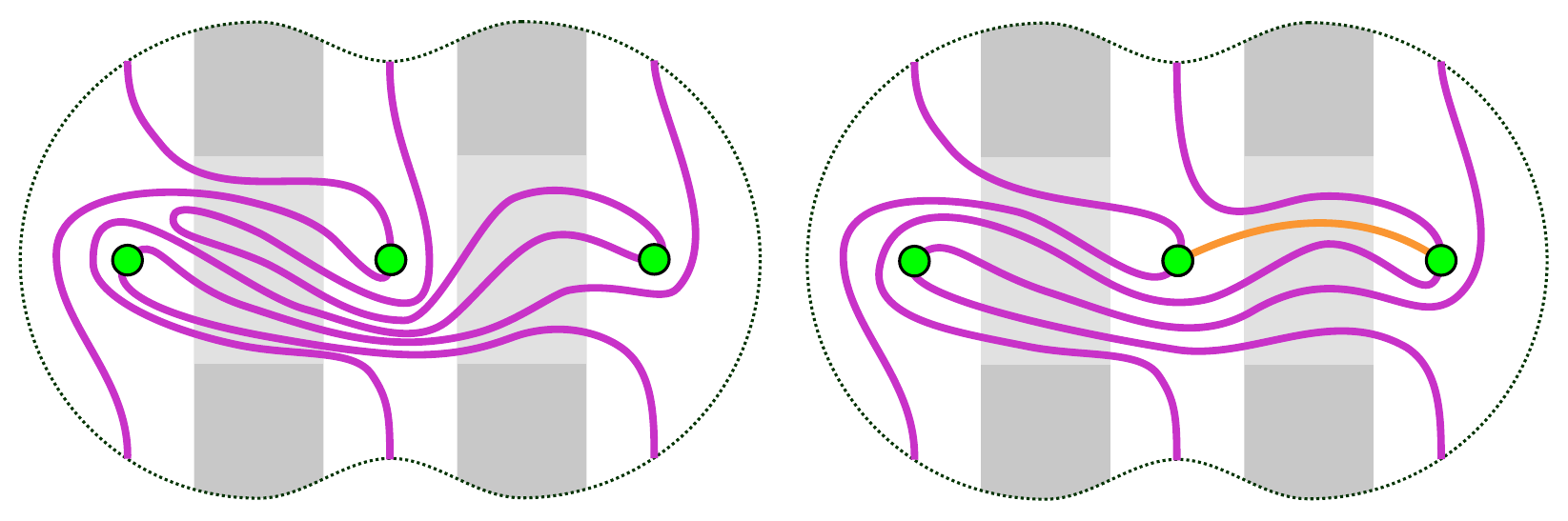}
		\caption{The diagram after an application of $\sigma_i$ followed by an application of $\sigma_{i+1}$. The left and right figures are the same up to a number of applications of Moves I and II. A final application of $\sigma_{i}$ may now be performed by considering Move V along the orange curve of the right diagram.}
		\label{fig:Move_V_Braid_Relation_Step_2}
	\end{figure}
	In total, we find that $\wt{\sigma}_1\wt{\sigma}_{2}\wt{\sigma}_1$ acts via
	$$X \mapsto X+W(A) + W(B) + W(C),$$
	where $A$, $B$, and $C$ are the curves in the top row of Figure \ref{fig:Move_V_Braid_Relation_Curves}. Similarly, we find that $\wt{\sigma}_2\wt{\sigma}_1\wt{\sigma}_2$ acts via
	$$X \mapsto X+W(A') + W(B') + W(C')$$
	and similarly on the other generators, where $A'$, $B'$, and $C'$ are on the bottom row of Figure \ref{fig:Move_V_Braid_Relation_Curves}.
	\begin{figure}[h]
		\centering
		\includegraphics[width=\textwidth]{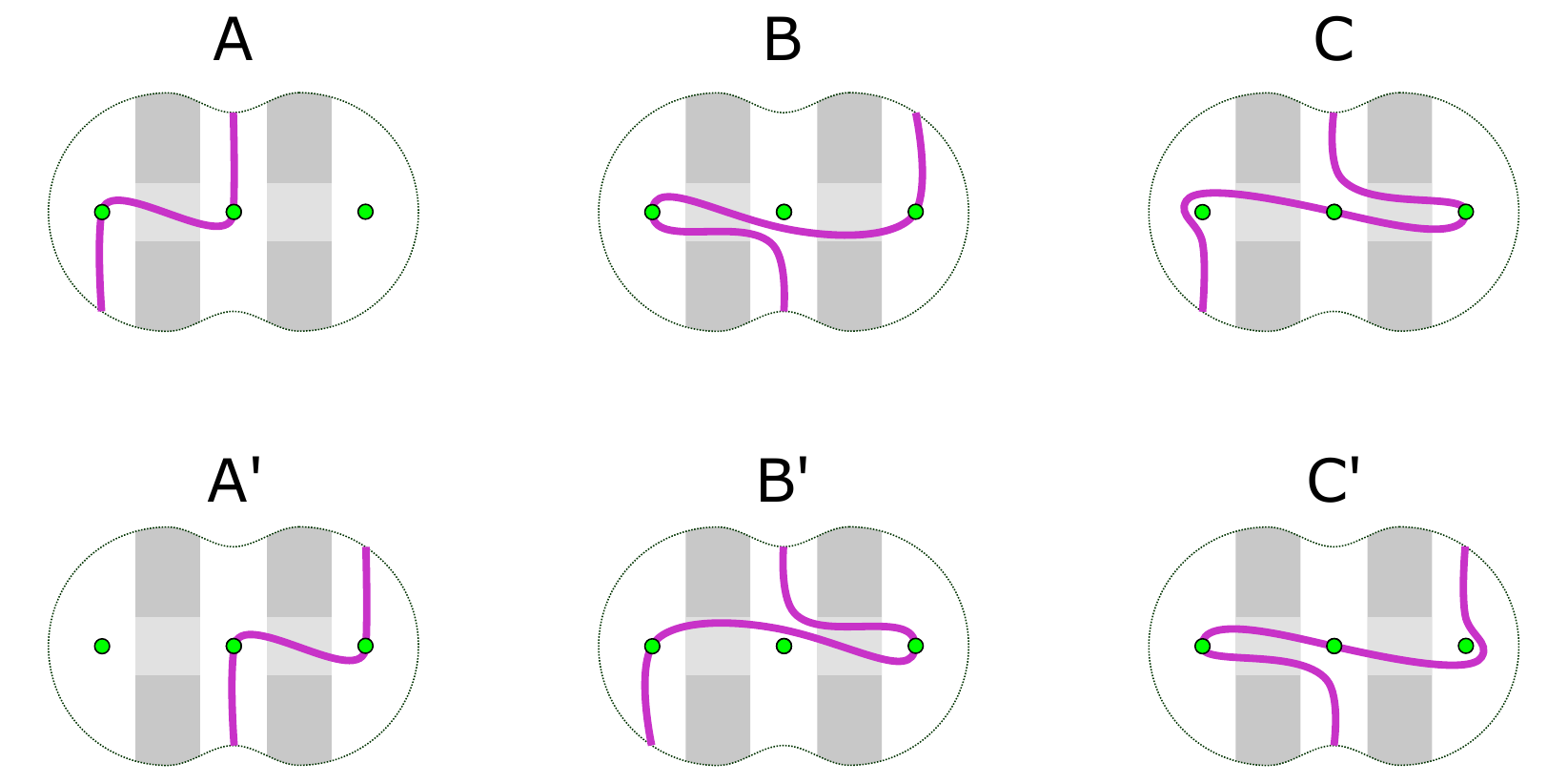}
		\caption{The curves corresponding to the words added by the action of $\sigma_1\sigma_2\sigma_1$, in the top row, and $\sigma_2\sigma_1\sigma_2$, in the bottom row.}
		\label{fig:Move_V_Braid_Relation_Curves}
	\end{figure}
	We see that up to homotopy through Moves I-IV:
	\begin{itemize}
		\item $A$ and $C'$ are the same curve, except that $C'$ is pushed to go around $f_3$ to the right, whereas $A$ goes around it to the left
		\item $B$ and $B'$ are the same curve, except that $B'$ is pushed to go over $f_2$, whereas $B$ goes under it
		\item $C$ and $A'$ are the same, except that $C$ is pushed to go around $f_1$ to the left, whereas $A'$ goes around it to the right
	\end{itemize}
	We described in Lemma \ref{lem:htpy_through_center} how pushing a non-degenerate curve through a center changes the associated word by an exact element, with primitive involving the face through which the curve is pushed. One easily sees that
	$$\left(W(A') + W(B') + W(C')\right) - \left(W(A) + W(B) + W(C)\right) = \wt{\partial}^{\scr{B}^+} W(\gamma_{f_1,f_2,f_3}^{\Gamma}),$$
	where $\gamma_{f_1,f_2,f_3}^{\Gamma}$ is the path which follows $\gamma_{1}^{s}$, then passes through $f_1$, then follows $\eta_1$, then passes through $f_2$, then follows $\eta_2$, then passes through $f_3$, and finally traverses $\gamma_3^t$. (Notice we have removed the garden from the notation on the right since for all such choices, the differential acts in the same way on the matrix of words, since the words only involve the degree $1$ face generators.) In other words, we have
	$$\left(\Phi_{\Gamma}^{\wt{\sigma}_2\wt{\sigma}_1\wt{\sigma}_2} - \Phi_{\Gamma}^{\wt{\sigma}_1\wt{\sigma}_2\wt{\sigma}_1}\right)(X) = \wt{\partial}^{\scr{B}^+}W(\gamma_{f_1,f_2,f_3}^{\Gamma}).$$
	Applying $\left(\Phi_{\Gamma}^{\wt{\sigma_1}\wt{\sigma_2}\wt{\sigma_1}}\right)^{-1}$, which acts trivially on the right hand side since degree $1$ generators are unharmed, yields
	$$\Phi_{\Gamma}^{\wt{\sigma_1}^{-1}\wt{\sigma}_2^{-1}\wt{\sigma}_1^{-1}\wt{\sigma}_2\wt{\sigma}_1\wt{\sigma}_2}(X) = X + \wt{\partial}^{\scr{B}^+}W(\gamma_{f_1,f_2,f_3}^{\Gamma}).$$
	This is homotopic to the identity because it satisfies the hypotheses of Proposition \ref{prop:condition_htpc_identity}.
\end{proof}

\begin{rmk}
	From the proof, we see that everywhere we use the action of the free group $F_{g+2}$, we could instead use the action of the quotient by the normal subgroup generated by the commutators $[\sigma_i,\sigma_j]$ for $|i-j| \geq 2$.
\end{rmk}

\subsection{The full twist} \label{ssec:Full_Twist}

So far, we have understood the path components $\pi_0(\scr{G}(G;\xi))$ of gardens where we have fixed the seed as a $\mathrm{Br}_{g+3}$-torsor. Let us now study $\pi_0(\scr{G}(G))$. (Again, implicitly, we have fixed an orientation in this discussion.) We have a Serre fibration
$$\pi \colon \scr{G}(G) \rightarrow S^*(S^2 \setminus (G \cup W))$$
given by projecting to the seed. Here, $G$ is the union of the vertices and edges of the graph, and $W$ (for \emph{web}) consists of the threads and centers. By definition, $\pi^{-1}(\xi) = \scr{G}(G;\xi)$. On the other hand, we have $S^2 \setminus (G \cup W)$ consists of a union of open triangles, each with one vertex at a center opposite an edge of the triangle which is just an edge of the graph, and two vertices on the vertices of the graph, each opposite an edge lying along a thread of the web. Since each edge of the graph has a triangle on either side of it, there are $2|E| = 6g+6$ such triangles. Let $\mathrm{Tri}(G)$ be the collection of these triangles, just a finite set with $6g+6$ elements. Let $\nu \in \mathrm{Tri}(G)$ be one such triangle. Then we obtain a sequence
$$\pi_1(S^*\nu;\xi) \rightarrow \pi_0(\scr{G}(G;\xi)) \rightarrow \pi_0(\scr{G}(G))$$
which is exact at the central term. Notice that $S^*\nu \cong \nu \times S^1 \sim S^1$, so $\pi_1(S^*\nu;\xi) \cong \Z$. Geometrically, this represents the free action given by twisting $\xi$ by a full 360 degrees. This $\Z$-action on $\pi_0(\scr{G}(G;\xi))$ corresponds to the full braid twist in $\mathrm{Br}_{g+3}$, which is the generator of the center of $\mathrm{Br}_{g+3}$. We hence find that $\pi_0(\scr{G}(G))$ consists of a collection of $6g+6$ $\mathrm{Br}_{g+3}/\Z$-torsors, indexed by $\mathrm{Tri}(G)$ based on in which triangle the base point lies. We will study $\mathrm{Tri}(G)$ further in Section \ref{ssec:Moves_VI_VII}.

Even though the full twist acts non-trivially on $\pi_0(\scr{G}(G;\xi))$, it acts trivially on $\pi_0(\scr{G})$. Comparing with Theorem \ref{thm:nc_CM_functoriality}, we must therefore prove that the full twist is homotopic to the identity.

\begin{lem} \label{lem:full_twist}
	Suppose $G$ is a trivalent plane graph with garden $\Gamma \in \pi_0(\scr{G}(G;\xi))$. Then any lift $\wt{\Delta} \in F_{g+2}$ of the full twist $\Delta = (\sigma_{g+2}\cdots\sigma_1)^{g+3} \in \mathrm{Br}_{g+3}$ acts on $(\wt{\scr{B}}^+_{G},\wt{\partial}^{\scr{B}^+}_{G,\Gamma})$ via a tame dg-automorphism homotopic to the identity.
\end{lem}

\begin{proof}
	It suffices to prove the result for one lift, since a composition of dg-automorphisms homotopic to the identity is itself homotopic to the identity, and any two lifts differ by a dg-automorphism homotopic to the identity by Lemma \ref{lem:Move_V_invariance}. We use the lift
	$$\wt{\Delta} = (\wt{\sigma}_{g+2}\cdots\wt{\sigma}_1)^{g+3}.$$
	Let $\wt{C} = \wt{\sigma}_{g+2}\cdots\wt{\sigma}_1$ be a lift of $C \in \mathrm{Br}_{g+3}$, so $\wt{\Delta} = \wt{C}^{g+3}$. Identifying $\mathrm{Br}_{g+3} \cong \pi_1(\mathrm{UConf}_{g+3}(S^2 \setminus \{s\}))$, we see $C$ acts on the garden as follows. Up to homotopy, there is a unique decomposition $S^2 = \D_+ \cup_{S^1} \D_-$ into two disks with the property that $\D_+$ contains the seed, and that the centers appear along $S^1 = \D_+ \cap \D_-$ in order according to the boundary orientation of $\D_+$. Then $\wt{C}$ acts on the tines via a homotopy supported near the gluing $S^1$, so that the centers are rotated from one to the next. See Figure \ref{fig:Cyclical_Twist}.
	
	\begin{figure}[h]
		\centering
		\includegraphics[width=\textwidth]{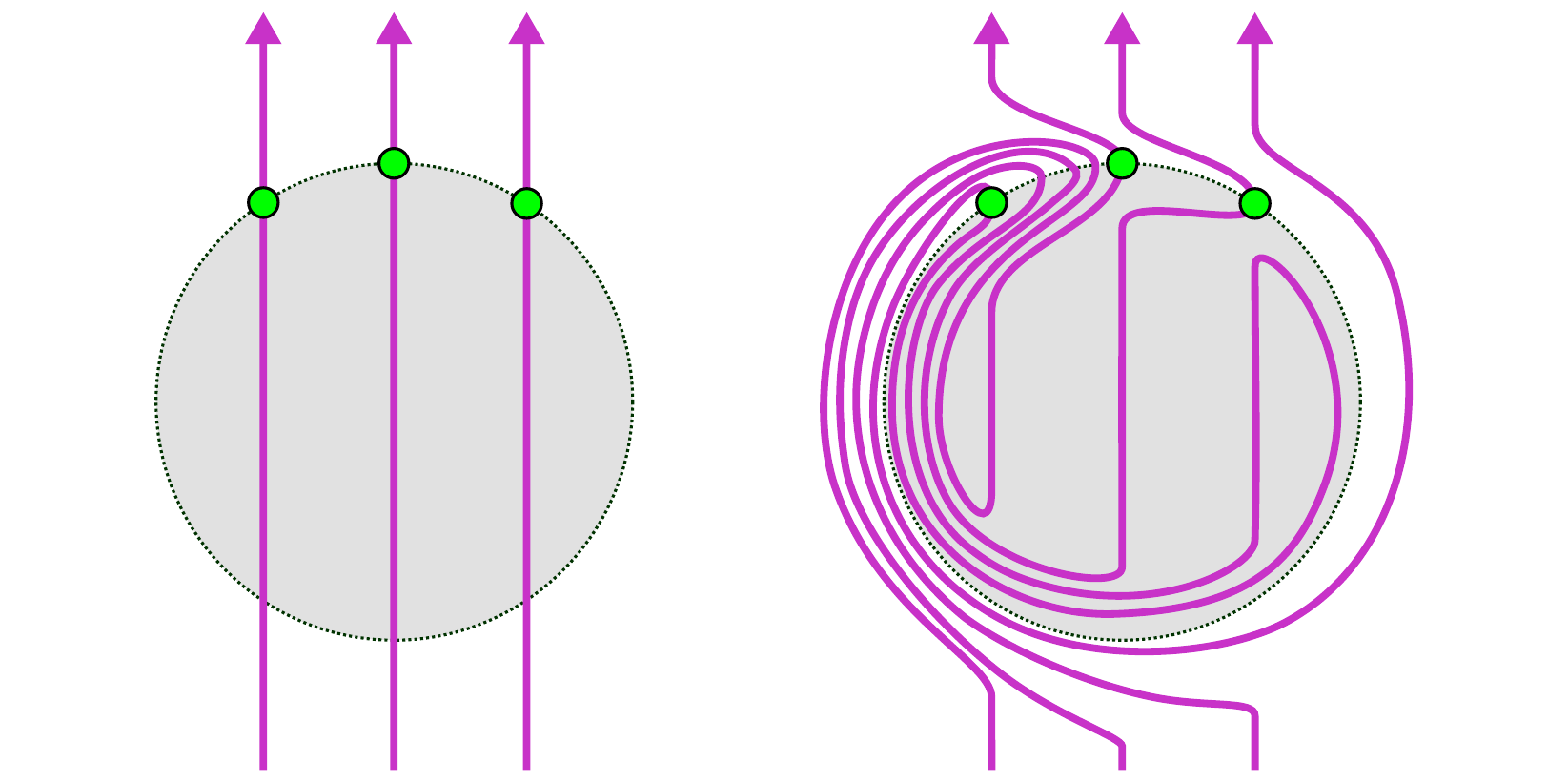}
		\caption{The action of the cyclical twist $C$ on the tines. In general, there may be more tines, but we only draw three for convenience. The disk $\D_-$ is shaded in grey, and the centers occur in the order from left to right across the top of the disk. The seed is off at infinity. We have drawn the action of $C$ so that it is supported in a neighborhood of the circle separating $\D_+$ and $\D_-$.}
		\label{fig:Cyclical_Twist}
	\end{figure}
	
	Let us consider the action of $\wt{C}$ on the dg-algebra. Notice that curves A and B in Figure \ref{fig:Move_V_Braid_Relation_Curves} already correspond to the curves $\gamma_{f,g}$ involved in $\wt{\sigma}_1$ followed by $\wt{\sigma}_2$ respectively. It is not hard to see how this pattern continues as we apply $\wt{C}$: $\wt{\sigma}_k$ corresponds to the curve which goes through $f_1$, under $f_2$ through $f_{k}$ and then through $f_{k+1}$. Let us denote these curves by $\gamma^{\Gamma}_{1,k+1}$.
	
	More generally, we may define curves $\gamma^{\Gamma}_{i,j}$ as the homotopy class of loop based at the base point of $\Gamma$ given by first following a path from the base point to $c_i$ through $\D_+$, then a path from $c_i$ to $c_j$ through $\D_-$, and then a path from $c_j$ back to the base point through $\D_+$. Notice that this description completely determines a unique curve up to homotopy since $\D_+$ and $\D_-$ are contractible. Further, notice that when $i=1$ and $j=k+1$, this matches the description above. Finally, up to homotopy, we have simply
	$$\gamma^{\Gamma}_{i,j} \sim (\gamma^{\Gamma}_i)^{-1} * \gamma^{\Gamma}_j,$$
	where $\gamma^{\Gamma}_k$ is the $k$th tine. The curves $\gamma^{\Gamma}_{i,j}$ are shown in Figure \ref{fig:Twist_Curves}.
	
	\begin{figure}[h]
		\centering
		\includegraphics[width=\textwidth]{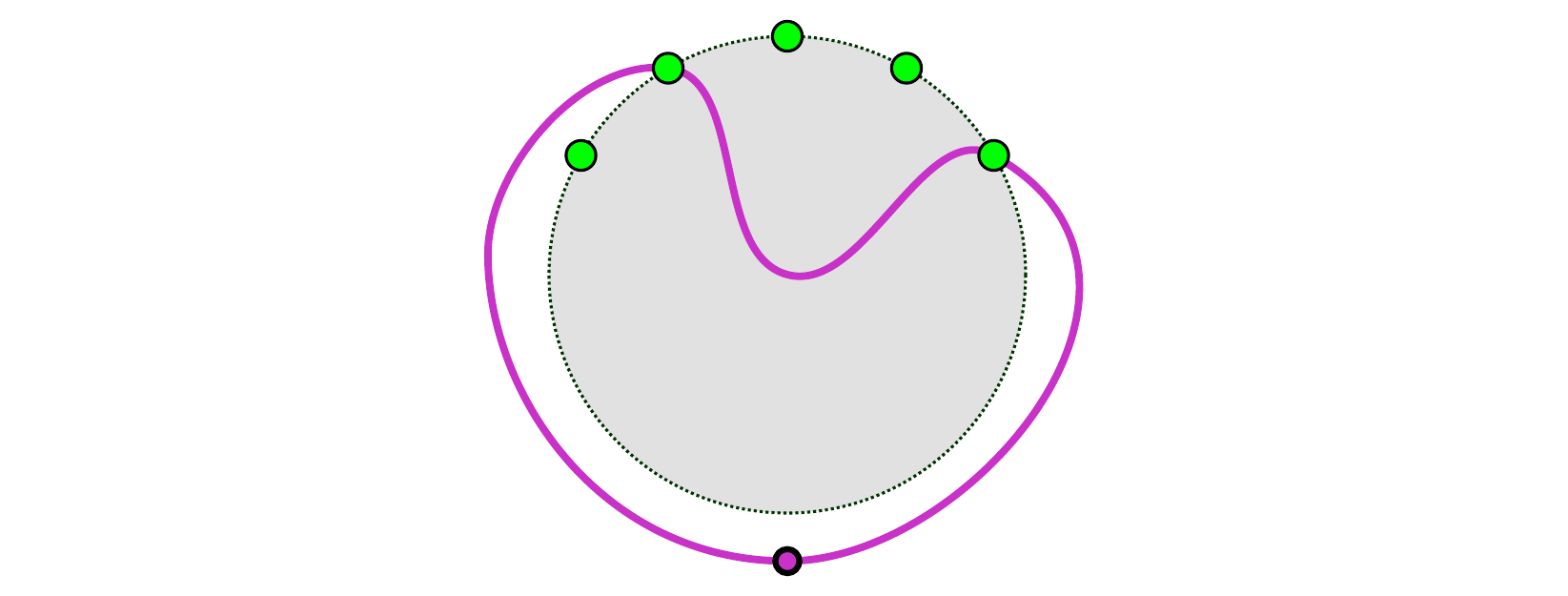}
		\caption{The curves $\gamma^{\Gamma}_{i,j}$, where $\D_-$ is highlighted in grey. In this case, imagining the centers are labelled in this case from $1$ to $5$ from left to right, then the curve depicted is either $\gamma^{\Gamma}_{2,5}$ or $\gamma^{\Gamma}_{5,2}$, depending upon the orientation of the curve. Here, we have moved the base point in from infinity for a slightly different perpsective.}
		\label{fig:Twist_Curves}
	\end{figure}
	
	Notice, either by inspection of Figure \ref{fig:Cyclical_Twist}, or by noting that the decomposition $S^2 = \D_+ \cup \D_-$ is compatible with both $\Gamma$ and $C \cdot \Gamma$, we have
	$$\gamma_{i,j}^{C \cdot \Gamma} = \gamma^{\Gamma}_{i+1,j+1},$$
	with the convention that the indices are taken modulo $g+3$. In fact, explicitly on tines, one sees that
	$$\gamma_i^{C \cdot \Gamma} = \eta_1^{-1} * \gamma_{i+1}^{\Gamma}$$
	where $\eta_1$ is a loop lying between $\gamma^{\Gamma}_1$ and $\gamma^{\Gamma}_2$ (as in Proposition \ref{prop:differential}).

	On the other hand, the action of $\wt{C}$ on the degree 2 generators is just
	$$\wt{C} \cdot X = X + \sum_{k=1}^{g+2}W(\gamma^{\Gamma}_{1,k+1}).$$
	Hence, applying $\wt{C}$ a total of $g+3$ times, we find
	\begin{align*}
		\wt{C}^{g+3} \cdot X &= X + \sum_{j=0}^{g+2}\sum_{k=2}^{g+3}W(\gamma^{C^j \cdot \Gamma}_{1,k}) \\
		&= \sum_{\left\{\begin{matrix}1 \leq i,j \leq g+3 \\ i \neq j\end{matrix}\right\}} W(\gamma^{\Gamma}_{i,j}) \\
		&= \sum_{\left\{\begin{matrix}1 \leq i,j \leq g+3 \\ i \neq j\end{matrix}\right\}} W((\gamma^{\Gamma}_{i})^{-1}) \cdot W(\gamma^{\Gamma}_j).
	\end{align*}
	But even further, notice that if $i=j$, then $\gamma_{i,i}$ can be homotoped so that it hits $c_{f_i}$ twice in a row, without intersecting a thread or edge in between. Hence $W(\gamma_{i,i}) = 0$, since no binary sequence can switch from $0$ to $1$ at both points it intersects $c_{f_i}$ since it doesn't have a chance to go from $1$ to $0$ in between. Hence, we may more simply write
	$$\wt{C}^{g+3} \cdot X = X + \left(\sum_{i=1}^{g+3}W(\gamma_i^{-1})\right) \cdot \left(\sum_{i=1}^{g+3}W(\gamma_j) \right).$$
	
	By Proposition \ref{prop:condition_htpc_identity}, we see that the action of $\wt{\Delta}$ satisfies the required conditions so long as we can prove that $\left(\sum_{i=1}^{g+3}W(\gamma_i^{-1})\right) \cdot \left(\sum_{j=1}^{g+3}W(\gamma_j) \right)$ is exact. To this end, let $\overline{\Gamma}$ be the garden obtained from $\Gamma$ by switching the orientation of all the tines simultaneously and negating $\xi$. Then since the differentials of $\Gamma$ and $\overline{\Gamma}$ are the same on faces, and hence on matrices of words,
	\begin{align*}
		\wt{\partial}^{\scr{B}^+}_{G,\Gamma}\left(\sum_{i=1}^{g+3} W(\gamma_i^{-1})\right) &= \wt{\partial}^{\scr{B}^+}_{G,\overline{\Gamma}}\left(\sum_{i=1}^{g+3} W(\gamma_i^{-1})\right) \\
		&= \left(\wt{\partial}^{\scr{B}^+}_{G,\overline{\Gamma}}\right)^2 X \\
		&= 0.
	\end{align*}
	On the other hand, $\wt{\partial}^{\scr{B}^+}_{G,\Gamma}X = \sum_{j=1}^{g+3}W(\gamma_j)$. Hence, we have
	$$\left(\sum_{i=1}^{g+3}W(\gamma_i^{-1})\right) \cdot \left(\sum_{j=1}^{g+3}W(\gamma_j) \right) = \wt{\partial}^{\scr{B}^+}_{G,\Gamma} \left[ \left(\sum_{i=1}^{g+3}W(\gamma_i^{-1})\right) \cdot X \right]$$
	is exact as desired.
\end{proof}

\subsection{Invariance under Moves VI and VII} \label{ssec:Moves_VI_VII}

Recall that Moves VI and VII pass a seed through a thread or edge respectively, as in Figure \ref{fig:Moves_VI-VII}. We now construct tame dg-automorphisms which intertwine the respective differentials for the gardens, in two separate lemmas.

\begin{lem} \label{lem:Move_VI}
	Suppose $\Gamma$ and $\Gamma'$ are two gardens which differ by Move VI as in the left side of Figure \ref{fig:Moves_VI-VII}, and suppose $\tau$ is the thread appearing in the diagram. Let $\sigma$ be $+1$ if the thread is oriented upwards and $-1$ if the thread is oriented downwards, and let $\wt{H}(\tau) = \sigma \cdot H(\tau)$ be the monomial associated to the thread, with an appropriate sign. Then we have a tame dg-automorphism
	$$\Phi_{\Gamma}^{\Gamma'} \colon (\wt{\scr{B}}_G^+,\wt{\partial}^{\scr{B}^+}_{G,\Gamma}) \rightarrow (\wt{\scr{B}}_G^+,\wt{\partial}^{\scr{B}^+}_{G,\Gamma'})$$
	defined on generators by
	\begin{align*}
		\Phi_{\Gamma}^{\Gamma'}(f) &= f \\
		\Phi_{\Gamma}^{\Gamma'}(X) &= \begin{pmatrix}1 & \wt{H}(\tau) \\ 0 & 1\end{pmatrix}X\begin{pmatrix}1 & -\wt{H}(\tau) \\ 0 & 1\end{pmatrix}
	\end{align*}
\end{lem}

\begin{proof}
	Suppose $\gamma$ is any curve on $S^2$ from the base point $s$ back to itself with initial and final tangent vectors matching the coorientation of the seed $\xi$, and suppose that the thread passes through $\xi$ according to Move VI. Notice that before the thread passes through, we may write $\gamma = \gamma^\tau * \gamma'$, whereas after, we have $\gamma = \gamma' * \gamma^\tau$, up to homotopy of endpoints not interacting with the garden $\Gamma$ in the first case and with the garden $\Gamma'$ in the second. In particular, in this decomposition, the binary sequences associated to $\gamma'$ and $\gamma^{\tau}$ are identified, so we may write $W(\gamma')$ and $W(\gamma^{\tau})$ without respect to the garden, even though we must write $W^{\Gamma}(\gamma)$ and $W^{\Gamma'}(\gamma)$ accordingly. The matrix of words for $\gamma^{\tau}$ is
	$$W(\gamma^{\tau}) = \begin{pmatrix}1 & \wt{H}(\tau) \\ 0 & 1\end{pmatrix}.$$
	Therefore, we find
	$$W^{\Gamma}(\gamma) = \begin{pmatrix}1 & \wt{H}(\tau) \\ 0 & 1\end{pmatrix}W(\gamma'), \qquad W^{\Gamma'}(\gamma) = W(\gamma')\begin{pmatrix}1 & \wt{H}(\tau) \\ 0 & 1\end{pmatrix}.$$
	Since $\begin{pmatrix}1 & \wt{H}(\tau) \\ 0 & 1 \end{pmatrix}^{-1} = \begin{pmatrix}1 & -\wt{H}(\tau) \\ 0 & 1 \end{pmatrix}$, we arrive at
	$$W^{\Gamma}(\gamma) = \begin{pmatrix}1 & \wt{H}(\tau) \\ 0 & 1\end{pmatrix}W^{\Gamma'}(\gamma)\begin{pmatrix}1 & -\wt{H}(\tau) \\ 0 & 1\end{pmatrix}.$$
	Summing these relations over all tines $\gamma_f$, we find the desired result
	$$\wt{\partial}^{\scr{B}^+}_{G,\Gamma}(X) = \begin{pmatrix}1 & \wt{H}(\tau) \\ 0 & 1\end{pmatrix} \cdot \left(\wt{\partial}^{\scr{B}^+}_{G,\Gamma'}(X)\right) \cdot \begin{pmatrix}1 & -\wt{H}(\tau) \\ 0 & 1\end{pmatrix} = \wt{\partial}^{\scr{B}^+}_{G,\Gamma'} \left(\begin{pmatrix}1 & \wt{H}(\tau) \\ 0 & 1\end{pmatrix} \cdot X \cdot \begin{pmatrix}1 & -\wt{H}(\tau) \\ 0 & 1\end{pmatrix}\right).$$
	This verifies that $\Phi_{\Gamma}^{\Gamma'}$ as in the statement in the lemma is a dg-homomorphism. Writing these expressions out in terms of $x$, $y$, $z$, and $w$, we have
	\begin{empheq}[left=\empheqbiglbrace]{align*}
		\Phi(x) &= x-y \cdot \wt{H}(\tau) \\
		\Phi(y) &= y \\
		\Phi(z) &= z + \wt{H}(\tau) \cdot y \\
		\Phi(w) &= w + \wt{H}(\tau) \cdot x - \wt{H}(\tau) \cdot y \cdot \wt{H}(\tau) - z \cdot \wt{H}(\tau)
	\end{empheq}
	from which it follows that the constructed dg-isomorphism is tame, given as a composition of four elementary dg-automorphisms.
\end{proof}

\begin{lem} \label{lem:Move_VII}
	Suppose $\Gamma$ and $\Gamma'$ are two gardens which differ by Move VII as in the right side of Figure \ref{fig:Moves_VI-VII}, and suppose $e$ is the edge appearing in the diagram. Let $\mathfrak{o}$ be the orientation of the edge upwards (regardless of the orientation induced by the garden), and set $\lambda := \lambda_{e,\mathfrak{o}}$ and $\overline{\lambda} := \lambda_{e,\overline{\mathfrak{o}}}$. Then we have a regeneration
	$$\Phi_{\Gamma}^{\Gamma'} \colon (\wt{\scr{B}}_G^+,\wt{\partial}^{\scr{B}^+}_{G,\Gamma}) \rightarrow (\wt{\scr{B}}_G^+,\wt{\partial}^{\scr{B}^+}_{G,\Gamma'})$$
	defined on generators by
	\begin{align*}
		\Phi_{\Gamma}^{\Gamma'}(f) &= f \\
		\Phi_{\Gamma}^{\Gamma'}(X) &= \begin{pmatrix}0 & -\lambda^{-1} \\ \overline{\lambda} & 0\end{pmatrix} X \begin{pmatrix}0 & \overline{\lambda}^{-1} \\ -\lambda & 0\end{pmatrix}
	\end{align*}
\end{lem}

\begin{proof}
	The proof is the same as in Lemma \ref{lem:Move_VI}, but where now we use $\gamma^e$ instead of $\gamma^\tau$, with
	$$W(\gamma^e) = \begin{pmatrix}0 & \pm \lambda^{-1} \\ \mp \overline{\lambda} & 0\end{pmatrix},$$
	where the signs depend upon whether the orientation of the edge in the garden points upwards or downwards. Either choice of sign yields the same end result.
\end{proof}

Notice that in the proof of Lemmas \ref{lem:Move_VI} and \ref{lem:Move_VII}, the key point is that the formulas for $X$ actually hold for any non-degenerate path from the seed back to itself. This is useful for the following reason. Suppose that $\Gamma$ is a garden, and suppose that $\Gamma'$ is a resulting garden after Move VI; Move VII is handled simlarly. Consider the application of Move V, according to $\wt{\sigma}_i$, on either $\Gamma$ or $\Gamma'$. Then we have
$$W(\gamma^{\Gamma}_{f_i,f_{i+1}}) = \begin{pmatrix}1 & \wt{H}(\tau) \\ 0 & 1\end{pmatrix}W(\gamma^{\Gamma'}_{f_i,f_{i+1}})\begin{pmatrix}1 & -\wt{H}(\tau) \\ 0 & 1\end{pmatrix}.$$
It follows, regardless of whether we apply Move V followed by Move VI, or Move VI followed by Move V, the image of the matrix of generators $X$ is given by the left and right hand sides of the following equality:
$$\begin{pmatrix}1 & \wt{H}(\tau) \\ 0 & 1\end{pmatrix}X\begin{pmatrix}1 & -\wt{H}(\tau) \\ 0 & 1\end{pmatrix} + W(\gamma^{\Gamma}_{f_i,f_{i+1}}) = \begin{pmatrix}1 & \wt{H}(\tau) \\ 0 & 1\end{pmatrix}(X+W(\gamma^{\Gamma'}_{f_i,f_{i+1}}))\begin{pmatrix}1 & -\wt{H}(\tau) \\ 0 & 1\end{pmatrix}.$$
In other words, the dg-isomorphisms by which Moves V and VI act commute. By the same reasoning, Moves V and VII also commute. Hence, we may essentially study Moves VI and VII independently from Move V, which we already understand quite well, as summarized by Lemma \ref{lem:Move_V_invariance}.

Recall from Section \ref{ssec:Full_Twist} that $\pi_0(\scr{G}(G))$ is given by a collection of $\mathrm{Br}_{g+3}/\Z$-torsors over the collection $\mathrm{Tri}(G)$ of triangles obtained by cutting $S^2$ along the graph and the web (which we have fixed up to isotopy, since the web is a contractible choice). Moves VI and VII are precisely the moves which allow us to move from one triangle to the next. In fact, suppose that triangles $\nu_1,\nu_2 \in \mathrm{Tri}(G)$ are adjacent to each other, sharing an edge which is either an edge of the graph or a thread. If we have a garden $\Gamma$ with base point in $\nu_1$, then we may homotope all of the tines (via Moves I-IV) so that they all pass through the edge shared between $\nu_1$ and $\nu_2$ before interacting with any other edges, threads, or centers. Then we can push the seed from $\nu_1$ to $\nu_2$ using Move VI or VII accordingly. Since these moves commute with Move V, this yields a bijection between the corresponding $\mathrm{Br}_{g+3}/\Z$-torsors of homotopy classes of gardens. In fact, this bijection is completely canonical, irrespective of how we applied Moves I-IV in order to arrive in a situation in which we could push the seed through an edge or thread via Move VI or VII. It so happens also that the formulas of Lemmas \ref{lem:Move_VI} and \ref{lem:Move_VII} only depend upon the orientation through which the seed is pushed. Hence, if $\Gamma_1$ is a homotopy class of garden with base point on $\nu_1$ and $\Gamma_2$ is the corresponding homotopy class of garden with base point on $\nu_2$ given by an application of Move VI or VII, then the corresponding formula for Move VI or VII is a canonical tame dg-automorphism or regeneration from $(\wt{\scr{B}}^+,\wt{\partial}^{\scr{B}^+}_{G,\Gamma_1})$ to $(\wt{\scr{B}}^+,\wt{\partial}^{\scr{B}^+}_{G,\Gamma_2})$.

In this way, for each triangle $\nu \in \mathrm{Tri}(G)$ and each of its edges (which is either an edge of $G$ or a thread in the web), we may associate a dg-automorphism corresponding to moving the seed through that edge. Furthermore, the triangle comes with a canonical orientation, and hence we can distinguish canonically between the two edges of the triangle which are actually threads. We write $L,R,S$, for \emph{left}, \emph{right}, and \emph{straight}, to denote the corresponding Move VI and VII actions, where the labels $L$, $R$, and $S$ appear in the order of the natural boundary orientation of the triangle, and where $S$ is placed on the unique edge of the triangle which is an edge of the graph. In this way, we naturally find an $F_3 = \langle L,R,S \rangle$-action on gardens up to homotopy which lift canonically to regenerative tame dg-automorphisms. That is, for each word $\zeta \in F_3$ and $\Gamma \in \pi_0(\scr{G}(G))$, we obtain canonical regenerative tame dg-automorphisms
$$\Phi^{\zeta}_{\Gamma} \colon (\wt{\scr{B}}_G^+,\wt{\partial}^{\scr{B}^+}_{G,\Gamma}) \rightarrow (\wt{\scr{B}}_G^+,\wt{\partial}^{\scr{B}^+}_{G,\zeta \cdot \Gamma})$$
which compose in the obvious way
$$\Phi^{\zeta\theta}_{\Gamma} = \Phi^{\zeta}_{\theta \cdot \Gamma} \circ \Phi^{\theta}_{\Gamma}.$$

As is suggested in Theorem \ref{thm:nc_CM_functoriality}, we would like to understand the stabilizer of the action on $F_3$ on gardens. If we push a seed back and forth through the same edge, then we come back to the same garden up to homotopy, i.e. $\Gamma = LR \cdot \Gamma = S^2 \cdot \Gamma = RL \cdot \Gamma$. In fact, this is reflected at the level of dg-automorphisms.

\begin{lem}
	We have that $L$ and $R$ act by inverses, and $S$ acts as its own inverse, i.e.
	$$\id = \Phi^{L}_{R \cdot \Gamma}\Phi^{R}_{\Gamma} = \Phi^{S}_{S \cdot \Gamma}\Phi^{S}_{\Gamma} = \Phi^{R}_{L \cdot \Gamma}\Phi^{L}_{\Gamma}.$$
\end{lem}
\begin{proof}
	This is obvious from Lemmas \ref{lem:Move_VI} and \ref{lem:Move_VII}, and the fact that we pass the thread or edge back and forth with opposite orientation.
\end{proof}

If a thread moves around a vertex of the graph, i.e. via the word $(SL)^3$, then the garden is again preserved up to homotopy, since the seed has just gone around a contractible loop in $S^2$ around a vertex. This is also reflected at the level of dg-automorphism.

\begin{lem} \label{lem:SL_cubed}
	We have
	$$\id = \Phi^S_{LSLSL \cdot \Gamma} \circ \Phi^L_{SLSL \cdot \Gamma} \circ \Phi^S_{LSL \cdot \Gamma} \circ \Phi^L_{SL \cdot \Gamma} \circ \Phi^S_{L \cdot \Gamma} \circ \Phi^L_{\Gamma}.$$
\end{lem}
\begin{proof}
	Upon applying $(SL)^3$, we traverse three threads and three edges. Let us label them $\tau_1,e_1,\tau_2,e_2,\tau_3,e_3$ in the order that they appear. Let $\lambda_k = A_{e_k}$ and $\overline{\lambda}_k = B_{e_k}$ with respect to the inward pointing orientation of the edge. Let
	$$Z := \begin{pmatrix}1 & H(\tau_1) \\ 0 & 1\end{pmatrix}\begin{pmatrix}0 & -\lambda_1^{-1} \\ \overline{\lambda}_1 & 0 \end{pmatrix}\begin{pmatrix}1 & H(\tau_2) \\ 0 & 1\end{pmatrix}\begin{pmatrix}0 & -\lambda_2^{-1} \\ \overline{\lambda}_2 & 0 \end{pmatrix}\begin{pmatrix}1 & H(\tau_3) \\ 0 & 1\end{pmatrix}\begin{pmatrix}0 & -\lambda_3^{-1} \\ \overline{\lambda}_3 & 0 \end{pmatrix}.$$
	The right hand side of the desired equation preserves the degree $1$ generators and acts on the degree $2$ generators as
	$$X \mapsto ZXZ^{-1}.$$
	But notice that $Z = \pm W(\gamma_v)$, where $\gamma_v$ is the loop going around a vertex once, and $W(\gamma_v) = I$, since $\gamma_v$ can be contracted to a point under Move IV.
\end{proof}

Our action of $F_{3} = \langle L,R,S \rangle$ is actually larger than we need, since even on the dg-algebra level, we have $LR = S^2 = (SL)^3 = 1$. We may therefore instead use the group
$$\Z_2 * \Z_3 \cong \langle S,(SL) \mid S^2, (SL)^3 \rangle \cong \langle L,R,S \mid LR, S^2, (SL)^3 \rangle.$$
In other words, we arrive at the following.
\begin{lem} \label{lem:Moves_V-VII_together}
	There is an action of $F_{g+2} \times (\Z_2 * \Z_3)$ on homotopy classes of gardens (with the same orientation). For each homotopy class of garden $\Gamma$ and element $\tau \in F_{g+2} \times (\Z_2 * \Z_3)$, there exists a regenerative tame dg-isomorphism
	$$\Phi_{\Gamma}^{\tau} \colon (\wt{\scr{B}}^+_G,\wt{\partial}_{G,\Gamma}^+) \rightarrow (\wt{\scr{B}}^+_G,\wt{\partial}_{G,\tau \cdot \Gamma}^+).$$
	These satisfy the composition property
	$$\Phi_{\Gamma}^{\tau\tau'} = \Phi_{\tau' \cdot \Gamma}^{\tau} \circ \Phi_{\Gamma}^{\tau'}.$$
\end{lem}
\begin{proof}
	We have defined an action of $\Z_2 * \Z_3$ on the dg-algebras which commute with the action of $F_{g+3}$ via Move V. In other words, we may write
	$$\Phi_{\Gamma}^{\tau} := \Phi_{\tau_2 \cdot \Gamma}^{\tau_1} \circ \Phi_{\Gamma}^{\tau_2} = \Phi_{\tau_1 \cdot \Gamma}^{\tau_2} \circ \Phi_{\Gamma}^{\tau_1}$$
	where $\tau = (\tau_1,\tau_2) \in F_{g+2} \times (\Z_2 * \Z_3)$. The commutativity allows us to group like terms. In other words, if $\tau = (\tau_1,\tau_2)$ and $\tau' = (\tau'_1,\tau'_2)$, then:
	\begin{align*}
		\Phi_{\tau' \cdot \Gamma}^{\tau} \circ \Phi_{\Gamma}^{\tau'} &= \Phi_{(\tau_2\tau') \cdot \Gamma}^{\tau_1} \circ \Phi_{\tau' \cdot \Gamma}^{\tau_2} \circ \Phi_{\tau'_2 \cdot \Gamma}^{\tau'_1} \circ \Phi_{\Gamma}^{\tau'_2} \\
			&= \Phi_{(\tau_2\tau'_1\tau'_2) \cdot \Gamma}^{\tau_1} \circ \Phi_{(\tau_1'\tau_2') \cdot \Gamma}^{\tau_2} \circ \Phi_{\tau'_2 \cdot \Gamma}^{\tau'_1} \circ \Phi_{\Gamma}^{\tau'_2} \\
			&= \Phi^{\tau_2}_{(\tau_1\tau'_1\tau'_2) \cdot \Gamma} \circ \Phi^{\tau_1}_{(\tau'_1\tau'_2) \cdot \Gamma} \circ \Phi_{\tau'_2 \cdot \Gamma}^{\tau'_1} \circ \Phi_{\Gamma}^{\tau'_2} \\
			&=\Phi^{\tau_2}_{(\tau_1\tau'_1\tau'_2) \cdot \Gamma} \circ \Phi^{\tau_1\tau'_1}_{\tau'_2 \cdot \Gamma} \circ \Phi_{\Gamma}^{\tau'_2} \\
			&= \Phi^{\tau_2}_{(\tau'_2\tau_1\tau'_1) \cdot \Gamma} \circ \Phi_{(\tau_1\tau'_1) \cdot \Gamma}^{\tau'_2} \circ \Phi_{\Gamma}^{\tau_1\tau'_1} \\
			&= \Phi^{\tau_2\tau'_2}_{(\tau_1\tau'_1) \cdot \Gamma} \circ \Phi_{\Gamma}^{\tau_1\tau'_1} \\
			&= \Phi^{(\tau_1\tau'_1,\tau'_2\tau'_2)}_{\Gamma} \\
			&= \Phi^{\tau\tau'}_{\Gamma}.
	\end{align*}
\end{proof}

One of the items in Theorem \ref{thm:nc_CM_functoriality} asks us to prove that the stabilizer at $\Gamma$ of the action of $\scr{H}(G)$ on homotopy classes of gardens acts in a manner homotopic to the identity. From our description, we are now at a point where we may understand this stabilizer. We notice that since the orientation of the edges of $G$ is invariant under homotopies, any element of the stabilizer is trivial in the factor $(\Z_2)^E$.

Suppose $\Gamma \in \pi_0(\scr{G}(G))$, and let us consider its stabilizer. Suppose that we fix $s \in S^2$. Let $C$ be the union of the centering of a garden, which we continue to assume is fixed. Then there is an injective homomorphism $$\phi \colon \pi_1(S^2 \setminus C;s) \rightarrow \Z_2 * \Z_3$$ as follows. If we choose a representative non-degenerate loop, then we may label it with a word in $F_3 = \langle L,R,S \rangle$ according to how it passes through threads and edges of the graph. Any two non-degenerate representatives of the same element in the fundamental group are equivalent up to the relations $LR = S^2 = (SL)^3 = 1$, as we can form a homotopy which cancels out going back and forth through the same edge or going around a vertex. The image of $\phi$ consists of those elements of $\Z_2 * \Z_3$ which preserve gardens with base point in the triangle in $\mathrm{Tri}(G)$ containing $s$.

Notice that $S^2 \setminus C$ is homotopy equivalent to a wedge of $g+2$ circles. Hence $\pi_1(S^2 \setminus C;s) \cong F_{g+2}$. One could use a canonical set of generators given by the curves $\eta_1, \ldots, \eta_{g+2}$ appearing in the proof of Proposition \ref{prop:differential}, with $\eta_k$ lying between the tines $\gamma_{f_k}$ and $\gamma_{f_{k+1}}$. One could study the action of Moves VI and VII as applied along $\eta_k$, though we take an alternative approach.

Suppose $c_f$ is one of the centers in $C$, and suppose that the corresponding face $f$ has $n_f$ vertices, or equivalently, $c_f$ has $n_f$ triangles with $c_f$ as a vertex. If $\nu \in \mathrm{Tri}(G)$ has $c_f$ as a vertex, and $\Gamma$ is a homotopy class of garden with base point in $\nu$, then $L^{n_f} \cdot \Gamma$ is a garden again with base point in $\nu$, though not the same garden. However, there will be some element $\tau_f \in \mathrm{Br}_{g+3}/\Z$ such that $\tau_f L^{n_f} \cdot \Gamma = \Gamma$.

In our presentation, however, we are interested in studying a garden $\Gamma$ with base point $s$ which is not necessarily in a triangle with vertex $c_f$. Suppose we choose a path from $s$ to (some fixed point in) $\nu$, where $\nu$ is a triangle with vertex $c_f$. The path corresponds to a word $P_f \in \Z_2 * \Z_3$. The loops based at $s$ corresponding to $P_f L^{n_f}P_f^{-1}$ generate $\pi_1(S^2 \setminus C;s)$. (In fact, any collection of all but one of them are a free set of generators.) It follows, therefore, that if $\Gamma$ is a garden with base point $s$, then its stabilizer in $\scr{H}(G)$ is
$$\mathrm{Stab}_{\scr{H}(G)}(\Gamma) = \langle P_f \wt{\tau}_f L^{n_f} P_f^{-1} \rangle,$$
the subgroup generated by the collection $\{P_f\wt{\tau}_f L^{n_f}P_f^{-1}\}$ over all faces $f$ and lifts $\wt{\tau}_f \in F_{g+2}$ of $\tau_f \in \mathrm{Br}_{g+3}/\Z$. (This is true regardless of the choices of paths $P_f$, precisely because the loops corresponding to $P_fL^{n_f}P_f^{-1}$ generate $\pi_1(S^2 \setminus C;s)$, and any other choice of $P_f$ differs by left multiplication by some element of this group.) Our goal, c.f. Theorem \ref{thm:nc_CM_functoriality}, is to prove that the various $P_f\wt{\tau}_f L^{n_f}P_f^{-1}$ act $(\wt{\scr{B}}^+_{G},\wt{\partial}^{\scr{B}^+}_{G,\Gamma})$ via dg-automorphisms homotopic to the identity.

\begin{lem} \label{lem:Move_VI_around_face}
	In the set up just described, the dg-automorphisms $(\wt{\scr{B}}^+_{G},\wt{\partial}^{\scr{B}^+}_{G,\Gamma})$ induced by the elements $P_f\wt{\tau}_f L^{n_f}P_f^{-1}$ are homotopic to the identity. 
\end{lem}
\begin{proof}
	It suffices to prove that $\wt{\tau}_f L^{n_f}$ is homotopic to the identity whenever $\Gamma$ has base point in one of the triangles with vertex $c_f$, since conjugation by $P_f$ does not affect this. In fact, for the same reason, by conjugating by a sequence of actions of Move V, we may assume without loss of generality that $\Gamma$ is such that the tine $\gamma_f$ passing through $c_f$ is the right-most tine exiting the seed. In other words, if we enumerate the faces $f_1,\ldots,f_{g+3}$ according to the order of the tines, then $f = f_{g+3}$ is the specific face we are focusing on. We label the centers $c_1,\ldots,c_{g+3}$ and the tines $\gamma_1,\ldots,\gamma_{g+3}$ accordingly. We write $n_{g+3}$ instead of $n_{f_{g+3}}$ and $\tau$ for $\tau_{f_{g+3}}$. Finally, we note that it suffices to prove the result for a single lift $\wt{\tau}$ of $\tau$, since any two lifts differ by a sequence of full twists, which acts by a dg-automorphism itself homotopic to the identity.

	With these assumptions out of the way, we must first understand the element $\tau \in \mathrm{Br}_{g+3}/\Z$. Consider a decomposition of $S^2 = \D_+ \cup_{S^1} \D_-$ into two disks, similar to in Lemma \ref{lem:full_twist}, but now where both the base point and $c_{g+3}$ are contained in the interior of $\D_+$, and $c_{1}$ through $c_{g+2}$ now occur in order around the boundary of $\D_+$, as in See Figure \ref{fig:Partial_Twist_Set-up}. Notice that if we apply $L^{n_f}$ to the garden, it has the effect given by twisting $\D_+$ clockwise by a full twist. Correspondingly, $\tau$ must be a counter-clockwise twist of $\D_-$, since this undoes the clockwise twist of $\D_+$. That is,
	$$\tau = (\sigma_{g+1}\cdots \sigma_2\sigma_1)^{g+2}.$$
	We may therefore use the lift
	$$\wt{\tau} := (\wt{\sigma}_{g+1}\cdots \wt{\sigma}_2\wt{\sigma}_1)^{g+2}.$$
	
	\begin{figure}[h]
		\centering
		\includegraphics[width=\textwidth]{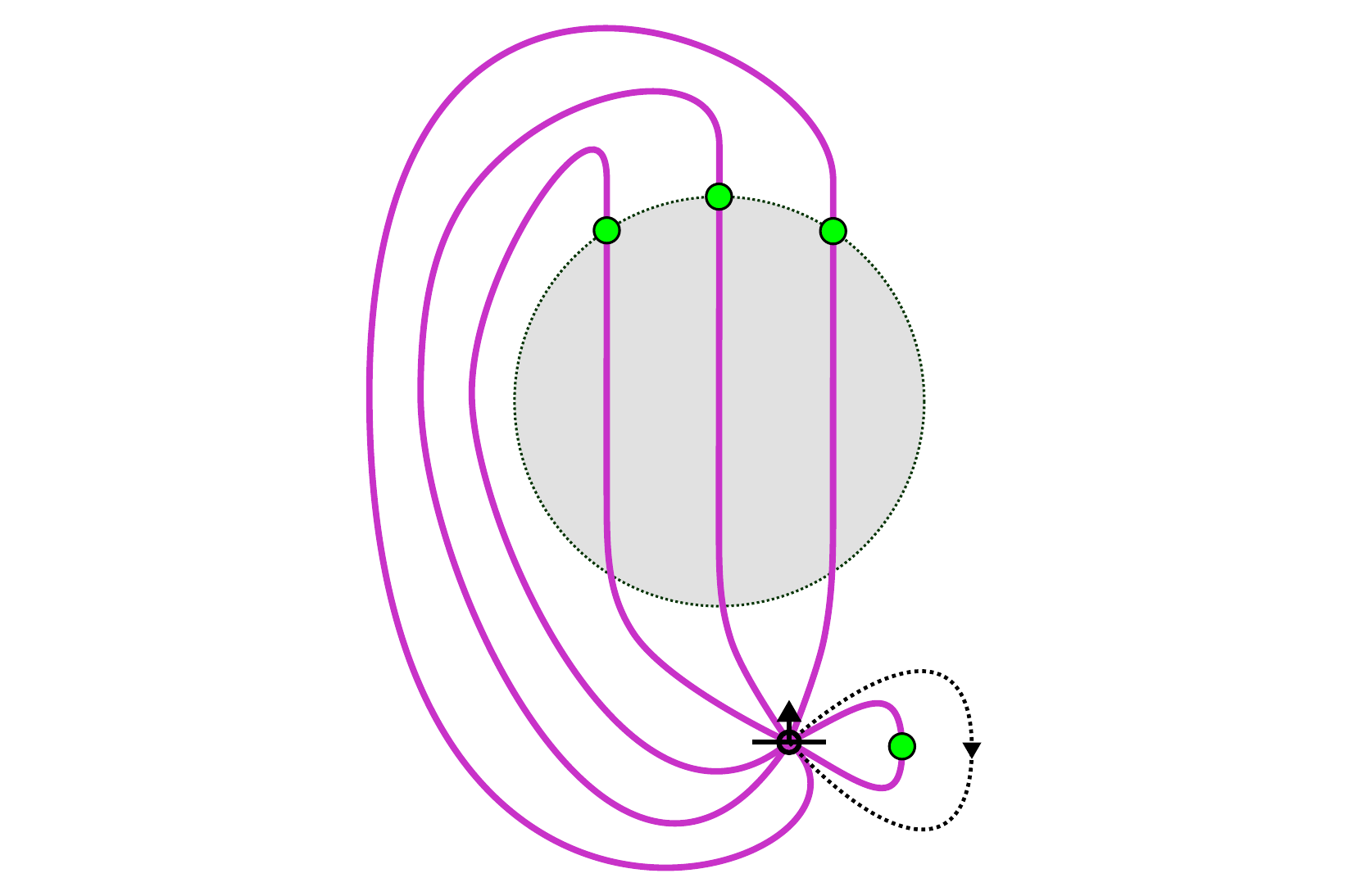}
		\caption{Our set-up for Lemma \ref{lem:Move_VI_around_face}, where $\D_-$ is shaded in grey, and the center $c_{g+3}$ is now in the interior of $\D_+$. For simplicity, we have drawn the situation $g=1$, so that there are only four tines. The dotted black arrow indicates that $L^{n_{g+3}}$ acts by homotoping the seed around this loop (under the assumption that the base point is located in the face corresponding to $f_{g+3}$), which is realized by a full clockwise rotation of $\D_+$.}
		\label{fig:Partial_Twist_Set-up}
	\end{figure}
	
	Let $\wt{D} := \wt{\sigma}_{g+1}\cdots \wt{\sigma}_2\wt{\sigma}_1$, so that $\wt{\tau} = \wt{D}^{g+2}$. Let $D = \pi(\wt{\tau})$, so that $\tau = D^{g+2}$. Then the action of $\wt{\tau}$ is given by preserving the degree $1$ generators and acting on the degree $2$ generators by
	$$\wt{\tau} \cdot X = X+\sum_{j=0}^{g+1}\sum_{k=2}^{g+2}W(\gamma^{D^j \cdot \Gamma}_{1,k}),$$
	where we have used the same notation $\gamma^{\Gamma}_{i,j} \sim (\gamma^{\Gamma}_i)^{-1} * \gamma^{\Gamma}_j$ as in the proof of Lemma \ref{lem:full_twist}. However, the sum is less simple here, because $D$ does not simply shift indices by $1$. Indeed, by inspection of when we apply the action of $D$,
	$$\gamma^{D \cdot \Gamma}_{i} \sim  \begin{cases}(\eta_1)^{-1} *\gamma^{\Gamma}_{i+1,j+1}, & i,j \neq g+2,g+3 \\
		(\eta_1)^{-1} * \gamma^{\Gamma}_{1} * \eta_{g+2}, & i = g+2 \\ \gamma^{\Gamma}_{g+3}, & i = g+3\end{cases},$$
	where $\eta_1$ is the path between $\gamma^{\Gamma}_1$ and $\gamma^{\Gamma}_2$ and $\eta_{g+2}$ is the path between $\gamma^{\Gamma}_{g+2}$ and $\gamma^{\Gamma}_{g+3}$, as in the proof of Proposition \ref{prop:differential}. (Notice $\eta_{g+2}$ appears in Figure \ref{fig:Partial_Twist_Set-up} as the dotted black line.) It follows that the action of $\wt{\tau}$ on the generators $X$ may be written
	$$\wt{\tau} \cdot X = X + \sum_{1 \leq i<j \leq g+2 \\ i \neq j} W(\gamma^{\Gamma}_{i,j}) + \sum_{1 \leq j < i \leq g+2} W(\gamma^{\Gamma}_{i,j}*\eta_{g+2}^{-1}).$$	
	
	On the other hand, the action of $L^{n_{g+3}}$ (as the seed moves counter-clockwise around $c_{g+3}$) is given by
	$$L^{n_{g+3}} \cdot X = \begin{pmatrix}1& H_1 \\ 0 & 1\end{pmatrix} \cdots \begin{pmatrix}1& H_{n_{g+3}} \\ 0 & 1\end{pmatrix}X \begin{pmatrix}1& -H_{n_{g+3}} \\ 0 & 1\end{pmatrix} \cdots \begin{pmatrix}1& -H_1 \\ 0 & 1\end{pmatrix}$$
	where $H_k$ are the monomials associated to the threads emanating from $c_{g+3}$ in order. (This formula holds for any $\Gamma$ with base point in a triangle with vertex $c_{g+3}$.) But multiplying all of these out, this is just
	$$L^{n_{g+3}} \cdot X = \begin{pmatrix}1 & \wt{\partial}^{\scr{B}^+}f_{g+3} \\ 0 & 1 \end{pmatrix} X \begin{pmatrix}1 & -\wt{\partial}^{\scr{B}^+}f_{g+3} \\ 0 & 1 \end{pmatrix}.$$

	For simplicity, let us write
	\begin{align*}
		W_i := W(\gamma_i) &\qquad \qquad \overline{W}_i := W(\gamma_i^{-1})\\
		W := \sum_{i=1}^{g+3}W_i &\qquad \qquad \overline{W} := \sum_{i=1}^{g+3}\overline{W}_i
	\end{align*}
	We have already seen, in the course of the proof of Lemma \ref{lem:full_twist}, that
	$$\wt{\partial}^{\scr{B}^+}\overline{W} = \wt{\partial}^{\scr{B}^+}W = 0$$
	$$X = \wt{\partial}^{\scr{B}^+}X = W.$$
	Notice also that
	$$W_{g+3} = \begin{pmatrix}0 & f_{g+3} \\ 0 & 0\end{pmatrix} = \overline{W}_{g+3}$$
	since they both matrices of words associated to out-and-back paths from the seed to $c_{g+3}$, which lie in the same face. Finally, notice that
	$$W(\eta_{g+2}) = \begin{pmatrix} 1 & -\wt{\partial}^{\scr{B}^+}f_{g+3} \\ 0 & 1\end{pmatrix} = I - \wt{\partial}^{\scr{B}^+}W_{g+3}.$$
	This allows us to notationally simplify both the actions of $\wt{\tau}$ and $L^{n_{g+3}}$. For the action of $\wt{\tau}$, notice that if we include terms with $i=j$, then those terms are $0$. So we have
	\begin{align*}
		\wt{\tau} \cdot X &= X + \sum_{1 \leq i<j \leq g+2} \overline{W}_iW_j + \sum_{1 \leq j<i \leq g+2} \overline{W}_iW_j(I - \wt{\partial}^{\scr{B}^+}W_{g+3}) \\
		&= X+\left(\sum_{i=1}^{g+2}\overline{W}_i\right)\left(\sum_{j=1}^{g+2}W_j\right) - \sum_{1 \leq j<i \leq g+2} \overline{W}_iW_j(\wt{\partial}^{\scr{B}^+}W_{g+3}) \\
		&= X+(\overline{W}-W_{g+3})(W-W_{g+3})+ W_{g+3}W(\wt{\partial}^{\scr{B}^+}W_{g+3}) - \sum_{1 \leq j<i \leq g+3} \overline{W}_iW_j(\wt{\partial}^{\scr{B}^+}W_{g+3}) \\
		&= X+\overline{W}W - \overline{W}W_{g+3} - W_{g+3}W + W_{g+3}W(\wt{\partial}^{\scr{B}^+}W_{g+3}) - \sum_{1 \leq j<i \leq g+3 } \overline{W}_iW_j(\wt{\partial}^{\scr{B}^+}W_{g+3}),
	\end{align*}
	where the tricky terms $W_{g+3}W(\wt{\partial}^{\scr{B}^+}W_{g+3})$ arises because we include in our sum the case $i = g+3$. The action of $L^{n_{g+3}}$ simplifies to
	\begin{align*}
		L^{n_{g+3}} \cdot X &= (I+\wt{\partial}^{\scr{B}^+}W_{g+3})X(I-\wt{\partial}^{\scr{B}^+}W_{g+3}) \\
		&= X + (\wt{\partial}^{\scr{B}^+}W_{g+3}) \cdot X - X(\wt{\partial}^{\scr{B}^+}W_{g+3}) - (\wt{\partial}^{\scr{B}^+}W_{g+3})X(\wt{\partial}^{\scr{B}^+}W_{g+3}).
	\end{align*}
	Combining these actions, we have the following, upon reordering terms.
	\begin{align*}
		\wt{\tau}L^{n_{g+3}} \cdot X - X &= \overline{W}W \\
		&+ (\wt{\partial}^{\scr{B}^+}W_{g+3}) \cdot X - W_{g+3} \cdot W\\
		&+W_{g+3}W(\wt{\partial}^{\scr{B}^+}W_{g+3})- (\wt{\partial}^{\scr{B}^+}W_{g+3})X(\wt{\partial}^{\scr{B}^+}W_{g+3}) \\
		&- \left(X+\sum_{1 \leq j<i \leq g+3} \overline{W}_iW_j \right)(\wt{\partial}^{\scr{B}^+}W_{g+3}) - \overline{W}W_{g+3}
	\end{align*}
	The first three lines on the right are exact:
	\begin{align*}
		\wt{\tau}L^{n_{g+3}} \cdot X - X &= \wt{\partial}^{\scr{B}^+}(\overline{W}X + W_{g+3}X - W_{g+3}X\wt{\partial}^{\scr{B}^+}W_{g+3} \\
		&- \left(X+\sum_{1 \leq j<i \leq g+3} \overline{W}_iW_j \right)(\wt{\partial}^{\scr{B}^+}W_{g+3}) - \overline{W}W_{g+3}.
	\end{align*}
	By Proposition \ref{prop:condition_htpc_identity}, it suffices to prove that the last line is exact. For this, it suffices to prove that
	$$\wt{\partial}^{\scr{B}^+}\left(\sum_{1 \leq j<i \leq g+3} \overline{W}_iW_j\right) = \overline{W}-W,$$
	since then the last line is just $\wt{\partial}^{\scr{B}^+}\left[\left(\sum_{1 \leq j<i \leq g+3} \overline{W}_iW_j + X\right)W_{g+3}\right].$
	
	Let us prove this final claim. Let $\eta_0$ through $\eta_{g+3}$ be as in the proof of Proposition \ref{prop:differential}, so that $\eta_k$ is between the tines $\gamma^{\Gamma}_k$ and $\gamma^{\Gamma}_{k+1}$. Then, as was proved in that proposition,
	$$\wt{\partial}^{\scr{B}^+}W_j = W(\eta_{j})-W(\eta_{j-1}),$$
	and similarly we may conclude that
	$$\wt{\partial}^{\scr{B}^+}\overline{W}_i = W((\eta_{i-1})^{-1}) - W((\eta_{i})^{-1}).$$
	Therefore,
	\begin{align*}
		\wt{\partial}^{\scr{B}^+}\left(\sum_{1 \leq j<i \leq g+3} \overline{W}_iW_j\right) &= \sum_{1 \leq j<i \leq g+3}\left[(W((\eta_{i-1})^{-1}) - W((\eta_{i})^{-1}))W_j - \overline{W}_i(W(\eta_{j})-W(\eta_{j-1})) \right] \\
		&= \sum_{j=1}^{g+3}\left(W((\eta_j)^{-1})-W(\eta_{g+3}^{-1})\right)W_j - \sum_{i=1}^{g+3}\overline{W}_i(W(\eta_{i-1})-W(\eta_0))
	\end{align*}
	by summing over one index. But $\eta_0$ and $\eta_{g+3}$ are contractible via homotopies not interacting with the centering, and so $W(\eta_0) = W((\eta_{g+3})^{-1}) = I$. We find, therefore,
	$$\wt{\partial}^{\scr{B}^+}\left(\sum_{1 \leq j<i \leq g+3} \overline{W}_iW_j\right) = \overline{W} - W + \sum_{j=1}^{g+3}\left(W(\eta_j^{-1} * \gamma_j)-W(\gamma_j^{-1} * \eta_{j-1})\right).$$
	But $\eta_j^{-1} * \gamma_j \sim \gamma_j^{-1} * \eta_{j-1}$ are both homotopic to the out-and-back paths from the base point to the center $c_j$, so the last sum cancels, and we have the desired claim.
\end{proof}

\subsection{Proof of Theorem \ref{thm:nc_CM_enlarged}}

In order to prove Theorem \ref{thm:nc_CM_enlarged}, we essentially just need to check that orientation changes commute with Moves V-VII. We include an explicit proof for completeness.

\begin{proof}[Proof of Theorem \ref{thm:nc_CM_enlarged}]
	Suppose $\zeta = (\iota,\tau) \in (\Z_2)^E \times \left( F_{g+2} \times (\Z_2 * \Z_3)\right)$, and $\Gamma$ is a homotopy class of garden. Then we may take
	$$\Phi_{\Gamma}^{\zeta \cdot \Gamma}(\tau) \colon (\wt{\scr{B}}_G^+, \wt{\partial}^{\scr{B}^+}_{G,\Gamma}) \rightarrow (\wt{\scr{B}}_G^+, \wt{\partial}^{\scr{B}^+}_{G,\tau \cdot \Gamma})$$
	as the composition
	$$\Phi_{\Gamma}^{\zeta \cdot \Gamma}(\tau) := \Phi_{\iota} \circ \Phi_{\Gamma}^{\tau},$$
	where $\Phi_{\iota}$ is as in Lemma \ref{lem:orientation_change} describe the action of changing orientations, and where $\Phi_{\Gamma}^{\tau}$ is as in Lemma \ref{lem:Moves_V-VII_together} describing the action of Moves V, VI, and VII together. We now verify the desired properties.
	\begin{itemize}
		\item \textbf{Composition Property:} It suffices to check that orientation changes commute with Moves V--VII, in which case we have that
		$$\Phi_{\Gamma}^{\zeta \cdot \Gamma}(\zeta) = \Phi_{\iota} \circ \Phi_{\Gamma}^{\tau} = \Phi_{\iota \cdot \Gamma}^{\tau} \circ \Phi_{\iota}.$$
		Once we establish this claim, the proof of the composition property is now the same as in the proof of Lemma \ref{lem:Moves_V-VII_together}, where we use the fact that Move V and Moves VI and VII commuted in the same way. We work towards establishing the claim.
		
		First, note that if $\gamma$ is a non-degenerate path, then
		$$\Phi_{\iota}(W^{\Gamma}(\gamma)) = W^{\iota \cdot \Gamma}(\gamma).$$
		Indeed, we may check this discontinuity by discontinuity for each binary sequence. Equality at centers through which $\gamma$ passes is clear because $\Phi^{\iota}\begin{pmatrix} 0 & f \\ 0 & 0 \end{pmatrix} = \begin{pmatrix} 0 & f \\ 0 & 0 \end{pmatrix}$. Similarly, for discontinuities at edges of the graph, since our sign conventions are precisely so that $A_e$ and $B_e$ come with opposite signs, the result is clear. Finally, for discontinuities at threads, we have already discussed the relevant statement in the proof of Lemma \ref{lem:orientation_change}.
		
		With this in mind, it is clear that orientation changes commute with Move V. This is clear on the ring and on the degree $1$ generators. On the degree $2$ generators, this follows from the formula of Move V in Proposition \ref{prop:Move_V_action} combined with the fact just discussed, which implies
		$$\Phi_{\iota}(W^{\Gamma}(\gamma_{f,g})) = W^{\iota \cdot \Gamma}(\gamma_{f,g})$$
		where $\gamma_{f,g} = \gamma^{\Gamma}_{f,g} = \gamma^{\iota \cdot \Gamma}_{f,g}$ since the curve is independent of the orientations of the edges of the graph.
		
		Move VI, with formula as in Lemma \ref{lem:Move_VI}, is similar, since $\wt{H}(\tau)$ appearing in the formula transforms in the same way. As for Move VII, the matrix $\begin{pmatrix} 0 & -\lambda^{-1} \\ \overline{\lambda} & 0 \end{pmatrix}$ is completely independent of the orientation of the edge, whereas applying $\Phi_{\iota}$ either preserves or negates this matrix, depending upon whether the corresponding edge has its orientation reversed. Either way, since we conjugate by this matrix, even if the orientation is reversed, the minus signs cancel, and we have that changing orientation commutes with Move VII.\\
		\item \textbf{Orientation changes:} This is just a part of Lemma \ref{lem:orientation_change}.\\
		\item \textbf{Modifying the garden:} This follows from Lemma \ref{lem:Moves_V-VII_together} and Lemma \ref{lem:Move_VI_around_face} (with the understanding that the stabilizer of $\Gamma$ is generated by the elements occurring in Lemma \ref{lem:Move_VI_around_face}, as discussed in Section \ref{ssec:Moves_VI_VII}).
	\end{itemize}
\end{proof}

\section{Removing the enlargements} \label{sec:remove_enlargements}

We now wish to prove Theorem \ref{thm:nc_CM_functoriality}, our starting point being Theorem \ref{thm:nc_CM_enlarged} of the last section. We must form the following two operations:
\begin{itemize}
	\item We must shrink the ring from $\wt{R}^{nc} = \Z[\Pi]$ to $R_T^{nc} = \Z[\Pi_T]$, using functoriality properties from $\wt{\scr{B}}^+_{G}$ (with its differentials depending upon a garden) to yield functoriality properties on $\scr{B}^+_{G,T}$ (with its differentials depending upon a garden as well as the tree $T$).
	\item In the special case that the garden $\Gamma$ is finite-type (Definition \ref{defn:finite-type_garden} below), we must further show that we can destabilize to cancel the generators $f_{\Gamma}$ and $w$.
\end{itemize}
We spend a subsection on each of these, and a final short subsection to check that we have proved Theorems \ref{thm:nc_CM} and \ref{thm:nc_CM_functoriality}.

\subsection{Shrinking the ring} \label{ssec:tree_time}

The projections $\pi_T \colon \Pi \rightarrow \Pi_T$ from Section \ref{ssec:nc-coef} induce graded algebra morphisms $\pi_T \colon \wt{\scr{B}}^+_{G} \rightarrow \scr{B}^+_{G,T}$, simply by reducing all of the coefficients. On the other hand, we have right inverses $i_T \colon \Pi_T \rightarrow \Pi$ constructed in Corollary \ref{cor:right_inverse}, which induces graded algebra morphisms $i_T \colon \scr{B}^+_{G,T} \rightarrow \wt{\scr{B}}^+_{G}$. We may now simply check that
$$\pi_T \circ \wt{\partial}^{\scr{B}^+}_{G,\Gamma} \circ i_T \circ \pi_T = \pi_T \circ \wt{\partial}^{\scr{B}^+}_{G,\Gamma}$$
by checking on generators (as an algebra) and extending to the entire algebra by using the Leibniz rule of the differential together with $\Z$-linearity. By Lemma \ref{lem:induced_morphisms}, the differentials $\wt{\partial}^{\scr{B}^+}_{G,\Gamma}$ hence induce differentials
$$\partial^{\scr{B}^+}_{G,\Gamma,T} = \pi_T \circ \wt{\partial}^{\scr{B}^+}_{G,\Gamma} \circ i_T \colon \scr{B}^+_{G,T} \rightarrow \scr{B}^+_{G,T}.$$
In practice, the new differential $\partial^{\scr{B}^+}_{G,\Gamma,T}$ acts in the same way as it usually would, but every time we see a generator $\lambda_{e,\mathfrak{o}}$ for an edge $e$ not in $T$ arising from the differential of an element of $F \cup \{x,y,z,w\}$, we simply set it equal to $1$.

Recall that we defined morphisms $\wt{\scr{C}}_T^{T'}(\gamma) \colon \Pi \rightarrow \Pi$ given on generators $\lambda_{e,\mathfrak{o}}$, which we may think of as a transverse signed path from $c_f^+$ to $c_g^-$ in $\Lambda_G$, by
$$(\wt{\scr{C}}_T^{T'}(\gamma))(\lambda_{e,\mathfrak{o}}) = \beta_T^{T'}(f,\gamma,+) \cdot \lambda_{e,\mathfrak{o}} \cdot \beta_T^{T'}(f,\gamma,-)^{-1}.$$
In turn, this induces a graded algebra isomorphism
$$\wt{\scr{C}}_T^{T'}(\gamma) \colon \wt{\scr{B}}^+_{G} \rightarrow \wt{\scr{B}}^+_{G},$$
with the caveat that it does not intertwine the differential $\wt{\partial}^{\scr{B}^+}_{G,\Gamma}$. In order to fix this, we may now consider the regeneration
$$\wt{\scr{R}}_T^{T'}(\gamma) \colon \wt{\scr{B}}^+_{G} \rightarrow \wt{\scr{B}}^+_{G}$$
defined on generators by:
\begin{align*}
	\left(\wt{\scr{R}}_T^{T'}(\gamma)\right)(f) &= \beta_T^{T'}(f,\gamma,-) \cdot f \cdot \beta_T^{T'}(f,\gamma,+)^{-1} \\
	\left(\wt{\scr{R}}_T^{T'}(\gamma)\right)(x) &= \beta_T^{T'}(f_{\Gamma},\gamma,+) \cdot x \cdot \beta_T^{T'}(f_{\Gamma},\gamma,+)^{-1} \\
	\left(\wt{\scr{R}}_T^{T'}(\gamma)\right)(y) &= \beta_T^{T'}(f_{\Gamma},\gamma,+) \cdot y \cdot \beta_T^{T'}(f_{\Gamma},\gamma,-)^{-1} \\
	\left(\wt{\scr{R}}_T^{T'}(\gamma)\right)(z) &= \beta_T^{T'}(f_{\Gamma},\gamma,-) \cdot z \cdot \beta_T^{T'}(f_{\Gamma},\gamma,-)^{-1} \\
	\left(\wt{\scr{R}}_T^{T'}(\gamma)\right)(w) &= \beta_T^{T'}(f_{\Gamma},\gamma,-) \cdot w \cdot \beta_T^{T'}(f_{\Gamma},\gamma,+)^{-1} 
\end{align*}
In other words, we simply treat $f$ as though it were a path from $c_f^-$ to $c_f^+$, and similarly, the generators in $X$ as paths from $c_{f_{\Gamma}}^{\pm}$ to $c_{f_{\Gamma}}^{\pm}$, depending upon which particular generator we are looking at. (For example, the generator $w$, corresponding to binary sequences of type $(1,0)$, are thought of as paths from $c_{f_{\Gamma}}^+$ to $c_{f_{\Gamma}}^-$.)

\begin{prop} \label{prop:changing_trees}
	For a fixed garden $\Gamma$, the graded algebra isomorphism
	$$\wt{\Phi}_T^{T'}(\gamma) := \wt{\scr{R}}_{T}^{T'}(\gamma) \circ \wt{\scr{C}}_T^{T'}(\gamma) \colon \wt{\scr{B}}^+_{G} \rightarrow \wt{\scr{B}}^+_{G}$$
	is a dg-automorphism with respect to the differential $\wt{\partial}^{\scr{B}^+}_{G,\Gamma}$. Furthermore, these satisfy the composition rule
	$$\wt{\Phi}_{T'}^{T''}(\gamma') \circ \wt{\Phi}_{T}^{T'}(\gamma) = \wt{\Phi}_T^{T''}(\gamma*\gamma').$$	
\end{prop}

\begin{proof}
	We first check that $\wt{\Phi}_T^{T'}(\gamma)$ intertwines the differential by checking this property on generators. For the degree $1$ generators $f$, we have
	\begin{align*}
		\left(\wt{\partial}^{\scr{B}^+}_{G,\Gamma} \circ \wt{\Phi}_T^{T'}(\gamma)\right)(f) &= \wt{\partial}^{\scr{B}^+}_{G,\Gamma} \left(\beta_T^{T'}(f,\gamma,-) \cdot f \cdot \beta_T^{T'}(f,\gamma,+)^{-1}\right)\\
			&= \beta_T^{T'}(f,\gamma,-) \cdot \sum_{v \in f}H(\tau_f(v)) \cdot \beta_T^{T'}(f,\gamma,+)^{-1}
	\end{align*}
	whereas
	$$(\wt{\Phi}_T^{T'}(\gamma) \circ \wt{\partial}^{\scr{B}^+}_{G,\Gamma})(f) = \wt{\scr{C}}_T^{T'}(\gamma)\left(\sum_{v \in f} H(\tau_f(v))\right).$$
	If we write out $H(\tau_f(v))$ as a (signed) product of three generators in $W$, we see that
	$$\wt{\scr{C}}_T^{T'}(\gamma)(H(\tau_f(v))) = \beta_T^{T'}(f,\gamma,-) \cdot H(\tau_f(v)) \cdot \beta_T^{T'}(f,\gamma,+)^{-1}$$
	since the internal $\beta$ terms cancel out.
	
	Similarly, for the degree $2$ generators, we have the same principle, but where we notice that our binary sequences are based at curves starting and ending at the seed, located in the face $f_{\Gamma}$. For example, for the generator $x$, we have
	$$\partial^{\scr{B}^+}_{G,\Gamma}(x) = \sum_{f} w_{1,0}(\gamma_f^s) \cdot f \cdot w_{1,1}(\gamma_f^t),$$
	and we notice correspondingly that each word comprising $w_{1,0}(\gamma_f^s)$ corresponds to a geometric element of the group $\Pi$ which originates with label $+$ in the face $f_{\Gamma}$ and similarly $w_{1,1}(\gamma_f^t)$ terminates with label $+$ in $f_{\Gamma}$.

	Finally, for the composition property, we may check it on the ring and on the generators. On the ring, we just have $\wt{\scr{C}}_T^{T'}(\gamma)$ which satisfy the desired composition property. On face generators, we have
	\begin{align*}
		(\wt{\Phi}_{T'}^{T''}(\gamma') \circ \wt{\Phi}_{T}^{T'}(\gamma))(f) &= (\wt{\Phi}_{T'}^{T''}(\gamma')) (\beta_T^{T'}(f,\gamma,-) \cdot f \cdot \beta_T^{T'}(f,\gamma,+)^{-1}) \\
			&= \left(\wt{\scr{C}}_{T'}^{T''}(\gamma')\right)(\beta_T^{T'}(f,\gamma,-)) \\
			& \qquad \qquad \cdot \left(\wt{\scr{R}}_{T'}^{T''}(\gamma)\right)(f)\\
			& \qquad \qquad \qquad \qquad \cdot \left(\wt{\scr{C}}_{T'}^{T''}(\gamma')\right)(\beta_T^{T'}(f,\gamma,+)^{-1}) \\
			&= \left(\beta_{T'}^{T''}(f,\gamma',-) \cdot \beta_T^{T'}(f,\gamma,-) \cdot \beta_{T'}^{T''}(f,\gamma',-)^{-1}\right) \\
			& \qquad \qquad \cdot \left(\beta_{T'}^{T''}(f,\gamma',-) \cdot f \cdot \beta_{T'}^{T''}(f,\gamma',+)^{-1}\right) \\
			& \qquad \qquad \qquad \qquad \cdot \left(\beta_{T'}^{T''}(f,\gamma',+) \cdot \beta_T^{T'}(f,\gamma,+) \cdot \beta_{T'}^{T''}(f,\gamma',+)^{-1}\right) \\
			&= \beta_{T}^{T''}(f,\gamma*\gamma',-) \cdot f \cdot \beta_{T}^{T''}(f,\gamma*\gamma',+)^{-1} \\
			&= (\wt{\Phi}_T^{T''}(\gamma * \gamma'))(f)
	\end{align*}
	A similar computation works for the degree $2$ generators.
\end{proof}

Upon checking that
$$\pi_{T'} \circ \wt{\Phi}_T^{T'}(\gamma) \circ i_T \circ \pi_T = \pi_{T'} \circ \wt{\Phi}_T^{T'}(\gamma)$$
on generators and hence on the entire graded algebra, we may apply Lemma \ref{lem:induced_morphisms}, so that for fixed garden $\Gamma$, the dg-isomorphisms $\wt{\Phi}_T^{T'}(\gamma)$ induce dg-isomorphisms
$$\Phi_T^{T'}(\gamma) \colon (\scr{B}^+_{G,T}, \partial^{\scr{B}^+}_{G,\Gamma,T}) \rightarrow (\scr{B}^+_{G,T'}, \partial^{\scr{B}^+}_{G,\Gamma,T'}).$$
The fact that they intertwine the differentials also follows from Lemma \ref{lem:induced_morphisms}, since the morphisms $\Phi_T^{T'}(\gamma) \circ \partial^{\scr{B}^+}_{G,\Gamma,T}$ and $\partial^{\scr{B}^+}_{G,\Gamma,T'} \circ \Phi_T^{T'}(\gamma)$ are induced by the same morphism upstairs given by $\wt{\Phi}_T^{T'}(\gamma) \circ \wt{\partial}^{\scr{B}^+}_{G,\Gamma} = \wt{\partial}^{\scr{B}^+}_{G,\Gamma} \circ \wt{\Phi}_T^{T'}(\gamma)$, and hence are equal.

\subsection{Destabilization for finite-type gardens} \label{ssec:finite-type}

Recall from Remark \ref{rmk:finite-type} that the Casals--Murphy definition of a garden excludes a tine passing through the face at infinity. We use the following slightly different definition.

\begin{defn} \label{defn:finite-type_garden}
	A garden $\Gamma$ on a trivalent plane graph $G$ is said to be of \textbf{finite-type} if the tine $\gamma_{f_\infty}$ passing through the center of the face at infinity $c_{f_{\Gamma}}$ comes either first or last as it exits the seed.
\end{defn}

Notice that if a garden is finite type, then the tine $\gamma_{f_{\infty}}$ may be homotoped so that it is completely contained inside of the face $f_{\infty}$. After this homotopy, the tine only interacts with the rest of the garden at threads and the center $c_{f_{\infty}}$. It follows that there is a unique binary sequence in $\scr{BS}(\gamma_{f_{\infty}})$, since any binary sequence must go from $0$ to $1$ as the tine passes through $f_{\infty}$, and can only go from $0$ to $1$ as it passes through threads. In other words, for a finite-type garden $\Gamma$, the differential on $\wt{\scr{B}}^+_G$ on degree $2$ generators is of the form
$$\wt{\partial}^{\scr{B}^+}_{G,\Gamma} X = \begin{pmatrix} 0 & f_{\Gamma} \\ 0 & 0 \end{pmatrix} + \sum_{f \in F_{\mathrm{fin}}} W(\gamma_f) = \begin{pmatrix} 0 & f_{\Gamma} \\ 0 & 0 \end{pmatrix} + \sum_{f \in F_{\mathrm{fin}}} W(\gamma_f^s) \cdot \begin{pmatrix} 0 & f \\ 0 & 0\end{pmatrix} \cdot W(\gamma_f^T).$$
In other words, the differentials of the generators $x$, $y$, and $z$ only see terms corresponding to the tines through the finite faces, whereas the differential of $w$ picks up $f_{\Gamma}$ for the face at infinity. This already explains algebraically why, when Casals and Murphy defined a garden, they were able to simply not include a tine through the face at infinity - they were only looking at the generators $x$, $y$, and $z$!

Projecting under $\pi_T$, we have that for a finite type garden,
$$\partial^{\scr{B}^+}_{G,\Gamma,T} X = \begin{pmatrix} 0 & f_{\Gamma} \\ 0 & 0 \end{pmatrix} + \sum_{f \in F_{\mathrm{fin}}} \pi_T\left(W(\gamma_f^s)\right) \cdot \begin{pmatrix} 0 & f \\ 0 & 0\end{pmatrix} \cdot \left(W(\gamma_f^T)\right).$$
Singling out the generator $w$,
$$\partial^{\scr{B}^+}_{G,\Gamma,T} w = f_{\Gamma}+\sum_{f \in F_{\mathrm{fin}}} \pi_T(w_{0,0}(\gamma_f^s)) \cdot f \cdot \pi_T(w_{1,1}(\gamma_f^t)).$$

Consider now the elementary automorphism $\mu_{G,\Gamma,T}$ of $\scr{B}^+_{G,T}$ given by
$$\mu_{G,\Gamma,T}(f_{\Gamma}) = f_{\Gamma} - \sum_{f \in F_{\mathrm{fin}}} \pi_T(w_{0,0}(\gamma_f^s)) \cdot f \cdot \pi_T(w_{1,1}(\gamma_f^t))$$
(and the identity on all other generators), inducing a dg-isomorphism
$$\mu_{G,\Gamma,T} \colon (\scr{B}^+_{G,T},\partial^{\scr{B}^+}_{G,\Gamma,T}) \xrightarrow{\sim} (\scr{B}^+_{G,T}, (\partial^{\scr{B}^+}_{G,\Gamma,T})'),$$
with induced differential the same on all generators except $w$ and $f_{\Gamma}$, for which
$$(\partial^{\scr{B}^+}_{G,\Gamma,T})'w = f_{\Gamma}$$
$$(\partial^{\scr{B}^+}_{G,\Gamma,T})'f_{\Gamma} = 0$$
We may therefore destabilize, yielding the desired dg-algebra $(\scr{B}_{G,f_{\Gamma},T},\partial^{\scr{B}}_{G,\Gamma,T})$. Explicitly, the formula for $\partial^{\scr{B}}_{G,\Gamma,T}$ on the generators in $F_{\mathrm{fin}}$ and $x$, $y$, and $z$ is just the same as the formula for $\partial^{\scr{B}^+}_{G,\Gamma,T}$, where we notice that none of these differentials include a term $f_{\Gamma}$ or $w$ by inspection.

\subsection{Proof of Theorems \ref{thm:nc_CM} and \ref{thm:nc_CM_functoriality}}

We now have all of the ingredients to prove our desired theorems.

\begin{proof}[Proof of Theorem \ref{thm:nc_CM_functoriality}]
	We have constructed the dg-algebras $(\scr{B}_{G,T}^+,\partial^{\scr{B}^+}_{G,\Gamma,T})$ in Section \ref{ssec:tree_time}. For a fixed garden, we have discussed changing trees in the form of the morphisms $\wt{\Phi}_{T}^{T'}(\gamma)$ in Proposition \ref{prop:changing_trees} and its induced morphisms $\Phi_T^{T'}(\gamma)$ downstairs. On the other hand, by Lemma \ref{lem:induced_morphisms} (along with the usual check that the lemma can be applied), the morphisms $\Phi_{\Gamma}^{\Gamma'}(\zeta)$ constructed in Section \ref{thm:nc_CM_enlarged} via the actions of Moves V--VII, where $\Gamma' = \zeta \cdot \Gamma$, descend from $\wt{\scr{B}}_G^+$ under $\pi_T$, i.e. we obtain
	$$\Phi_{\Gamma,T}^{\Gamma',T}(\zeta) \colon (\scr{B}^+_{G,T},\partial^{\scr{B}^+}_{G,\Gamma,T}) \rightarrow (\scr{B}^+_{G,T},\partial^{\scr{B}^+}_{G,\Gamma',T})$$
	arising as the unique graded algebra morphism fitting into the diagram
	$$\xymatrix{\wt{\scr{B}}_G^+ \ar[rr]^{\Phi_{\Gamma}^{\Gamma'}(\zeta)} \ar[d]_{\pi_T} & & \wt{\scr{B}}_{G}^+ \ar[d]^{\pi_T} \\ \scr{B}^+_{G,T} \ar@{-->}[rr]_{\Phi_{\Gamma,T}^{\Gamma',T}(\zeta)} & & \scr{B}^+_{G,T}}$$
	By Lemma \ref{lem:induced_morphisms}, the fact that it intertwines the differentials follows from the same property upstairs. We may therefore define
	$$\Phi_{\Gamma,T}^{\Gamma',T'}(\zeta,\gamma) := \Phi_T^{T'}(\gamma) \circ \Phi_{\Gamma,T}^{\Gamma',T}(\zeta).$$
	We now check the desired properties.
	\begin{itemize}
		\item \textbf{Composition property:} Just like in the proof of Theorem \ref{thm:nc_CM_enlarged}, it suffices to check that the two factors $\Phi_{\Gamma,T}^{\Gamma',T}(\zeta)$ and $\Phi_T^{T'}(\gamma)$ commute with each other, in the sense that
		$$\Phi_T^{T'}(\gamma) \circ \Phi_{\Gamma,T}^{\Gamma',T}(\zeta) = \Phi_{\Gamma,T'}^{\Gamma',T'}(\zeta) \circ \Phi_T^{T'}(\gamma),$$
		since if this is satisfied, then the composition property of Theorem \ref{thm:nc_CM_enlarged} and the composition property of Proposition \ref{prop:changing_trees} yield the desired result. It suffices to check this composition property upstairs, i.e. that
		$$\wt{\Phi}_T^{T'}(\gamma) \circ \Phi_{\Gamma}^{\Gamma'}(\zeta) = \Phi_{\Gamma}^{\Gamma'}(\zeta) \circ \wt{\Phi}_T^{T'}(\gamma).$$
		Because $\Phi_{\Gamma}^{\Gamma'}(\zeta)$ is itself constructed as a composition of orientation changes and Moves V--VII, it suffices to check this commutativity for each of these types of moves, each of which is easy to check. For example, if $\zeta$ acts via an application of Move V, then since it preserves all degree $1$ generators, it suffices to check commutativity on the degree $2$ generators. That is, if the application of Move V is with respect to the faces $f$ and $g$, we have, for example
		\begin{align*}
			(\wt{\Phi}_T^{T'}(\gamma) \circ  \Phi_{\Gamma}^{\Gamma'}(\zeta))(x) &= \wt{\Phi}_T^{T'}(\gamma) \left(x+w_{1,1}(\gamma_{f,g}^{\Gamma})\right) \\ 
				&=(\wt{\Phi}_T^{T'}(\gamma)) (x)+(\wt{\Phi}_T^{T'}(\gamma))w_{1,1}(\gamma_{f,g}^{\Gamma}) \\
				&= \beta_T^{T'}(f_{\Gamma},\gamma,+) \cdot x \cdot \beta_T^{T'}(f_{\Gamma},\gamma,+)^{-1} \\
				& \qquad \qquad + \beta_T^{T'}(f_{\Gamma},\gamma,+) \cdot w_{1,1}(\gamma^{\Gamma}_{f,g}) \cdot \beta_T^{T'}(f_{\Gamma},\gamma,+)^{-1} \\
				&= \beta_T^{T'}(f_{\Gamma},\gamma,+) \cdot (x+w_{1,1}(\gamma_{f,g}^{\Gamma})) \cdot  \beta_T^{T'}(f_{\Gamma},\gamma,+)^{-1} \\
				&= \beta_T^{T'}(f_{\Gamma},\gamma,+) \cdot ((\Phi_{\Gamma}^{\Gamma'})(\zeta))(x) \cdot  \beta_T^{T'}(f_{\Gamma},\gamma,+)^{-1} \\
				&= (\Phi_{\Gamma}^{\Gamma'}(\zeta))(\beta_T^{T'}(f_{\Gamma},\gamma,+) \cdot x \cdot \beta_T^{T'}(f_{\Gamma},\gamma,+)^{-1}) \\
				&= (\Phi_{\Gamma}^{\Gamma'}(\zeta) \circ \wt{\Phi}_T^{T'}(\gamma))(x).
		\end{align*}
		The other generators $y$, $z$, $w$ are analogous, and we leave Moves VI and VII for the reader.\\
		\item \textbf{Orientation changes:} This descends from the same property in Theorem \ref{thm:nc_CM_enlarged} under Lemma \ref{lem:induced_morphisms}.\\
		\item \textbf{Fixing the trees and capping paths; modifying the garden:} This also descends from the enlarged property.\\
		\item \textbf{Fixing the garden; modifying the trees and capping paths:} This follows from Proposition \ref{prop:changing_trees}.\\
		\item \textbf{Finite-type gardens:} We have defined the differential on $\partial^{\scr{B}}_{G,\Gamma,T}$ precisely in this way in Section \ref{ssec:finite-type}.
	\end{itemize}
\end{proof}

\begin{proof}[Proof of Theorem \ref{thm:nc_CM}]
	We have already found the desired differentials yielding the dg-algebras $(\scr{B}_{G,f_{\Gamma},T},\partial^{\scr{B}}_{G,\Gamma,T})$. That these partially abelianize to the Casals--Murphy dg-algebras is clear by comparing our construction with their construction.
	
	Suppose $\Gamma$ and $\Gamma'$ are two finite-type gardens and $T,T' \subset G$ are two trees spanning all but one vertex. For any homotopy class of path $\gamma$ from $*_T$ to $*_{T'}$ and element $\zeta \in \scr{H}(G)$ such that $\zeta \cdot \Gamma = \Gamma'$, we then have a quasi-isomorphism induced under destabilization as in Lemma \ref{lem:induced_morphisms} as in the following diagram:	
	$$\xymatrix{(\scr{B}_{G,f_{\Gamma},T}, \partial^{\scr{B}}_{G,\Gamma,T}) \ar@{-->}[ddd] & (\scr{B}^+_{G,T},(\partial^{\scr{B}^+}_{G,\Gamma,T})') \ar[d]^-{\mu_{G,\Gamma,T}^{-1}} \ar[l]^{\mathrm{destab}} \\
		& (\scr{B}^+_{G,T},\partial^{\scr{B}^+}_{G,\Gamma,T}) \ar[d]^{\Phi_{\Gamma,T}^{\Gamma',T'}(\gamma)} \\ 
		& (\scr{B}^+_{G,T'},\partial^{\scr{B}^+}_{G,\Gamma',T'}) \ar[d]^-{\mu_{G,\Gamma',T'}}	\\
		(\scr{B}_{G,f_{\Gamma'},T'}, \partial^{\scr{B}}_{G,\Gamma',T'}) & (\scr{B}^+_{G,T'},(\partial^{\scr{B}^+}_{G,\Gamma',T'})') \ar[l]^{\mathrm{destab}} }$$
	This is clearly a dg-isomorphism, since its inverse is constructed in the same way but swapping $(\Gamma,T)$ with $(\Gamma',T')$ and using $\gamma^{-1}$ in place of $\gamma$.
\end{proof}

\section{Representations and rank $r$ face colorings}

The goal of this section is to prove Theorem \ref{thm:reps_are_sheaves}.

\subsection{Representations-to-colors for the enlarged dg-algebra}

As with functoriality between the various dg-algebras $(\scr{B}^+_{G,T},\partial^{\scr{B}^+}_{G,\Gamma,T})$, when studying representations, it is convenient to begin with $(\wt{\scr{B}}_G^+,\wt{\partial}^{\scr{B}^+}_{G,\Gamma})$, and to therefore verify a version of Theorem \ref{thm:reps_are_sheaves} but with respect to the dg-isomorphism $\Phi_{\Gamma}^{\Gamma'}(\zeta)$ appearing in Theorem \ref{thm:nc_CM_enlarged}. Let us therefore begin with a study of $\mathrm{Rep}_{r}(\wt{\scr{B}}_G^+,\wt{\partial}^{\scr{B}^+}_{G,\Gamma};\F)$, and see how we can extract from it a rank $r$ face coloring.

For degree reasons, a representation $\epsilon \in \mathrm{Rep}_r(\wt{\scr{B}}_G^+,\wt{\partial}^{\scr{B}^+}_{G,\Gamma};\F)$ kills all of the generators in nonzero degree, and hence, we may identify $\epsilon$ with a representation of $\Pi$, which we think of as a ring homomorphism $$\epsilon_\Pi \colon \Z[\Pi] \rightarrow \mathrm{Mat}_{r \times r}(\F).$$ That $\epsilon$ is a representation of the dg-algebra is just the requirement that
$$\epsilon_\Pi(\wt{\partial}^{\scr{B}^+}_{G,\Gamma} f) = 0$$
for all $f \in F$. Writing this out, we are simply looking at representations of $\Pi$ with the property that
$$\sum_{v \in f} \epsilon_\Pi(H(\tau_f(v)) = 0.$$

Recall that associated to any path $\gamma$ on the sphere $S^2$, non-degenerate with respect to the graph $G$ and its web, there is an associated word $W(\gamma) \in \mathrm{Mat}_{2 \times 2}(\Z[\Pi]*\Z\langle F \rangle)$. If the path $\gamma$ furthermore is disjoint from the face centers, then all of the faces disappear from the expression, and $W(\gamma) \in \mathrm{Mat}_{2 \times 2}(\Z[\Pi])$, in which case $W(\gamma)$ is a composition of matrices either of one of the following two forms:
$$\begin{pmatrix}1 & H \\ 0 & 1\end{pmatrix} \qquad \pm \begin{pmatrix} 0 & \lambda \\ -\overline{\lambda}^{-1} & 0 \end{pmatrix}$$
Here, $H$ is a monomial in $\Pi$, with sign $\pm 1$, and $\lambda$ and $\overline{\lambda}$ are just an elements of $\Pi$. Hence, under $\epsilon_\Pi$, the matrices
$$\epsilon_\Pi\begin{pmatrix}1 & H \\ 0 & 1\end{pmatrix} = \begin{pmatrix}I & \epsilon_\Pi(H) \\ 0 & I\end{pmatrix}$$
$$\epsilon_\Pi\begin{pmatrix} 0 & \lambda \\ -\overline{\lambda}^{-1} & 0 \end{pmatrix} = \begin{pmatrix} 0 & \epsilon_\Pi(\lambda) \\ -\epsilon_\Pi(\overline{\lambda})^{-1} & 0 \end{pmatrix}$$
are both invertible matrices, i.e. they live in $\mathrm{GL}_{2r}(\F)$.

Suppose $\gamma$ is a small loop around $c_f$. Then
$$W(\gamma) = \begin{pmatrix}1 & \wt{\partial}^{\scr{B}^+}_{G,\Gamma}f \\ 0 & 1 \end{pmatrix}.$$
The condition that $\epsilon$ is a representation implies that $\epsilon_\Pi(W(\gamma)) = I_{2r} \in \mathrm{GL}_{2r}(\F)$. Hence, even though $W(\gamma)$ is invariant only up to homotopies (with fixed endpoints) not passing through the centering, the matrix $\epsilon_\Pi(W(\gamma))$ is now independent under all homotopies (with fixed endpoints). Since our paths lie on $S^2$, any two paths with the same endpoints are homotopic, so the corresponding element $\epsilon_\Pi(W(\gamma)) \in \mathrm{GL}_{2r}(\F)$ only depends upon the endpoints of $\gamma$. Furthermore, two points in the same triangle $\mathrm{Tri}(G)$ are essentially indistinguishable. We therefore obtain, for any two triangles $\nu_1,\nu_2 \in \mathrm{Tri}(G)$, an element
$$M^{\epsilon}(\nu_1,\nu_2) \in \mathrm{GL}_{2r}(\F).$$
These satisfy the following properties:
\begin{itemize}
	\item For all $\nu_1,\nu_2,\nu_3 \in \mathrm{Tri}(G)$, we have $$M^{\epsilon}(\nu_1,\nu_2) \cdot M^{\epsilon}(\nu_2,\nu_3) = M^{\epsilon}(\nu_1,\nu_3).$$
	\item For all $\nu \in \mathrm{Tri}(G)$, we have $M^{\epsilon}(\nu,\nu) = I_{2r}$.
\end{itemize}

Recall now that as part of the data of a garden $\Gamma$, we have a base point, which lives in one of the triangles $\nu_{\Gamma} \in \mathrm{Tri}(G)$. Hence, for each $\nu \in \mathrm{Tri}(G)$, we may define
$$M^{\epsilon,\Gamma}(\nu) := M^{\epsilon}(\nu_{\Gamma},\nu) \in \mathrm{GL}_{2r}(\F).$$

Consider the $2r \times r$ matrix $\begin{pmatrix} I \\ 0 \end{pmatrix}$. The product $M^{\epsilon,\Gamma}(\nu) \cdot \begin{pmatrix}I \\ 0 \end{pmatrix}$ is the $2r \times r$ matrix consisting of the first $r$ columns of $M^{\epsilon,\Gamma}(\nu)$. Suppose that $\nu_1$ and $\nu_2$ are two triangles sharing an edge along a thread of the graph. Then we have that for some $H \in \mathrm{GL}_r(\F)$,
\begin{align*}
	M^{\epsilon,\Gamma}(\nu_2) \cdot \begin{pmatrix}I \\ 0\end{pmatrix} &= M^{\epsilon,\Gamma}(\nu_1) \cdot M^{\epsilon}(\nu_1,\nu_2) \cdot \begin{pmatrix}I \\ 0\end{pmatrix} \\
		&=  M^{\epsilon,\Gamma}(\nu_1) \cdot \begin{pmatrix}I & H \\ 0 & I \end{pmatrix} \cdot \begin{pmatrix}I \\ 0\end{pmatrix} \\
		&= M^{\epsilon,\Gamma}(\nu_1) \cdot \begin{pmatrix}I \\ 0\end{pmatrix}.
\end{align*}
Therefore, for each $f \in F$, we obtain a canonical $2r \times r$ matrix $M^{\epsilon,\Gamma}(f)$ given as
$$M^{\epsilon,\Gamma}(f) := M^{\epsilon,\Gamma}(\nu) \cdot \begin{pmatrix}I \\ 0\end{pmatrix}$$
for any triangle $\nu \in \mathrm{Tri}(G)$ contained in $f$. In fact, this matrix is of rank $r$, on account of the fact that the columns of $M^{\epsilon,\Gamma}(\nu)$ are linearly independent, a fortiori the first $r$ columns are linearly independent. Taking the span of the columns of $M^{\epsilon,\Gamma}(f)$ yields an element $\chi^{\epsilon,\Gamma}(f) \in \mathrm{Gr}(r,2r;\F)$.

Similarly, suppose that $\nu_1,\nu_2 \in \mathrm{Tri}(G)$ are two triangles sharing an edge which is also an edge of the graph. Let us write $f_1,f_2 \in F$ for the two faces containing $\nu_1$ and $\nu_2$ respectively. Then we have that there are some $\Lambda,\Lambda' \in \mathrm{GL}_r(\F)$ such that
\begin{align*}
	M^{\epsilon,\Gamma}(f_2) &=	M^{\epsilon,\Gamma}(\nu_2) \cdot \begin{pmatrix}I \\ 0\end{pmatrix} \\
	&= M^{\epsilon,\Gamma}(\nu_1) \cdot M^{\epsilon}(\nu_1,\nu_2) \cdot \begin{pmatrix}I \\ 0\end{pmatrix} \\
	&=  M^{\epsilon,\Gamma}(\nu_1) \cdot \begin{pmatrix}0 & \Lambda' \\ \Lambda & 0 \end{pmatrix} \cdot \begin{pmatrix}I \\ 0\end{pmatrix} \\
	&= M^{\epsilon,\Gamma}(\nu_1) \cdot \begin{pmatrix}0 \\ \Lambda\end{pmatrix} \\
	&= M^{\epsilon,\Gamma}(\nu_1) \cdot \begin{pmatrix}0 \\ I \end{pmatrix} \cdot \Lambda \\
\end{align*}
Notice that $M^{\epsilon,\Gamma}(\nu_1) \cdot \begin{pmatrix}0 \\ I \end{pmatrix}$ consists of the last $r$ column vectors of $M^{\epsilon,\Gamma}(\nu_1)$, which are linearly independent with span transverse to the span of the first $r$ columns, i.e. $\chi^{\epsilon,\Gamma}(f_1)$. Multiplying by $\Lambda$ on the right does not change the span. Therefore, we see that $\chi^{\epsilon,\Gamma}(f_1)$ and $\chi^{\epsilon,\Gamma}(f_2)$ are transverse.

Let us be explicit about what our construction has given. In the first step of the construction, we had a morphism
$$M^{\Gamma} \colon \mathrm{Rep}_r((\wt{\scr{B}}^+_G,\wt{\partial}^{\scr{B}^+}_{G,\Gamma});\F) \rightarrow \left(\mathrm{GL}_{2r}(\F)\right)^{\mathrm{Tri}(G)}.$$
To each element of $\mathrm{GL}_{2r}(\F)$, we then multiply on the right by $\begin{pmatrix}I \\ 0 \end{pmatrix}$ to obtain, for each $\nu \in \mathrm{Tri}(G)$, an element of $\mathrm{Mat}_{2r \times r}^{\mathrm{\mathrm{rk}=r}}(\F)$, by which we mean a $2r \times r$ matrices of rank $r$, yielding the composition
$$\mathrm{Rep}_r(\wt{\scr{B}}^+_G,\wt{\partial}^{\scr{B}^+}_{G,\Gamma};\F) \xrightarrow{M^{\Gamma}} \left(\mathrm{GL}_{2r}(\F)\right)^{\mathrm{Tri}(G)} \xrightarrow{\bullet \begin{pmatrix}I \\ 0 \end{pmatrix}} \left(\mathrm{Mat}_{2r \times r}^{\mathrm{\mathrm{rk}=r}}(\F)\right)^{\mathrm{Tri}(G)}.$$
The image of this composition has the property that it lands inside of $\left(\mathrm{Mat}_{2r \times r}^{\mathrm{\mathrm{rk}=r}}(\F)\right)^{F} \subset \left(\mathrm{Mat}_{2r \times r}^{\mathrm{\mathrm{rk}=r}}(\F)\right)^{\mathrm{Tri}(G)}$, those elements for which all triangles in the same face have the same matrix. Taking the span of the column vectors yields a map
$$\mathrm{span} \colon \mathrm{Mat}_{2r \times r}^{\mathrm{rk}=r}(\F) \rightarrow \mathrm{Gr}(r,2r;\F),$$
which we may now apply to each face to yield a morphism
$$\mathrm{Span} \colon \left(\mathrm{Mat}_{2r \times r}^{\mathrm{rk}=r}(\F)\right)^F \rightarrow \left(\mathrm{Gr}(r,2r;\F)\right)^F.$$
If we trace this through, starting at a representation, we end up with an element of $\mathrm{Col}_r(G;\F)$.

Summarizing, we have therefore constructed a map
$$\wt{\Psi}_{G,\Gamma} \colon \mathrm{Rep}_r(\wt{\scr{B}}^+_{G},\wt{\partial}^{\scr{B}^+}_{G,\Gamma}) \rightarrow \mathrm{Col}_r(G;\F)$$
given by
$$\wt{\Psi}_{G,\Gamma}(\epsilon) := (\chi^{\epsilon,\Gamma}(f))_{f \in F}.$$
which fits into the following diagram, where the dashed arrows indicate that these exist (and are obviously uniquely determined):
$$\xymatrix{& \left(\mathrm{GL}_{2r}(\F)\right)^{\mathrm{Tri}(G)} \ar[r]^{\bullet \begin{pmatrix}I \\ 0\end{pmatrix}} & \left(\mathrm{Mat}_{2r \times r}^{\mathrm{rk}=r}(\F)\right)^{\mathrm{Tri}(G)} & \\ \mathrm{Rep}_r(\wt{\scr{B}}^+_G,\wt{\partial}^{\scr{B}^+}_{G,\Gamma};\F) \ar@{-->}[rr] \ar@{-->}@/_2pc/[rrrd]_{\wt{\Psi}_{G,\Gamma}} \ar[ur]^{M^{\Gamma}} & & \left(\mathrm{Mat}_{2r \times r}^{\mathrm{rk}=r}(\F)\right)^F \ar@{^{(}->}[u] \ar[r]^{\mathrm{Span}} & \left(\mathrm{Gr}(r,2r;\F)\right)^F \\ & & & \mathrm{Col}_r(G;\F) \ar@{^{(}->}[u]}$$

Let us trace the conjugation action of $\mathrm{GL}_r(\F)$ on $\mathrm{Rep}_r(\wt{\scr{B}}^+_G,\wt{\partial}^{\scr{B}^+}_{G,\Gamma};\F)$ through this diagram. Suppose $g \in \mathrm{GL}_r(\F)$. Then by a simple computation,
$$(g \cdot \epsilon)_\Pi\begin{pmatrix}1 & H \\ 0 & 1 \end{pmatrix} = \begin{pmatrix}g & 0 \\ 0 & g\end{pmatrix} \begin{pmatrix}1 & \epsilon_\Pi(H) \\ 0 & 1 \end{pmatrix}\begin{pmatrix}g^{-1} & 0 \\ 0 & g^{-1}\end{pmatrix}$$
and
$$(g \cdot \epsilon)_\Pi\begin{pmatrix}0 & \lambda \\ -\overline{\lambda}^{-1} & 0 \end{pmatrix} = \begin{pmatrix}g & 0 \\ 0 & g\end{pmatrix} \begin{pmatrix}0 & \epsilon_{\Pi}(\lambda) \\ -\epsilon_{\Pi}(\overline{\lambda})^{-1} & 0 \end{pmatrix}\begin{pmatrix}g^{-1} & 0 \\ 0 & g^{-1}\end{pmatrix}.$$
We find, therefore, that since our map $M^{\Gamma}$ was given as a composition of these matrices, that,
$$M^{g \cdot \epsilon,\Gamma}(\nu) = \begin{pmatrix}g & 0 \\ 0 & g\end{pmatrix}M^{\epsilon,\Gamma}(\nu)\begin{pmatrix}g^{-1} & 0 \\ 0 & g^{-1}\end{pmatrix}.$$
In other words $M^{\Gamma}$ intertwines the $\mathrm{GL}_r(\F)$ action on representations with simultaneous conjugation by the matrix $\begin{pmatrix}g & 0 \\ 0 & g\end{pmatrix}$. Multiplying by the matrix $\begin{pmatrix}I \\ 0\end{pmatrix}$ on the right, we see that the unlabelled dotted map in the previous diagram intertwines the $\mathrm{GL}_r(\F)$-conjugation on representations with the action of $\mathrm{GL}_r(\F)$ on $(\mathrm{Mat}^{\mathrm{rk}=r}_{2r \times r}(\F))^F$ given by simultaneous action on each component by
$$g \cdot M = \begin{pmatrix}g & 0 \\ 0 & g \end{pmatrix} \cdot M \cdot g^{-1}$$
for $M \in \mathrm{Mat}^{\mathrm{rk}=r}_{2r \times r}(\F)$. Finally, applying $\mathrm{Span}$, we find that the multiplication by $g^{-1}$ on the right in the above expression does not change the span, and we are hence left with the fact that $\wt{\Psi}_{G,\Gamma}$ intertwines the conjugation action of $\mathrm{GL}_r(\F)$ with the action of $\mathrm{PGL}_{2r}(\F)$ on $\mathrm{Col}_r(G;\F)$ via the map from $\mathrm{GL}_r(\F)$ to $\mathrm{PGL}_{2r}(\F)$ given by
$$g \mapsto \begin{bmatrix}g & 0 \\ 0 & g\end{bmatrix}.$$

In conclusion, because $\wt{\Psi}_{G,\Gamma}$ has this compatibility between the conjugation action on representations by elements of $\mathrm{GL}_r(\F)$ and the action on colorings by elements of $\mathrm{PGL}_{2r}(\F)$, it descends to the corresponding moduli spaces as
$$\Psi_{G,\Gamma} \colon \scr{M}^{\mathrm{Rep}}_r(\wt{\scr{B}}^+_G,\wt{\partial}^{\scr{B}^+}_{G,\Gamma};\F) \rightarrow \scr{M}^{\mathrm{Col}}_r(G;\F).$$

\subsection{Functoriality for the representations-to-colors map}

Notice that the maps $\bullet \begin{pmatrix}I\\0\end{pmatrix}$ and $\mathrm{Span}$ both intertwine a natural simultaneous left action of $\mathrm{GL}_{2r}(\F)$ which reduces to the desired $\mathrm{PGL}_{2r}(\F)$-action on $\mathrm{Color}_r(G;\F)$. We wish to intertwine the functoriality properties for the dg-algebra $(\wt{\scr{B}}^+_{G},\wt{\partial}^{\scr{B}^+}_{G,\Gamma})$ appearing in Theorem \ref{thm:nc_CM_enlarged} with the $\mathrm{PGL}_{2r}(\F)$-action on $\mathrm{Color}_r(G;\F)$.

Suppose $\zeta \in \scr{H}(G) = (\Z_2)^E \times F_{g+2} \times (\Z_2 * \Z_3)$ and let $\Gamma$ and $\Gamma'$ be two gardens with $\Gamma' = \zeta \cdot \Gamma$. Then we have maps
$$\Phi_{\Gamma}^{\Gamma'}(\zeta) \colon (\wt{\scr{B}}^+_{G},\wt{\partial}^{\scr{B}^+}_{G,\Gamma}) \rightarrow (\wt{\scr{B}}^+_{G},\wt{\partial}^{\scr{B}^+}_{G,\Gamma'}),$$
inducing pullbacks
$$(\Phi_{\Gamma}^{\Gamma'}(\zeta))^* \colon \mathrm{Rep}_r(\wt{\scr{B}}^+_{G},\wt{\partial}^{\scr{B}^+}_{G,\Gamma'};\F) \rightarrow \mathrm{Rep}_r(\wt{\scr{B}}^+_{G},\wt{\partial}^{\scr{B}^+}_{G,\Gamma};\F).$$
The following proposition yields the desired functoriality for the maps $\Psi_{G,\Gamma}$.

\begin{prop}
	For each $\Gamma,\Gamma',\zeta$ as above, we have
	$$\Psi_{G,\Gamma'} = \Psi_{G,\Gamma} \circ (\Phi_{\Gamma}^{\Gamma'}(\zeta))^*.$$
\end{prop}

\begin{proof}
	Let us be more explicit about what we wish to prove. Suppose we are given
	$$\epsilon \in \mathrm{Rep}_r(\wt{\scr{B}}^+_G,\wt{\partial}^{\scr{B}^+}_{G,\Gamma'};\F).$$
	Then we have
	$$\begin{cases} M^{\Gamma'}(\epsilon) \in (\mathrm{GL}_{2r}(\F))^{\mathrm{Tri}(G)}\\ M^{\Gamma}((\Phi_{\Gamma}^{\Gamma'}(\zeta))^*\epsilon) \in (\mathrm{GL}_{2r}(\F))^{\mathrm{Tri}(G)}\end{cases}.$$
	It suffices to find an element $Z_{\Gamma}^{\Gamma'}(\zeta,\epsilon) \in \mathrm{GL}_{2r}(\F)$ such that
	$$Z_{\Gamma}^{\Gamma'}(\zeta,\epsilon) \cdot M^{\Gamma}((\Phi_{\Gamma}^{\Gamma'}(\zeta))^*\epsilon) = M^{\Gamma'}(\epsilon),$$
	since then we see that
	$$\wt{\Psi}_{G,\Gamma'}(\epsilon) = [Z_{\Gamma}^{\Gamma'}(\zeta,\epsilon)] \cdot \wt{\Psi}_{G,\Gamma'}((\Phi_{\Gamma}^{\Gamma'}(\zeta))^*\epsilon),$$
	where $[Z_{\Gamma}^{\Gamma'}(\zeta,\epsilon)]$ is the image of $Z_{\Gamma}^{\Gamma'}(\zeta,\epsilon)$ under the projection $\mathrm{GL}_{2r}(\F) \rightarrow \mathrm{PGL}_{2r}(\F)$.

	Recall that $\Phi_{\Gamma}^{\Gamma'}(\gamma)$ arises via a composition of three types of behavior: orientation changes, changing the tines (Move V), and moving the base-point (Moves VI-VII). It suffices, therefore, to study each type of behavior, since we may simply compose the various $Z_{\Gamma}^{\Gamma'}(\zeta,\epsilon)$ under compositions. Let us check each of these individually.
		
	\begin{itemize}
		\item \textbf{Orientation changes:} We have already seen how orientation changes affect the elements $H(\tau)$ (the proof of Lemma \ref{lem:orientation_change}) and $\lambda_{e,\mathfrak{o}}$ (by definition); it follows that we may take each
		$$Z_{\Gamma}^{\Gamma'}(\zeta,\epsilon) = I$$
		for any $\epsilon$ when $\zeta \in (\Z_2)^E$.\\
		\item \textbf{Tine changes (Move V):} Tine changes only affect the differentials of the generators of degree $2$, and hence do not affect the representations at all. Again, we may take
		$$Z_{\Gamma}^{\Gamma'}(\zeta,\epsilon)=I$$
		in this setting as well.\\
		\item \textbf{Moving the base-point (Moves VI-VII):} Move VI or VII simply moves the base-point through a tine or thread. In particular, it follows that we may take
		$$Z_{\Gamma}^{\Gamma'}(\zeta,\epsilon) := M^{\epsilon}(\nu_{\Gamma'},\nu_{\Gamma}) = M^{\Phi^*\epsilon}(\nu_{\Gamma'},\nu_{\Gamma}),$$
		where the two matrices on the right are the same because the representations $\epsilon$ and $\Phi^*\epsilon$ take the same values on elements of the coefficient ring because $\Phi$ preserves the degree $1$ generators (i.e. the maps $\epsilon_\Pi$ and $(\Phi^*\epsilon)_\Pi$ are the same, as ring homomorphisms from $\Z[\Pi]$ to $\mathrm{Mat}_{r \times r}(\F)$).
	\end{itemize}
\end{proof}

Given a tree $T$, we have a projection $\pi_T \colon (\wt{\scr{B}}^+_{G},\wt{\partial}^{\scr{B}^+}_{G,\Gamma}) \rightarrow  (\scr{B}^+_{G,T},\partial^{\scr{B}^+}_{G,\Gamma,T})$ inducing pullbacks
$$(\pi_{T})^* \colon \mathrm{Rep}_r(\scr{B}^+_{G,T},\partial^{\scr{B}^+}_{G,\Gamma,T};\F)) \rightarrow \mathrm{Rep}_r(\wt{\scr{B}}^+_{G},\wt{\partial}^{\scr{B}^+}_{G,\Gamma};\F).$$
and
$$(\pi_{T})^* \colon \scr{M}^{\mathrm{Rep}}_r(\scr{B}^+_{G,T},\partial^{\scr{B}^+}_{G,\Gamma,T};\F)) \rightarrow \scr{M}^{\mathrm{Rep}}_r(\wt{\scr{B}}^+_{G},\wt{\partial}^{\scr{B}^+}_{G,\Gamma};\F).$$
In particular, we obtain the maps
$$\wt{\Psi}_{G,\Gamma,T} := \wt{\Psi}_{G,\Gamma} \circ (\pi_T)^*$$
and
$$\Psi_{G,\Gamma,T} := \Psi_{G,\Gamma} \circ (\pi_T)^*,$$
the latter of which is the one occurring in Theorem \ref{thm:reps_are_sheaves}.

Recall that if we wish to change the tree, then we have dg-isomorphisms
$$\Phi_{T}^{T'}(\gamma) \colon (\scr{B}^+_{G,T},\partial^{\scr{B}^+}_{G,\Gamma,T}) \rightarrow (\scr{B}^+_{G,T'},\partial^{\scr{B}^+}_{G,\Gamma,T'}).$$
We will prove functoriality with respect to these transformations, which will easily allow us to project to the version in which we have included the tree.

\begin{lem}
	For any fixed garden $\Gamma$ on $G$ and any homotopy class of path $\gamma$ from $*_T$ to $*_{T'}$ in $\Lambda_G$, we have
	$$(\wt{\Phi}_{T}^{T'}(\gamma))^* \colon \mathrm{Rep}_r(\wt{\scr{B}}^+_{G},\wt{\partial}^{\scr{B}^+}_{G,\Gamma};\F) \rightarrow \mathrm{Rep}_r(\wt{\scr{B}}^+_{G},\wt{\partial}^{\scr{B}^+}_{G,\Gamma};\F)$$
	is just the identity.
\end{lem}

\begin{proof}
	Recall that $\wt{\Phi}_T^{T'}(\gamma)$ acts on degree $0$ generators as $\wt{\scr{C}}_T^{T'}(\gamma)$, given by
	$$(\wt{\scr{C}}_T^{T'}(\gamma))(\lambda_{e,\mathfrak{o}}) = \beta_T^{T'}(f,\gamma,+) \cdot \lambda_{e,\mathfrak{o}} \cdot \beta_T^{T'}(f,\gamma,-))^{-1}.$$
	Recall further that the words $\beta$ were associated to certain loops of paths on $\Lambda_G$. In particular, the terms $\beta$ appear as the diagonal terms in the matrix of words for these loops. Upon applying any representation $\epsilon$, we find that $\epsilon_\Pi(\beta) = I$. It follows now that if $\epsilon \in \mathrm{Rep}_r(\wt{\scr{B}}^+_{G},\wt{\partial}^{\scr{B}^+}_{G,\Gamma};\F)$ and $\lambda_{e,\mathfrak{o}} \in \Pi$, then
	$$\epsilon_\Pi(\lambda_{e,\mathfrak{o}}) = \epsilon_\Pi((\wt{\scr{C}}_T^{T'}(\gamma))(\lambda_{e,\mathfrak{o}})).$$
	Hence, $\epsilon = (\wt{\Phi}_T^{T'}(\gamma))^*\epsilon$.
\end{proof}

\begin{cor} \label{cor:reps_are_sheaves_functoriality}
	For any of the dg-isomorphisms $\Phi_{\Gamma,T}^{\Gamma',T'}(\zeta,\gamma)$, we have
	$$\Psi_{G,\Gamma,T} \circ (\Phi_{\Gamma,T}^{\Gamma',T'}(\zeta,\gamma))^* = \Psi_{G,\Gamma',T'}.$$
\end{cor}

\begin{proof}
	We have proved that $\Psi_{G,\Gamma}$ commutes with orientation changes and Moves V-VII, as well as the morphisms $\wt{\Phi}_T^{T'}(\gamma)$. Every $\Phi_{\Gamma,T}^{\Gamma',T'}(\zeta,\gamma)$ descends from a composition of these simple changes, i.e. $\Phi_{\Gamma,T}^{\Gamma',T'}(\zeta,\gamma)$ fits into a diagram of the form
	$$\xymatrix{(\wt{\scr{B}}^+_{G},\wt{\partial}^{\scr{B}^+}_{G,\Gamma}) \ar[r]^-{\wt{\Phi}_{\Gamma,T}^{\Gamma',T'}(\zeta,\gamma)} \ar[d]_{\pi_T} & (\wt{\scr{B}}^+_{G},\wt{\partial}^{\scr{B}^+}_{G,\Gamma'}) \ar[d]^{\pi_{T'}} \\ (\scr{B}^+_{G,T},\partial^{\scr{B}^+}_{G,\Gamma,T}) \ar[r]_-{\Phi_{\Gamma,T}^{\Gamma',T'}(\zeta\gamma)} & (\scr{B}^+_{G,T'},\partial^{\scr{B}^+}_{G,\Gamma',T'})}$$
	with the property that $\Psi_{G,\Gamma} \circ (\wt{\Phi}_{\Gamma,T}^{\Gamma',T'}(\zeta,\gamma))^* = \Psi_{G,\Gamma'}$. In a similar manner as in the composition property of Lemma \ref{lem:induced_morphisms}, the desired equation is induced from this one.
\end{proof}

\subsection{The colors-to-representations map}

In this final subsection, we prove the following proposition, which is enough to prove our desired theorem.

\begin{prop} \label{prop:reps_are_sheaves_inverse}
	The maps
	$$\Psi_{G,\Gamma,T} \colon \scr{M}^{\mathrm{Rep}}_r(\scr{B}^+_{G,T},\partial^{\scr{B}^+}_{G,\Gamma,T};\F) \rightarrow \scr{M}^{\mathrm{Col}}_r(G;\F)$$
	are bijections for each choice of $\Gamma$ and $T$.
\end{prop}

\begin{proof}[Proof of Theorem \ref{thm:reps_are_sheaves} (assuming Proposition \ref{prop:reps_are_sheaves_inverse})]
	By Proposition \ref{prop:reps_are_sheaves_inverse}, the maps $\Psi_{G,\Gamma,T}$ are bijections. They were proved to satisfy the functoriality property in Corollary \ref{cor:reps_are_sheaves_functoriality}.
\end{proof}

In the rest of this subsection, we prove Proposition \ref{prop:reps_are_sheaves_inverse}, essentially by constructing inverse maps. Although we wish to find an inverse to $\Psi_{G,\Gamma,T}$, we will work with representations instead of moduli of representations, until we need to.

To begin, we wish to understand which elements of $(\mathrm{GL}_{2r}(\F))^{\mathrm{Tri}(G)}$ belong in the image of $M^{\Gamma}$. Instead of attacking this directly, note that we may encode more information in the construction of $M^{\Gamma}$ as follows. For each representation $\epsilon \in \mathrm{Rep}_r(\wt{\scr{B}}^+_{G},\wt{\partial}^{\scr{B}^+}_{G,\Gamma};\F)$, we may choose to encode $\epsilon_\Pi(H(\tau_f(v))) \in \mathrm{GL}_r(\F)$ for each thread $\tau_f(v)$, as well as $\epsilon_\Pi(\lambda_{e,\mathfrak{o}}) \in \mathrm{GL}_r(\F)$ for each generator $\lambda_{e,\mathfrak{o}} \in \Pi$ corresponding to some oriented edge $(e,\mathfrak{o}) \in \wt{E}$. In other words, we have a modified morphism
$$\wt{M}^{\Gamma} \colon  \mathrm{Rep}_r(\wt{\scr{B}}^+_{G},\wt{\partial}^{\scr{B}^+}_{G,\Gamma};\F) \rightarrow (\mathrm{GL}_{2r}(\F))^{\mathrm{Tri}(G)} \times (\mathrm{GL}_r(\F))^{\mathrm{Thr}(G)} \times (\mathrm{GL}_r(\F))^{\wt{E}}$$
where we have denoted by $\mathrm{Thr}(G)$ the collection of threads. We may write this map in terms of the three factors on the right as
$$\wt{M}^{\Gamma} = (M^{\Gamma},H^{\Gamma},\Lambda^{\Gamma}),$$
where $M^{\Gamma}$ is the map described previously, and
$$H^{\Gamma}(\epsilon,\tau) := \epsilon_\Pi(H(\tau)), \qquad \qquad \Lambda^{\Gamma}(\epsilon,\lambda_{e,\mathfrak{o}}) := \epsilon_\Pi(\lambda_{e,\mathfrak{o}})$$
for all $\tau \in \mathrm{Thr}(G)$ and $(e,\mathfrak{o}) \in \wt{E}$ respectively.

\begin{lem}
	The map $\wt{M}^{\Gamma}$ is injective.
\end{lem}
\begin{proof}
	We have that $\Lambda^{\Gamma}(\epsilon)$ is precisely encoding the images $\epsilon_\Pi(\lambda_{e,\mathfrak{o}})$. But since the various $\lambda_{e,\mathfrak{o}}$ generate $\Pi$, and $\epsilon$ is only non-trivial on elements of $\Pi$, we have that $\Lambda^{\Gamma}(\epsilon)$ alone is enough to recover $\epsilon$.
\end{proof}

An element in the image of $\wt{M}^{\Gamma}$ satisfies a number of properties which are clear by construction:
\begin{itemize}
	\item \textbf{Base-point fixing:} If $\nu_{\Gamma} \in \mathrm{Tri}(G)$ is the triangle containing the base-point of the garden $M^{\Gamma}(\nu_{\Gamma}) = I$ (by which we mean this is the case for every $\epsilon$).\\
	\item \textbf{Edge relation:} For every edge $e$, if $\mathfrak{o}$ is the orientation of $e$ occuring in $\Gamma$, and $\nu_1,\nu_2$ are the adjacent triangles to the left and right of the edge respectively, then
	$$M^{\Gamma}(\nu_2) = M^{\Gamma}(\nu_1)\begin{pmatrix}0 & \Lambda^{\Gamma}(\lambda_{e,\mathfrak{o}})^{-1} \\ -\Lambda^{\Gamma}(\lambda_{e,\overline{\mathfrak{o}}}) & 0\end{pmatrix}$$
	(again for every $\epsilon$).\\ 
	\item \textbf{Thread relation:} For every thread $\tau$, with its canonical orientation from center to vertex, if $\nu_1$ and $\nu_2$ are the triangles on the left and right of the thread, respectively, then
	$$M^{\Gamma}(\nu_2) = M^{\Gamma}(\nu_1)\begin{pmatrix}I & H^{\Gamma}(\tau) \\ 0 & I\end{pmatrix}$$
	(for every $\epsilon$).\\
	\item \textbf{Compatibility:} For every vertex $v \in V$ and every thread $\tau$ terminating at $v$, let $e_1$, $e_2$, and $e_3$ be the edges around $E$ appearing counterclockwise around $v$ starting at the face containing $\tau$, and let $\mathfrak{o}_1$, $\mathfrak{o}_2$, and $\mathfrak{o}_3$ be the outward orientations. Then
	$$H^{\Gamma}(\tau) = (-1)^{r_v}\Lambda^{\Gamma}(\lambda_{e_1,\mathfrak{o}_1})^{-1}\Lambda^{\Gamma}(\lambda_{e_2,\mathfrak{o}_2})\Lambda^{\Gamma}(\lambda_{e_3,\mathfrak{o}_3}^{-1}).$$
\end{itemize}

In fact, the compatibility condition is extraneous, as the following lemma demonstrates.

\begin{lem}\label{lem:compatibility_removal}
	Suppose we have a vertex $v \in V$. Suppose for each of the six triangles around the vertex, each incident thread, and each incident edge with orientation, we have associated an element of $\mathrm{GL}_{2r}(\F)$, $\mathrm{GL}_r(\F)$, and $\mathrm{GL}_r(\F)$, with the property that these associated elements satisfy the edge and thread relations around each incident thread and edge. Then the compatibility condition is automatically satisfied.
\end{lem}
\begin{proof}
	For simplicity, we assume all of the edges are oriented inwards in $\Gamma$; we may change the orientation of an edge by negating the two elements of $\mathrm{GL}_r(\F)$ associated to that edge. Then, just as in the proof of Lemma \ref{lem:SL_cubed}, we have that for $\nu$ a triangle adjacent to the vertex such that the counter-clockwise loop around $v$ traverses the threads and edges $\tau_1,e_1,\tau_2,e_2,\tau_3,e_3$, in that order, we find
	$$M(\nu) = M(\nu) \cdot Z,$$
	where, letting $H_j \in \mathrm{GL}_r(\F)$ be the component corresponding to $\tau_j$ and $\Lambda_j \in \mathrm{GL}_r(\F)$ be the component corresponding to $e_j$ inwardly oriented,
	$$Z := \begin{pmatrix}1 & H_1 \\ 0 & 1\end{pmatrix}\begin{pmatrix}0 & -\Lambda_1^{-1} \\ \overline{\Lambda}_1 & 0 \end{pmatrix}\begin{pmatrix}1 & H_2 \\ 0 & 1\end{pmatrix}\begin{pmatrix}0 & \Lambda_2^{-1} \\ -\overline{\Lambda}_2 & 0 \end{pmatrix}\begin{pmatrix}1 & H_3 \\ 0 & 1\end{pmatrix}\begin{pmatrix}0 & -\Lambda_3^{-1} \\ \overline{\Lambda}_3 & 0 \end{pmatrix}.$$
	Since $M^{\Gamma}(\nu) \in \mathrm{GL}_{2r}(\F)$ is invertible, it follows that $Z = I$. Multiplying out and looking at the bottom right $r \times r$ matrix yields
	$$I = \overline{\Lambda}_1H_2\overline{\Lambda}_2\Lambda_3^{-1},$$
	and hence
	$$H_2 = \overline{\Lambda}_1^{-1}\Lambda_3\overline{\Lambda}_2,$$
	which is one of the compatibility relations. The others follow by symmetry.
\end{proof}

Let
$$V_r(G,\Gamma;\F) \subset (\mathrm{GL}_{2r}(\F))^{\mathrm{Tri}(G)} \times (\mathrm{GL}_r(\F))^{\mathrm{Thr}(G)} \times (\mathrm{GL}_r(\F))^{\wt{E}}$$
be the subset satisfying these relations simultaneously. As the lemma indicates, since the compatibility relation is extraneous, we will henceforth ignore it.

\begin{prop}
	The image of $\wt{M}^{\Gamma}$ is $V_r(G,\Gamma;\F)$. Hence, there is an induced bijection
	$$\mathrm{Rep}_r(\wt{\scr{B}}^+_{G},\wt{\partial}^{\scr{B}^+}_{G,\Gamma};\F) \cong V_r(G,\Gamma;\F).$$
\end{prop}
\begin{proof}
	As just described, the image clearly lands in this subset. Conversely, if we have an element of this subset, then we may simply read off the representation $\epsilon$ from $\Lambda^{\Gamma}$. It suffices to check that $\sum H^{\Gamma}(\tau) = 0$ where the sum is over threads around a center, but this is clear by using the thread relation around this center, and noting that for $H_1,\ldots,H_m \in \mathrm{GL}_r(\F)$, we have
	$$\begin{pmatrix}I & H_1 \\ 0 & I\end{pmatrix} \cdots \begin{pmatrix}I & H_m \\ 0 & I\end{pmatrix} = \begin{pmatrix}I & \sum_{j=1}^{m}H_m \\ 0 & I\end{pmatrix}.$$
	The last sentence follows because we have already proved that $\wt{M}^{\Gamma}$ is injective.
\end{proof}

Now let us include the tree into the story. Notice that elements in the image of $\wt{M}^{\Gamma} \circ (\pi_T)^*$ have the following further property:
\begin{itemize}
	\item \textbf{Tree triviality:} If $e \in E \setminus T$, then $\Lambda^{\Gamma}(\lambda_{e,\mathfrak{o}}) = I$ for either orientation $\mathfrak{o}$ of $e$.
\end{itemize}
Let $V_r(G,\Gamma,T;\F) \subset V_r(G,\Gamma;\F)$ be the further subset of elements with the tree triviality property. We have
\begin{align*}
	V_r(G,\Gamma,T;\F) &\subset (\mathrm{GL}_{2r}(\F))^{\mathrm{Tri}(G)} \times (\mathrm{GL}_r(\F))^{\mathrm{Thr}(G)} \times (\mathrm{GL}_r(\F))^{\wt{T}} \\
	&\subset (\mathrm{GL}_{2r}(\F))^{\mathrm{Tri}(G)} \times (\mathrm{GL}_r(\F))^{\mathrm{Thr}(G)} \times (\mathrm{GL}_r(\F))^{\wt{E}}
\end{align*}
where the second inclusion is given by setting each coordinate of $(\mathrm{GL}_r(\F))^{\wt{E\setminus T}}$ to the identity matrix $I \in \mathrm{GL}_r(\F)$.

\begin{prop}
	The image of $\wt{M}^{\Gamma} \circ (\pi_T)^*$ is $V_r(G,\Gamma,T;\F)$. Hence, there is an induced bijection
	$$\mathrm{Rep}_r(\scr{B}^+_{G,T},\partial^{\scr{B}^+}_{G,\Gamma,T};\F) \cong V_r(G,\Gamma,T;\F).$$
\end{prop}
\begin{proof}
	The proof is essentially identical, noting that the only requirement of including the tree is that the edges not in the tree are sent to the identity. The last sentence follows because both $(\pi_T)^*$ and $\wt{M}^{\Gamma}$ are injective, and hence their composition is injective.
\end{proof}

Our description of $V_r(G,\Gamma,T;\F)$ is a seemingly extraneous construction, in that the coordinates $H$ and $\Lambda$ are simply disregarded in the map $M^{\Gamma}$; we are interested in $M^{\Gamma}$ instead of $\wt{M^{\Gamma}}$ after all. Our goal is to show that $V_r(G,\Gamma,T;\F)$ is itself naturally in bijective correspondence with the image of $M^{\Gamma} \circ (\pi_T)^*$. To do this, we will essentially remove the tree and thread coordinates one by one. Suppose:
\begin{itemize}
	\item $T'$ a subtree of $T$ (possibly including $T' = \emptyset$)
	\item $Y \subset \mathrm{Thr}(G)$ a subset of threads
\end{itemize}
Let us consider the set
$$V_r(G,\Gamma,T' \subset T,Y;\F) \subset (\mathrm{GL}_{2r}(\F))^{\mathrm{Tri}(G)} \times (\mathrm{GL}_r(\F))^{Y} \times (\mathrm{GL}_r(\F))^{\wt{T \setminus T'}}$$
satisfying essentially the same properties as before wherever they can, but where the edge and thread relations are replaced with the following weaker version:
\begin{itemize}
	\item \textbf{Edge* relation:} We require the old edge relation to hold for edges in $E \setminus T'$ (as we don't have access to edge data for elements in $T'$). For edges in $T'$ with adjacent triangles $\nu_1$ and $\nu_2$, we require the weaker condition that the span of the first $r$ columns of $M^{\Gamma}(\nu_1)$ are equal to the span of the last $r$ columns of $M^{\Gamma}(\nu_2)$ and vice versa. Equivalently, there exist elements $\Lambda_1,\Lambda_2 \in \mathrm{GL}_{r}(\F)$ such that $M^{\Gamma}(\nu_2) = \begin{pmatrix} 0 & \Lambda_1^{-1} \\ -\Lambda_2 & 0 \end{pmatrix}M^{\Gamma}(\nu_1)$.
	\item \textbf{Thread* relation:} We require the old thread relation for elements of $Y$. For threads $\tau$ in $\mathrm{Thr}(G) \setminus Y$, we require the weaker condition that there exists $H \in \mathrm{GL}_r(\F)$ such that for the triangles $\nu_1$ and $\nu_2$ neighboring $\tau$, we have $M^{\Gamma}(\nu_2) = \begin{pmatrix} I & H \\ 0 & I \end{pmatrix}M^{\Gamma}(\nu_1)$.
\end{itemize}
That is, instead of requiring that the data of the values of $H$ and $\Lambda$ are encoded, we require that locally, such values exist (we do not require uniqueness; the point is essentially that uniqueness is automatic). Notice that when $T' = \emptyset$ and $Y = \mathrm{Thr}(G)$, we have
$$V_r(G,\Gamma,T;\F) = V_r(G,\Gamma,\emptyset \subset T,\mathrm{Thr}(G);\F).$$
We prove now that we may increase $T'$ or decrease $Y$ at will and still end up with an isomorphic set under the corresponding projection.

\begin{lem} \label{lem:star_to_non-star}
	Suppose $T',T''$ are two subtrees of $T$ with $T' \subset T''$, and suppose $Y'' \subset Y' \subset \mathrm{Thr}(G)$. Consider the projection
	$$\xymatrix{(\mathrm{GL}_{2r}(\F))^{\mathrm{Tri}(G)} \times (\mathrm{GL}_r(\F))^{Y'} \times (\mathrm{GL}_r(\F))^{\wt{T \setminus T'}} \ar[d]^{\Pi} \\(\mathrm{GL}_{2r}(\F))^{\mathrm{Tri}(G)} \times (\mathrm{GL}_r(\F))^{Y''} \times (\mathrm{GL}_r(\F))^{\wt{T \setminus T''}}}.$$
	Then it induces a bijection of subsets 
	$$V_r(G,\Gamma,T' \subset T,Y';\F) \cong V_r(G,\Gamma,T'' \subset T,Y'';\F).$$
\end{lem}

\begin{proof}
	It is clear that an element of $V_r(G,\Gamma,T' \subset T,Y;\F)$ lands in $V_r(G,\Gamma,T'' \subset T, Y'';\F)$, since we simply require weaker relations to hold. We prove that this is a bijection by constructing an inverse. It suffices to fill in the data on $Y' \setminus Y''$ and $T'' \setminus T'$ so that the edge and thread relations hold. But notice that by the edge* and thread* relations, we have that there exist choices which make the edge and thread relation holds on the nose. It suffices to check they are unique.
	
	For the case of edges, notice that we have
	$$\begin{pmatrix}0 & \Lambda_1^{-1} \\ -\Lambda_2 & 0\end{pmatrix} = M(\nu_2) \cdot M(\nu_1)^{-1},$$
	so $\Lambda_1$ and $\Lambda_2$ are obviously determined by the values of $M(\nu_1)$ and $M(\nu_2)$. The same reasoning works for the threads.
\end{proof}

In particular, we have that
\begin{align*}
	\mathrm{Rep}_r(\scr{B}^+_{G,T},\partial^{\scr{B}^+}_{G,\Gamma,T};\F) &\cong V_r(G,\Gamma,T;\F) \\
		&\cong V_r(G,\Gamma,\emptyset \subset T,\mathrm{Thr}(G)) \\
		&\cong V_r(G,\Gamma,T\subset T,\emptyset;\F)
\end{align*}
But this last set, $V_r(G,\Gamma,T\subset T,\emptyset;\F)$, is just the image of $M^{\Gamma} \circ (\pi_T)^*$, simply because $V_r(G,\Gamma,T;\F)$ was the image of $\wt{M^{\Gamma}} \circ (\pi_T)^*$, and the bijection between the two is just the projection down to $(\mathrm{GL}_r(\F))^{\mathrm{Tri}(G)}$, the same projection which when composed with $\wt{M^{\Gamma}}$ yields $M^{\Gamma}$.

Finally, we are in a position to build an inverse map to $\Psi_{G,\Gamma,T}$. Instead of constructing the inverse with image in $\mathrm{Rep}_r(\scr{B}^+_{G,T},\partial^{\scr{B}^+}_{G,\Gamma,T};\F)$, we will instead construct it to have image in our convenient model $V_r(G,\Gamma,T \subset T, \emptyset; \F)$, where we only need to check the edge* and thread* relations.

By the functoriality of Corollary \ref{cor:reps_are_sheaves_functoriality}, we may assume that the triangle $\nu_{\Gamma}$ containing the base point is adjacent to the vertex $*_T$. The reason we do this is that $\wt{\Psi}_{G,\Gamma,T}(\epsilon) \in \mathrm{Col}_r(G,\F)$ always has that for the three faces around $*_T$ given as $f_{\Gamma}$ (containing $\nu_{\Gamma}$), $f'_{\Gamma}$, and $f''_{\Gamma}$, we always have that
$$\wt{\Psi}_{G,\Gamma,T}(\epsilon,f_{\Gamma}) = \mathrm{Span}\begin{pmatrix}I \\ 0\end{pmatrix}$$
$$\wt{\Psi}_{G,\Gamma,T}(\epsilon,f'_{\Gamma}) = \mathrm{Span}\begin{pmatrix}0 \\ I\end{pmatrix}$$
$$\wt{\Psi}_{G,\Gamma,T}(\epsilon,f''_{\Gamma}) = \mathrm{Span}\begin{pmatrix}I \\ I\end{pmatrix}.$$
This is simply because in our construction of $M^{\Gamma}$, we could have taken all of our paths to be located in a small disk around the vertex, so this is a local computation.

Suppose we are given an element $\chi \in \scr{M}_r^{\mathrm{Col}}(G;\F)$. Let us choose a representative $\wt{\chi} \in \mathrm{Col}_r(G;\F) \subset (\mathrm{Gr}(r,2r;\F))^F$, by which we mean $\chi$ is the equivalence class of $\wt{\chi}$ under the $\mathrm{PGL}_{2r}(\F)$ action, such that the following conditions are satisfied:
$$\wt{\chi}(f_{\Gamma}) = \mathrm{Span}\begin{pmatrix}I \\ 0\end{pmatrix}$$
$$\wt{\chi}(f'_{\Gamma}) = \mathrm{Span}\begin{pmatrix}0 \\ I\end{pmatrix}$$
$$\wt{\chi}(f''_{\Gamma})  = \mathrm{Span}\begin{pmatrix}I \\ I\end{pmatrix}.$$
Notice that such $\wt{\chi}$ exists, because the three faces involved are pairwise adjacent to each other, and any three pairs of transverse Grassmannian $r$-planes in $\F^{2r}$ may be arranged in this way by the action of $\mathrm{GL}_{2r}(\F)$. Note also that the elements of $\mathrm{GL}_{2r}(\F)$ which fix these three planes are of the form $\begin{pmatrix}g & 0 \\ 0 & g\end{pmatrix}$ for $g \in \mathrm{GL}_r(\F)$. We see that our choice of $\wt{\chi}$ is therefore well-defined up to action by $\mathrm{GL}_r(\F)$ on $\mathrm{Col}_r(\F)$; we recall that $\wt{\Psi}_{G,\Gamma,T}$ intertwined the conjugation action on representations with this action on colorings, which is a good sign that we have picked the right gauge fixing condition.

Suppose now we have a curve $\gamma \colon [0,1] \rightarrow S^2$ satisfying the properties:
\begin{itemize}
	\item $\gamma(0)$ is the base point of $\Gamma$
	\item $\gamma(1)$ is located in the interior of a triangle
	\item $\gamma$ intersects none of the edges of the tree $T$, nor the vertex $*_T$, nor any of the centers of the garden
	\item $\gamma$ is otherwise transverse to the edges (of $E \setminus T$) and threads that it crosses, say at a finite set of times $0 < t_1 < \cdots < t_n < 1$
\end{itemize}
Associated to $\gamma$ and $\wt{\chi}$ is a locally constant function
$$\Theta(\gamma,\wt{\chi}) \colon [0,1] \setminus \{t_1,\ldots,t_n\} \rightarrow \mathrm{GL}_{2r}(\F)$$
determined by first setting $\Theta(\gamma,\wt{\chi})(0) = I$ and providing the following rules for how this function changes as we pass through $t_i$ as follows
\begin{itemize}
	\item $\Theta(\gamma,\wt{\chi})(0) = I$ is the identity matrix
	\item If $\gamma$ crosses through an edge at time $t_i$, then for $0<\epsilon < \max(t_i-t_{i-1},t_{i+1}-t_i)$, we have
	$$\Theta(\gamma,\wt{\chi})(t_i+\epsilon) = \pm  \Theta(\gamma,\wt{\chi})(t_i-\epsilon) \cdot \begin{pmatrix}0 & I \\ -I & 0 \end{pmatrix}$$
	where the sign is determined by whether $\gamma$ crosses through the edge positively or negatively.
	\item If $\gamma$ crosses through a thread at time $t_i$ from $\nu_1$ to $\nu_2$, then similarly
	$$\Theta(\gamma,\wt{\chi})(t_i+\epsilon) = \pm  \Theta(\gamma,\wt{\chi})(t_i-\epsilon) \cdot \begin{pmatrix}I & H \\ 0 & I \end{pmatrix},$$
	where $H$ is the unique element of $\mathrm{GL}_r(\F)$ such that the resulting product $\Theta(\gamma,\wt{\chi})(t_i-\epsilon) \cdot \begin{pmatrix}I & H \\ 0 & I \end{pmatrix}$ has that the span of the last $r$ columns is the element of $\wt{\chi}$ attached to the face adjacent to $\nu_2$.
\end{itemize}
Notice that it is easy to check inductively that for all $t \neq t_i$, we have that if $\gamma(t) \in \nu \subset f$, with $f'$ the face opposite $\nu$ via an edge, then the span of the first $r$ columns of $\Theta(\gamma,\wt{\chi})(t)$ is given by $\wt{\chi}(f)$, whereas the span of the last $r$ columns is $\wt{\chi}(f')$. Indeed, this is certainly true for $t=0$, and remains true when we pass through an edge since we essentially swap the first and last $r$ columns up to a sign, and also remains true when we pass through a thread by the very definition of how $\Theta^{\Gamma}(\gamma,\wt{\chi})$ changes.

\begin{lem}
	Suppose $\nu_{\gamma}$ is the triangle containing $\gamma(1)$. We wish to claim that if $\nu_{\gamma} = \nu_{\gamma'}$, then
	$$\Theta(\gamma,\wt{\chi})(1) = \Theta(\gamma',\wt{\chi})(1).$$
\end{lem}

\begin{proof}
	Notice that $S^2 \setminus T$ is topologically a disk, so $\gamma$ and $\gamma'$ are homotopic, in fact via a homotopy so that the ending points remain in $\nu_{\gamma} = \nu_{\gamma'}$. It follows that we may change $\gamma$ into $\gamma'$ via a small handful of moves, each of which we check does not affect the resulting value of $\Theta(\cdot,\wt{\chi})(1)$:
	\begin{itemize}
		\item We may isotope $\gamma$ in such a way that it still intersects the same threads and edges in the same order. The resulting value clearly does not change.
		\item We may push $\gamma$ through a tangency with an edge or thread (as in Moves I and II, but where the tine is replaced by $\gamma$). When we go back and forth through an edge, we see that we end up at the same place, because we multiply by
		$$\begin{pmatrix} 0 & \pm I \\ \mp I & 0\end{pmatrix} \cdot \begin{pmatrix} 0 & \mp I \\ \pm I & 0\end{pmatrix} = I_{2r},$$
		where the signs exactly cancel because we traverse the edge with opposite orientations. On the other hand, suppose we go back and forth through a thread, from $\nu_1$ to $\nu_2$ back to $\nu_1$. Let $\Theta$ be the matrix we have at the beginning, when we are in $\nu_1$. Then we have that the matrix goes from $\Theta$ to $\Theta \cdot \begin{pmatrix}I & H \\ 0 & I\end{pmatrix}$ to $\Theta \cdot \begin{pmatrix}I & H \\ 0 & I\end{pmatrix} \cdot \begin{pmatrix} I & H' \\ 0 & I \end{pmatrix}$ for some uniquely determined $H,H' \in \mathrm{GL}_r(\F)$. But these must satisfy that the end result has that the last $r$ columns spanning the face opposite $\nu_1$ via an edge. But inductively, $\Theta$ already has this property, so by the uniqueness, we have that $H'=-H$. Hence, after going back and forth, we arrive back at $\Theta$.
		\item We may push $\gamma$ through a center. As we go around the center, starting at some triangle $\nu$ and ending up back at $\nu$, because our curve goes through threads, if we started at $\Theta$, we accumulate a product $$\Theta \cdot \begin{pmatrix}I & H_1 \\ 0 & I\end{pmatrix} \cdots \begin{pmatrix} I & H_k \\ 0 & I \end{pmatrix}$$ where $k$ is the number of threads around the center, again with the property that the last $r$ columns must span the same as the last $r$ columns of $\Theta$. We see that $H_k = -H_1 - \cdots H_{k-1}$ seems to work, and by the uniqueness, it suffices to check that this is indeed an element of $\mathrm{GL}_r(\F)$. But this is true because the last $r$ columns of $$\Theta \cdot \begin{pmatrix}I & H_1 \\ 0 & I\end{pmatrix} \cdots \begin{pmatrix} I & H_{k-1} \\ 0 & I \end{pmatrix}$$ span a subspace transverse to the span of the last $r$ columns of $\Theta$, by the fact that the last $r$ columns span the element of $\wt{\chi}$ on the opposite face.
		\item We may push $\gamma$ through the vertex $*_T$ (as in Move III). But because everything is uniquely determined as we go around $*_T$, everything matches a standard local model up to multiplication on the left by the original $\Theta$.
	\end{itemize}
\end{proof}

Therefore, given $\wt{\chi}$, we obtain an element in $\Theta(\wt{\chi}) \in (\mathrm{GL}_{2r}(\F))^{\mathrm{Tri}(G)}$, taking the value $\Theta(\gamma,\wt{\chi})(1)$ on $\nu_{\gamma}$ for any path $\gamma$ from the base-point to $\nu_{\gamma}$. In fact, by construction, we have that $\Theta(\wt{\chi})$ satisfies the edge* relations with respect to the full tree $T$ and the thread* relations with respect to the empty set of threads. That is, we have
$$\Theta(\wt{\chi}) \in V_r(G,\Gamma,T\subset T,\emptyset;\F).$$
Although $\wt{\chi}$ was only determined from $\chi$ up to the action of $\mathrm{GL}_r(\F)$, this action intertwined with the conjugation action on representations, or in turn the conjugation action on $(\mathrm{GL}_{2r}(\F))^{\mathrm{Tri}(G)}$ by the corresponding element $\begin{pmatrix} g & 0 \\ 0 & g \end{pmatrix}$. Hence, the map $\wt{\chi} \mapsto \Theta(\wt{\chi})$ descends to a map
$$\Theta_{G,\Gamma,T} \colon \scr{M}_r^{\mathrm{Col}}(G;\F) \rightarrow V_r(G,\Gamma,T \subset T,\emptyset;\F)/\mathrm{GL}_r(\F) \cong \scr{M}_r^{\mathrm{Rep}}(\scr{B}^+_{G,T},\partial^{\scr{B}^+}_{G,\Gamma,T};\F).$$
If we follow $\Theta(\wt{\chi})$ back along the construction of $\wt{\Phi}_{G,\Gamma,T}$, i.e. taking
$$\mathrm{Span}\left(\Theta(\wt{\chi}) \cdot \begin{pmatrix}I \\ 0 \end{pmatrix}\right),$$
by construction, we simply arrive back at $\wt{\chi}$, which implies that
$$\Psi_{G,\Gamma,T} \circ \Theta_{G,\Gamma,T} = \id.$$
Therefore $\Psi_{G,\Gamma,T}$ is surjective. In fact, since $\Theta(\wt{\chi})$ was essentially determined \emph{uniquely} by requiring
\begin{itemize}
	\item the edge* relations hold,
	\item the thread* relations hold, and
	\item $\wt{\Psi}_{G,\Gamma,T}(\Theta(\wt{\chi}))$ matches $\wt{\chi}$ on the three faces around $*_T$
\end{itemize}
we have that actually $\Theta(\wt{\chi})$ is the \emph{unique} element of $V_r(G,\Gamma,T\subset T,\emptyset;\F)$ such that $\wt{\Psi}_{G,\Gamma,T}(\Theta(\wt{\chi})) = \wt{\chi}$. But the specific value of $\wt{\chi}$ only depended up to the action of $\mathrm{GL}_r(\F)$ (given $\chi$), and hence $\Theta(\wt{\chi})$ only varies by the action of $\mathrm{GL}_r(\F)$ as we change our choice of $\wt{\chi}$. In particular, if we have $\Psi_{G,\Gamma,T}(\epsilon) = \chi$, then that implies that
$$\wt{\Psi}_{G,\Gamma,T}(\epsilon) = \wt{
\chi}$$
for one of the lifts of $\chi$, and in turn by this uniqueness, we have
$$\epsilon = g \cdot \Theta(\wt{\chi})$$
for some $g \in \mathrm{GL}_r(\F)$. Hence, we have that $\Psi_{G,\Gamma,T}$ is injective. Summarizing:

\begin{proof}[Proof of Proposition \ref{prop:reps_are_sheaves_inverse}]
	We have proved that $\Psi_{G,\Gamma,T}$ is both injective and surjective, which is what we needed to show.
\end{proof}

\begin{rmk}
	We have done slightly better than to prove Proposition \ref{prop:reps_are_sheaves_inverse}; we have actually constructed an explicit inverse in the form of $\Theta_{G,\Gamma,T}$.
\end{rmk}

\appendix
\section{Computational example} \label{appx:example}

We provide a sample computation of the enlarged dg-algebra $(\wt{\scr{B}}_G^+,\wt{\partial}^{\scr{B}^+}_{G,\Gamma})$ for the specific pair $(G,\Gamma)$ of Figure \ref{fig:Garden_2} (with $g=1$). We note that with commutative coefficients, this computation was performed by Casals and Murphy \cite[Section 5]{CM_DGA}, so the reader is invited to compare our computation to theirs.

\begin{figure}[h]
	\centering
	\includegraphics[width=0.6\textwidth]{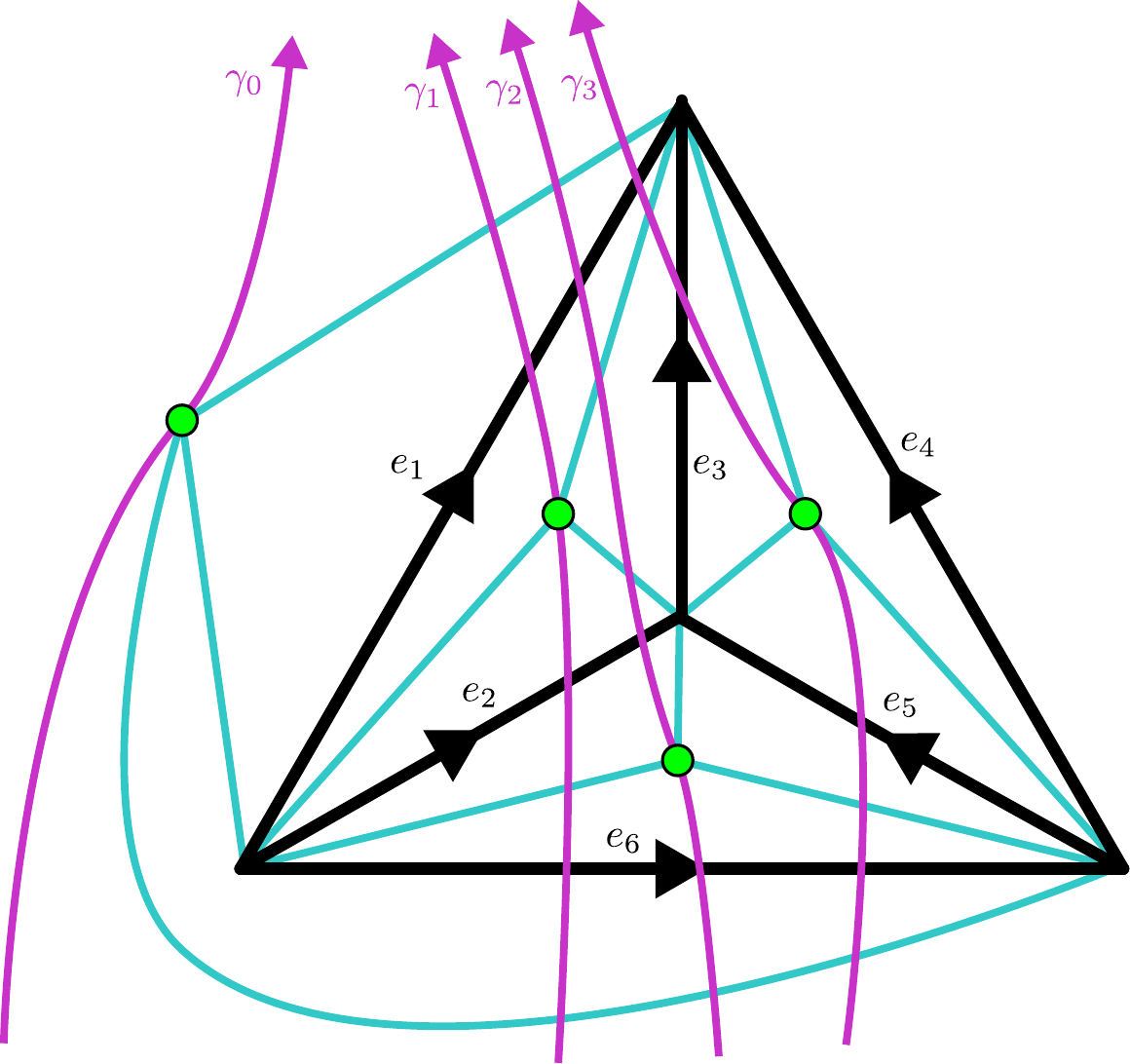}
	\caption{The graph and garden for which the enlarged dg-algebra is computed in Appendix \ref{appx:example}.}
	\label{fig:Garden_2}
\end{figure}

In the figure, the tines are labeled $\gamma_j$ for $j=0,1,2,3$. We will write $f_j$ for the generator corresponding to the center through which $\gamma_j$ passes. We have also labeled the edges and oriented all edges $e_k$ for $k=1,\ldots,6$. We will use the generators $A_k$ and $B_k$ as in Figure \ref{fig:Generator_Convention}.

The differentials of the faces are as follows:
$$\wt{\partial}f_0 = B_{1}^{-1}A_{3}B_{4}^{-1} - A_{6}^{-1}B_{2}A_{1}^{-1}+A_{4}^{-1}B_{5}B_{6}^{-1}$$
$$\wt{\partial}f_1 = B_{3}^{-1}A_{4}B_{1}^{-1}-A_{1}^{-1}B_{6}A_{2}^{-1}-B_{2}^{-1}A_{5}A_{3}^{-1}$$
$$\wt{\partial}f_2 = -B_{5}^{-1}B_{3}B_{2}^{-1}-A_{2}^{-1}B_{1}A_{6}^{-1}+B_{6}^{-1}B_{4}A_{5}^{-1}$$
$$\wt{\partial}f_3 = B_{4}^{-1}A_{1}B_{3}^{-1}- A_{3}^{-1}A_{2}B_{5}^{-1}+A_{5}^{-1}A_{6}A_{4}^{-1}$$
Meanwhile, we may write
$$\begin{pmatrix}\wt{\partial} z & \wt{\partial} w \\ \wt{\partial} y & \wt{\partial} x\end{pmatrix} = \sum_{j=0}^{3}W(\gamma_j)$$
where we compute each $W(\gamma_j)$ by breaking each tine $\gamma_j$ into elementary pieces in which it either crosses and edge, thread, or face:
\begin{eqnarray*}
	W(\gamma_0) &=& \begin{pmatrix}0 & f_0 \\ 0 & 0\end{pmatrix} \\
	W(\gamma_1) &=& \begin{pmatrix}1 & A_{4}^{-1}B_{5}B_{6}^{-1} \\ 0 & 1\end{pmatrix}\begin{pmatrix}0 & A_6^{-1} \\ -B_6 & 0 \end{pmatrix}\begin{pmatrix}1 & A_{2}^{-1}B_{1}A_{6}^{-1} \\ 0 & 1\end{pmatrix}\\
		&& \qquad\cdot\begin{pmatrix}0 & A_2^{-1} \\ -B_2 & 0 \end{pmatrix}\begin{pmatrix}0 & f_1 \\ 0 & 0\end{pmatrix}\begin{pmatrix}0 & A_1^{-1} \\ -B_1 & 0 \end{pmatrix}\begin{pmatrix}1 & B_{3}^{-1}A_{4}B_{1}^{-1} \\ 0 & 1\end{pmatrix} \\
	W(\gamma_2) &=& \begin{pmatrix}1 &  A_{4}^{-1}B_{5}B_{6}^{-1} \\ 0 & 1\end{pmatrix}\begin{pmatrix}0 & A_6^{-1} \\ -B_6 & 0 \end{pmatrix}\begin{pmatrix}0 & f_2 \\ 0 & 0\end{pmatrix}\begin{pmatrix}0 & A_2^{-1} \\ -B_2 & 0 \end{pmatrix}\\
		&& \qquad \cdot \begin{pmatrix}1 & -B_{2}^{-1}A_{5}A_{3}^{-1} \\ 0 & 1\end{pmatrix}\begin{pmatrix}1 & B_{3}^{-1}A_{4}B_{1}^{-1} \\ 0 & 1\end{pmatrix}\begin{pmatrix}0 & A_1^{-1} \\ -B_1 & 0 \end{pmatrix}\begin{pmatrix}1 & B_{1}^{-1}A_{3}B_{4}^{-1} \\ 0 & 1\end{pmatrix} \\
	W(\gamma_3) &=& \begin{pmatrix}1 & A_{4}^{-1}B_{5}B_{6}^{-1} \\ 0 & 1\end{pmatrix}\begin{pmatrix}0 & A_6^{-1} \\ -B_6 & 0 \end{pmatrix}\begin{pmatrix}1 & B_{6}^{-1}B_{4}A_{5}^{-1} \\ 0 & 1\end{pmatrix}\begin{pmatrix}0 & -B_5^{-1}\\ A_5 & 0\end{pmatrix}\\
		&& \qquad \cdot \begin{pmatrix}0 & f_3 \\ 0 & 0\end{pmatrix}\begin{pmatrix}0 & A_3^{-1} \\ -B_3 & 0\end{pmatrix}\begin{pmatrix}1 & B_{3}^{-1}A_{4}B_{1}^{-1} \\ 0 & 1\end{pmatrix}\begin{pmatrix}0 & A_1^{-1} \\ -B_1 & 0 \end{pmatrix}\begin{pmatrix}1 & B_{1}^{-1}A_{3}B_{4}^{-1} \\ 0 & 1\end{pmatrix}
\end{eqnarray*}
One may multiply these out and use the relations between geometric elements to simplify. For example, the expression for $W(\gamma_3)$ simplifies to
$$W(\gamma_3) = \begin{pmatrix}0 & 0 \\ -B_4f_3A_4 & 0 \end{pmatrix}.$$
This was expected by Lemma \ref{lem:invariance_of_word}, since we may homotope $\gamma_3$ relative to its endpoints at infinity and without passing through other faces so that it crosses $e_4$, then crosses $f_3$, then crosses $e_4$ again. We see that $f_3$ only appears in the differential of the generator $y$, and not in the differential of the generators $w$, $x$, and $z$. As a sanity check, this matches with the Casals--Murphy computation in the commutative coefficient setting: only the differential of $y$ has an $f_3$ term.

\bibliography{Bib}{}

\begin{thebibliography}{10}

\bibitem{AH2}
K.~Appel and W.~Haken.
\newblock Supplement to: ``{E}very planar map is four colorable. {I}.
  {D}ischarging''\ ({I}llinois {J}. {M}ath. {\bf 21} (1977), no. 3, 429--490)
  by {A}ppel and {H}aken; ``{II}. {R}educibility'' (ibid. {\bf 21} (1977), no.
  3, 491--567) by {A}ppel, {H}aken and {J}. {K}och.
\newblock {\em Illinois J. Math.}, 21(3):1--251. (microfiche supplement), 1977.

\bibitem{AH1}
Kenneth Appel and Wolfgang Haken.
\newblock The solution of the four-color-map problem.
\newblock {\em Sci. Amer.}, 237(4):108--121, 152, 1977.

\bibitem{AE}
Johan Asplund and Tobias Ekholm.
\newblock Chekanov-{E}liashberg dg-algebras for singular {L}egendrians.
\newblock {\em J. Symplectic Geom.}, 20(3):509--559, 2022.

\bibitem{CM_DGA}
Roger Casals and Emmy Murphy.
\newblock Differential algebra of cubic planar graphs (with an appendix by {K}.
  {S}ackel).
\newblock {\em Adv. Math.}, 338:401--446, 2018.

\bibitem{CZ}
Roger Casals and Eric Zaslow.
\newblock Legendrian weaves: {$N$}-graph calculus, flag moduli and
  applications.
\newblock {\em Geom. Topol.}, 26(8):3589--3745, 2022.

\bibitem{CDGG}
Baptiste Chantraine, Georgios Dimitroglou~Rizell, Paolo Ghiggini, and Roman
  Golovko.
\newblock Noncommutative augmentation categories.
\newblock In {\em Proceedings of the {G}\"{o}kova {G}eometry-{T}opology
  {C}onference 2015}, pages 116--150. G\"{o}kova Geometry/Topology Conference
  (GGT), G\"{o}kova, 2016.

\bibitem{CNS}
Baptiste Chantraine, Lenhard Ng, and Steven Sivek.
\newblock Representations, sheaves and {L}egendrian {$(2,m)$} torus links.
\newblock {\em J. Lond. Math. Soc. (2)}, 100(1):41--82, 2019.

\bibitem{Chekanov}
Yuri Chekanov.
\newblock Differential algebra of {L}egendrian links.
\newblock {\em Invent. Math.}, 150(3):441--483, 2002.

\bibitem{Ekholm}
Tobias Ekholm.
\newblock Morse flow trees and {L}egendrian contact homology in 1-jet spaces.
\newblock {\em Geom. Topol.}, 11:1083--1224, 2007.

\bibitem{EES}
Tobias Ekholm, John Etnyre, and Michael Sullivan.
\newblock The contact homology of {L}egendrian submanifolds in {${\Bbb
  R}^{2n+1}$}.
\newblock {\em J. Differential Geom.}, 71(2):177--305, 2005.

\bibitem{EES_noniso}
Tobias Ekholm, John Etnyre, and Michael Sullivan.
\newblock Non-isotopic {L}egendrian submanifolds in {$\Bbb R^{2n+1}$}.
\newblock {\em J. Differential Geom.}, 71(1):85--128, 2005.

\bibitem{EES_PxR}
Tobias Ekholm, John Etnyre, and Michael Sullivan.
\newblock Legendrian contact homology in {$P\times\Bbb R$}.
\newblock {\em Trans. Amer. Math. Soc.}, 359(7):3301--3335, 2007.

\bibitem{EL}
Tobias Ekholm and Yank\i Lekili.
\newblock Duality between {L}agrangian and {L}egendrian invariants.
\newblock {\em Geom. Topol.}, 27(6):2049--2179, 2023.

\bibitem{EP}
Y.~Eliashberg and L.~Polterovich.
\newblock Local {L}agrangian {$2$}-knots are trivial.
\newblock {\em Ann. of Math. (2)}, 144(1):61--76, 1996.

\bibitem{Eliashberg_invariants}
Yakov Eliashberg.
\newblock Invariants in contact topology.
\newblock In {\em Proceedings of the {I}nternational {C}ongress of
  {M}athematicians, {V}ol. {II} ({B}erlin, 1998)}, number Extra Vol. II, pages
  327--338, 1998.

\bibitem{ENS}
John~B. Etnyre, Lenhard~L. Ng, and Joshua~M. Sabloff.
\newblock Invariants of {L}egendrian knots and coherent orientations.
\newblock {\em J. Symplectic Geom.}, 1(2):321--367, 2002.

\bibitem{FJ}
Herbert Federer and Bjarni J\'{o}nsson.
\newblock Some properties of free groups.
\newblock {\em Trans. Amer. Math. Soc.}, 68:1--27, 1950.

\bibitem{GPS_microlocal}
Sheel {Ganatra}, John {Pardon}, and Vivek {Shende}.
\newblock {Microlocal Morse theory of wrapped Fukaya categories}.
\newblock {\em arXiv e-prints}, page arXiv:1809.08807, September 2018.

\bibitem{Gao_knot}
Honghao Gao.
\newblock Simple sheaves for knot conormals.
\newblock {\em J. Symplectic Geom.}, 18(4):1027--1070, 2020.

\bibitem{Gao_link}
Honghao {Gao}.
\newblock {Augmentations and sheaves for links}.
\newblock {\em arXiv e-prints}, page arXiv:2109.01285, September 2021.

\bibitem{Kalman}
Tam\'{a}s K\'{a}lm\'{a}n.
\newblock Contact homology and one parameter families of {L}egendrian knots.
\newblock {\em Geom. Topol.}, 9:2013--2078, 2005.

\bibitem{Nad}
David Nadler.
\newblock Microlocal branes are constructible sheaves.
\newblock {\em Selecta Math. (N.S.)}, 15(4):563--619, 2009.

\bibitem{NZ}
David Nadler and Eric Zaslow.
\newblock Constructible sheaves and the {F}ukaya category.
\newblock {\em J. Amer. Math. Soc.}, 22(1):233--286, 2009.

\bibitem{NRSSZ}
Lenhard Ng, Dan Rutherford, Vivek Shende, Steven Sivek, and Eric Zaslow.
\newblock Augmentations are sheaves.
\newblock {\em Geom. Topol.}, 24(5):2149--2286, 2020.

\bibitem{Nielsen}
J.~Nielsen.
\newblock Om regning med ikke-kommutative faktorer og dens anvendelse i
  gruppeteorien.
\newblock {\em Matematisk Tidsskrift. B}, pages 77--94, 1921.

\bibitem{Nielsen_trans}
J.~Nielsen.
\newblock On calculation with noncommutative factors and its application to
  group theory.
\newblock {\em Math. Sci.}, 6(2):73--85, 1981.

\bibitem{RS}
Dan Rutherford and Michael Sullivan.
\newblock Sheaves via augmentations of {L}egendrian surfaces.
\newblock {\em J. Homotopy Relat. Struct.}, 16(4):703--752, 2021.

\bibitem{STWZ}
Vivek Shende, David Treumann, Harold Williams, and Eric Zaslow.
\newblock Cluster varieties from {L}egendrian knots.
\newblock {\em Duke Math. J.}, 168(15):2801--2871, 2019.

\bibitem{TZ}
David Treumann and Eric Zaslow.
\newblock Cubic planar graphs and {L}egendrian surface theory.
\newblock {\em Adv. Theor. Math. Phys.}, 22(5):1289--1345, 2018.

\end{thebibliography}
\bibliographystyle{plain}

\end{document}